%% file: uniformization.tex
\documentclass{amsart}

\usepackage[alphabetic,initials,nobysame]{amsrefs}
\usepackage{amssymb,mathtools,mathrsfs,comment,todonotes,color}
\usepackage{hyperref} 
\usepackage[shortlabels]{enumitem}
\usepackage{caption}
\usepackage[percent]{overpic}
\usepackage{pgfplots}
\pgfplotsset{compat=1.17}
\usetikzlibrary{arrows,patterns.meta}

\DeclareFontFamily{OT1}{pzc}{}
\DeclareFontShape{OT1}{pzc}{m}{it}{<-> s * [1.10] pzcmi7t}{}
\DeclareMathAlphabet{\mathpzc}{OT1}{pzc}{m}{it}

\BibSpec{collection.article}{%
	+{}  {\PrintAuthors}				{author}
	+{,} { \textit}                     {title}
	+{.} { }                            {part}
	+{:} { \textit}                     {subtitle}
	+{,} { \PrintContributions}         {contribution}
	+{,} { \PrintConference}            {conference}
	+{}  {\PrintBook}                   {book}
	+{,} { }                            {booktitle}
	+{,} { }							{series}
	+{,} { }							{publisher}
	+{,} { }							{address}
	+{,} { \PrintDateB}                 {date}
	+{,} { pp.~}                        {pages}
	+{,} { }                            {status}
	+{,} { \PrintDOI}                   {doi}
	+{,} { available at \eprint}        {eprint}
	+{}  { \parenthesize}               {language}
	+{}  { \PrintTranslation}           {translation}
	+{;} { \PrintReprint}               {reprint}
	+{.} { }                            {note}
	+{.} {}                             {transition}
	+{}  {\SentenceSpace \PrintReviews} {review}
}

\theoremstyle{plain}
\newtheorem{theorem}{Theorem}
\newtheorem{lemma}[theorem]{Lemma}

\newtheorem{prop}[theorem]{Proposition}
\newtheorem{corollary}[theorem]{Corollary}

\theoremstyle{definition}
\newtheorem{definition}[theorem]{Definition}
\newtheorem{claim}[theorem]{Claim}

\theoremstyle{remark}
\newtheorem{remark}[theorem]{Remark}
\newtheorem{question}[theorem]{Question}

\numberwithin{equation}{section}
\numberwithin{theorem}{section}
\numberwithin{conjecture}{section}

\newcommand{\x}{\scalebox{1.2}{$\chi$} } 
\newcommand{\br}{\overline}
\newcommand{\R}{\mathbb R}

\newcommand{\C}{\mathbb C}

\newcommand{\N}{\mathbb N}

\newcommand{\p}{\mathbb P}

\DeclareMathOperator{\dist}{{\mathrm{dist}}}
\DeclareMathOperator{\diam}{{\mathrm{diam}}}

\DeclareMathOperator{\inter}{{\mathrm{int}}}

\DeclareMathOperator{\md}{\mathrm{Mod}}

\DeclareMathOperator{\CNED}{\mathit{CNED}}

\DeclareMathOperator{\clu}{\mathrm{Clu}}

\DeclareMathOperator*{\essinf}{ess\,inf}

\begin{document}
\title[Conformal uniformization by disk packings]{Conformal uniformization of planar packings by disk packings}
\author{Dimitrios Ntalampekos}
\address{Mathematics Department, Stony Brook University, Stony Brook, NY 11794, USA.}
\thanks{The author is partially supported by NSF Grant DMS-2000096.}
\email{dimitrios.ntalampekos@stonybrook.edu}
\date{\today}

\begin{abstract}
A Sierpi\'nski packing in the $2$-sphere is a countable collection of disjoint, non-separating continua with diameters shrinking to zero. We show that any Sierpi\'nski packing by continua whose diameters are square-summable can be uniformized by a disk packing with a packing-conformal map, a notion that generalizes conformality in open sets.  Being special cases of Sierpi\'nski packings, Sierpi\'nski carpets and some domains and can be uniformized by disk packings as well. 
As a corollary of the main result, the conformal loop ensemble (CLE) carpets can be uniformized conformally by disk packings, answering a question of Rohde--Werness.
\end{abstract}

\maketitle

\setcounter{tocdepth}{1}
\tableofcontents

\section{Introduction}

One of the most intriguing open problems in complex analysis is \textit{Koebe's conjecture} \cite{Koebe:Kreisnormierungsproblem}, predicting that every domain in the Riemann sphere is conformally equivalent to a \textit{circle domain}, i.e., a domain whose complementary components are geometric disks or points. This conjecture was established for finitely connected domains by Koebe himself \cite{Koebe:FiniteUniformization} and it took over 70 years until it was established for countably connected domains by He--Schramm \cite{HeSchramm:Uniformization}. This result was proved with a different method by Schramm \cite{Schramm:transboundary} in a seminal work, where the notion of transboundary modulus is introduced. More recently, Rajala \cite{Rajala:koebe} gave another proof of the result, providing a new perspective. Remarkably, in \cite{Schramm:transboundary}, Schramm establishes Koebe's conjecture for all cofat domains, i.e., domains whose complementary components satisfy a uniform geometric condition that we discuss below in Section \ref{section:fat}, independently of connectivity. The general case of Koebe's conjecture seems to be far out of reach.  Koebe's conjecture and uniformization problems for domains in  metric surfaces other than the Riemann sphere has been studied in \cites{MerenkovWildrick:uniformization, RajalaRasimus:koebe, Rehmert:Thesis}.

A topic very closely related to Koebe's conjecture is the uniformization of Sier\-pi\'nski carpets. A \textit{Sierpi\'nski carpet} is a continuum in the sphere that has empty interior and is obtained by removing from the sphere countably many open Jordan regions, called \textit{peripheral disks}, with disjoint closures and diameters shrinking to zero. A fundamental result of Whyburn \cite{Whyburn:theorem} states that all Sierpi\'nski carpets are homeomorphic to each other. Bonk \cite{Bonk:uniformization} proved that if the peripheral disks of a Sierpi\'nski carpet are uniformly relatively separated  uniform quasidisks, then the carpet can be mapped with a quasisymmetric map to a \textit{round carpet}, i.e., a carpet whose peripheral disks are geometric disks. Later, in \cite{Ntalampekos:CarpetsThesis} the author developed a potential theory on Sierpi\'nski carpets of area zero and proved that if the peripheral disks of such a carpet are uniformly fat and uniformly quasiround, then the carpet can be mapped in a natural way to a \textit{square carpet}, defined in the obvious manner, with a map that is \textit{carpet-conformal} in the sense that it preserves a type of modulus. We note that the geometric assumptions in \cite{Ntalampekos:CarpetsThesis} are weaker than in  \cite{Bonk:uniformization}. However, if one strengthens the assumptions to uniformly relatively separated uniform quasidisks, then the carpet-conformal map of \cite{Ntalampekos:CarpetsThesis} is upgraded to a  quasisymmetry. Both mentioned works of Bonk and the author depend crucially on the notion of transboundary modulus of Schramm \cite{Schramm:transboundary}.

In this work we push the results of \cites{Bonk:uniformization,Ntalampekos:CarpetsThesis} to their limit and we \textit{remove entirely the geometric assumptions} at the cost of weakening the topological properties and the regularity of the uniformizing conformal map. Instead, we only impose the square-summability of the diameters of the peripheral disks. Before stating the results we give the required definitions.

Let $\{p_i\}_{i\in \N}$ be a collection of pairwise disjoint and non-separating continua in the Riemann sphere $\widehat{\C}$ such that $\diam(p_i)\to 0$ as $i\to\infty$. The collection $\{p_i\}_{i\in \N}$ is called a \textit{Sierpi\'nski packing} and the set  $X=\widehat{\C}\setminus \bigcup_{i\in \N} p_i$ is its \textit{residual set}. When it does not lead to a confusion, we make no distinction between the terms packing and residual set. The continua $p_i$, $i\in \N$, are  called the \textit{peripheral continua} of $X$. Note that if the peripheral continua of $X$ are closed Jordan regions, then $\br X$ is a Sierpi\'nski carpet, provided that it has empty interior. Thus, Sierpi\'nski packings can be regarded as a generalization of Sierpi\'nski carpets. 

The natural spaces that can be used to parametrize a Sierpi\'nski packing are \textit{round} Sierpi\'nski packings, i.e., packings whose peripheral continua are (possibly degenerate) closed disks. We now state our main theorem.

\begin{theorem}\label{theorem:main}
Let $Y=\widehat{\C}\setminus \bigcup_{i\in \N} q_i$ be a Sierpi\'nski packing whose peripheral continua are closed Jordan regions or points with diameters in $\ell^2(\N)$.  
Then there exist
\begin{enumerate}[label=\upshape(\Alph*)]
	\item\label{theorem:main:set} a collection of disjoint closed disks $\{p_i\}_{i\in \N}$, where $p_i$ is degenerate if and only if $q_i$ is degenerate, a round Sierpi\'nski packing $X=\widehat{\C}\setminus \bigcup_{i\in \N}p_i$, 
	\item\label{theorem:main:map}  a continuous, surjective, and monotone map $H\colon \widehat\C\to \widehat\C$ with the property that $H^{-1}(\inter(q_i))=\inter(p_i)$ for each $i\in \N$, and
	\item\label{theorem:main:derivative} a non-negative Borel function $\rho_H\in L^2(\widehat{\C})$,
\end{enumerate}
with the following properties.
	\begin{itemize}
		\item (Transboundary upper gradient inequality) There exists a curve family $\Gamma_0$ in $\widehat{\C}$ with $\md_2\Gamma_0=0$ such that for all curves $\gamma\colon[a,b]\to \widehat \C$ outside $\Gamma_0$ we have
	\begin{align*}
\sigma( H(\gamma(a)), H(\gamma(b)) )\leq \int_{\gamma}\rho_H\, ds + \sum_{i:p_i\cap |\gamma|\neq \emptyset} \diam(q_i)
\end{align*}
		\item (Conformality) For each Borel set $E\subset \widehat\C$ we have
	$$\int_{H^{-1}(E)} \rho_H^2\, d\Sigma\leq  \Sigma(E\cap Y).$$ 
	\end{itemize}
Moreover, if $Y$ is cofat, then $H$ may be taken to be a homeomorphism of the sphere. 
\end{theorem}

Here $\sigma$  denotes the spherical distance and $\Sigma$ is the spherical measure. The monotoniciy of $H$ means that the preimage of every point is a continuum and is equivalent to the statement that $H$ is the uniform limit of homeomorphisms; see Section \ref{section:topological}. The map $H$ in the conclusion of the theorem is called a \textit{packing-conformal map}. Our definition of a packing-conformal map is motivated by the transboundary modulus of Schramm and by the so-called analytic definition of quasiconformality for maps between metric spaces \cite{Williams:qc}. Moreover, an analogous definition under the terminology \textit{weakly quasiconformal map} has been used recently by Romney and the author \cites{NtalampekosRomney:length, NtalampekosRomney:nonlength} in the solution of the problem of quasiconformal uniformization of spheres of finite area.

If $U$ is an open subset of $\widehat{\C}$ contained in the packing $Y$, then the map $H$ of Theorem \ref{theorem:main} is a conformal map in $H^{-1}(U)$ in the usual sense. However, not every conformal map between domains satisfies the transboundary upper gradient inequality. Nevertheless, one can show that this is always the case for countably connected domains.

\begin{remark}
We remark that although $H^{-1}(\inter(q_i))=\inter(p_i)$ in Theorem \ref{theorem:main} \ref{theorem:main:map}, the continuum $H^{-1}(q_i)$ might be larger than the disk $p_i$ when the packing $Y$ is not cofat. It is precisely this phenomenon that prevents us from proving Koebe's conjecture with this method. However, a non-trivial consequence of the topological and regularity conditions of Theorem \ref{theorem:main} is that the map $H$ is degenerate on the set $H^{-1}(q_i)\setminus p_i$, in the sense that it maps each continuum $E\subset H^{-1}(q_i)\setminus p_i$ to a point; see Figure \ref{figure:branches}. We prove this fact in Proposition \ref{proposition:constant}. 
\end{remark}

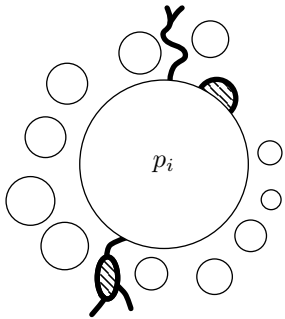
\begin{figure}
\centering
\input{branches.tikz}
\caption{The set $H^{-1}(q_i)$ might be larger than the disk $p_i$. In this figure it contains the shaded regions and bold ``branches", which are subsets of $X$. However,  the map $H$ is constant in each of them.}\label{figure:branches}
\end{figure}

In fact, Theorem \ref{theorem:main} is a consequence of a more general  uniformization theorem for Sierpi\'nski packings $Y$ without the topological assumption that the peripheral continua are closed Jordan regions or points. To each Sierpi\'nski packing $Y$ we can associate a topological sphere $\mathcal E(Y)$ by collapsing all peripheral continua to points, in view of Moore's decomposition theorem \cite{Moore:theorem}. 

\begin{theorem}\label{theorem:main:general}
Let $Y=\widehat{\C}\setminus \bigcup_{i\in \N} q_i$ be a Sierpi\'nski packing such that the diameters of the peripheral continua lie in $\ell^2(\N)$. Then there exists a round Sierpi\'nski packing $X$  and a packing-conformal map from $\mathcal E(X)$ onto $\mathcal E(Y)$. 
\end{theorem}

As we see, the uniformizing packing-conformal map exists only at the level of the topological spheres $\mathcal E(X),\mathcal E(Y)$, and in general does not induce a map between the packings $X,Y$ in the sphere $\widehat{\C}$. For the definition of packing-conformal maps between the associated topological spheres see Section \ref{section:packing_qc}. The  above theorem is restated as Corollary \ref{corollary:uniformization_disk}. The statement is proved via an approximation argument. We consider the finitely connected domains $Y_n=\widehat{\C}\setminus \bigcup_{i=1}^n q_i$ and we uniformize them conformally by finitely connected circle domains $X_n$ using Koebe's theorem. Then our task is to show that the conformal maps from $X_n$ to $Y_n$ converge in a uniform sense to the desired limiting map from $\mathcal E(X)$ onto $\mathcal E(Y)$. 

This proof strategy is also followed by Schramm \cite{Schramm:transboundary} in showing that cofat domains can be uniformized by circle domains and in \cite{Bonk:uniformization} in uniformizing Sierpi\'nski carpets by round carpets. The recent developments in the field of analysis on metric spaces and  our much more thorough understanding of quasiconformal maps between metric spaces allow us to identify the topological and regularity properties of the limiting map in our more fractal setting, where no uniform geometry is imposed, as in the works of Schramm and Bonk.  

We note that in unpublished work, Rohde and Werness \cite{RohdeWerness:CLE} show that the complementary disks of the circle domain $X_n$ converge in the Hausdorff sense after passing to a subsequence to a collection of pairwise disjoint disks. However, they were not able to identify the limit of the conformal maps from $X_n$ to $Y_n$.

Theorem \ref{theorem:main} is proved by showing that the topological assumptions on $Y$ allow one to lift a packing-conformal map between $\mathcal E(X)$ and $\mathcal E(Y)$ as in Theorem \ref{theorem:main:general} to the map $H$ in the sphere $\widehat{\C}$ that has the desired properties. This lifting process is achieved through Theorem \ref{theorem:continuous_extension}, which provides a monotone map $H$. In the case that $Y$ is cofat, the homeomorphism $H$ as in the last part of Theorem \ref{theorem:main} is provided by Theorem \ref{theorem:homeomorphic:packing}. Thus, Theorem \ref{theorem:main} is a consequence of Corollary \ref{corollary:uniformization_disk}, Theorem \ref{theorem:continuous_extension}, and  Theorem \ref{theorem:homeomorphic:packing}. 

Another generalization is that we do not need to restrict to round Sierpi\'nski packings $X$ in order to parametrize a given packing $Y$. Instead of using geometric disks as the peripheral continua of $X$, one can  use homothetic images of any countable collection of uniformly fat and non-separating continua, such as squares.  See Corollary \ref{corollary:uniformization_cofat_general} for the precise statement.

As a corollary of the main theorem we give an answer to a question of Rohde--Werness \cite{RohdeWerness:CLE} regarding the uniformization of the conformal loop ensemble (CLE) carpet. CLE was introduced by Sheffield--Werner \cite{SheffieldWerner:CLE}, as a random collection of Jordan curves in the unit disk that combines conformal invariance and a natural restriction property; see Figure \ref{figure:cle}. Each CLE gives rise to a Sierpi\'nski carpet with non-uniform geometry; hence the current carpet uniformization theory of \cites{Bonk:uniformization,Ntalampekos:CarpetsThesis} is not sufficient to treat them.  However, Rohde--Werness \cite{RohdeWerness:CLE} proved in unpublished work that, with probability $1$, the diameters of the peripheral disks of a CLE carpet are square-summable. Therefore,  we obtain the following corollary of the main theorem.

\begin{figure}
\includegraphics[scale=.1]{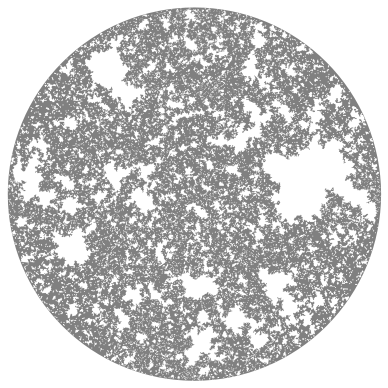}
\caption{A CLE carpet (simulation
by D.B.\ Wilson \cite{Wilson:cle}).}\label{figure:cle}
\end{figure}
\begin{corollary}
If $Y$ is a CLE carpet, almost surely there exists a round Sierpi\'nski packing $X$ and a packing-conformal map that maps $\br X$ onto $Y$. 
\end{corollary}

It would be interesting to obtain some stronger statements for the uniformization of CLE carpets. We pose several questions for further study.

\begin{question}
Under what conditions is the uniformizing round Sierpi\'nski packing $X$ and  the packing-conformal map $H$ of Theorem \ref{theorem:main} unique (up to M\"obius transformations)? 
\end{question}
If one could at least show the uniqueness of $X$, then this would imply that CLE gives rise to another stochastic process that generates round packings. 

\begin{question}Under what conditions is $\br X$ a carpet whenever $\br Y$ is a carpet?
\end{question}
Theorem \ref{theorem:main} already shows that a sufficient condition is the cofatness of $Y$. What about CLE carpets?

\begin{question}
Can one use the present techniques to prove Koebe's conjecture for domains whose complementary components have diameters in $\ell^2(\N)$?
\end{question}

Another natural question is whether one can obtain alternative proofs of results of \cites{Bonk:uniformization,Ntalampekos:CarpetsThesis} upon strengthening the geometric assumptions on the peripheral continua. 

\begin{question}
If the peripheral continua of a packing $Y$ are uniformly relatively separated uniform quasidisks, is the map $H$ of Theorem \ref{theorem:main} quasisymmetric? 
\end{question}

\begin{question}
If the peripheral continua of a packing $Y$ are uniformly fat and uniformly quasiround, is the map $H$ of Theorem \ref{theorem:main} carpet-conformal in the sense of \cite{Ntalampekos:CarpetsThesis}?
\end{question}

In the subsequent paper \cite{Ntalampekos:rigidity_cned}, we use the notion of packing-conformal maps and the results of the present paper in order to study the problem of \textit{conformal rigidity of circle domains}, a problem that is closely related to the uniqueness in Koebe's conjecture. A circle domain is \textit{conformally rigid} if every conformal map from that domain onto another circle domain is the restriction of a M\"obius transformation. Specifically, we prove that circle domains whose boundary is $\CNED$ (\textit{countably negligible for extremal distance}), as introduced in \cite{Ntalampekos:exceptional}, are conformally rigid. This result unifies and extends all previous works in the subject \cites{HeSchramm:Uniformization, HeSchramm:Rigidity, NtalampekosYounsi:rigidity}. Moreover, it provides strong evidence for a conjecture of He--Schramm, asserting that a circle domain is rigid if and only if its boundary is \textit{conformally removable}.

\subsection*{Acknowledgments}
I would like to thank Dennis Sullivan for many thought-provoking and motivating conversations on the subject of planar uniformization and Koebe's conjecture. I first learned of the problem of the uniformization of CLE from Huy Tran whom I also thank for various discussions.

\section{Preliminaries}

\subsection{Notation and terminology}

Let $(X,d)$ be a metric space. The open ball of radius $r>0$, centered at a point $x\in X$ is denoted by $B_d(x,r)$; the corresponding closed ball is $\br{B_d}(x,r)$. We also denote by $S(x,r)$ the circle $\{y\in X: d(x,y)=r\}$. The diameter of a set $E$ is denoted by $\diam_d(E)$. If the metric is implicitly understood we will often drop the symbol $d$ from the subscript. For the Euclidean metric in the plane we will use the subscript $e$, when necessary. For example, we write $B_e(x,r)$ and  $\diam_e(E)$. Finally, we denote by $\sigma$ the spherical metric and by $\Sigma$ the spherical measure on the Riemann sphere $\widehat{\C}$. We will often  regard $\C$ as a subset of $\widehat{\C}$, identified with its image through stereographic projection from the north pole of $\widehat{\C}$. 

A continuous function $\gamma$ from a compact interval $[a,b]$ into $X$ is called a \textit{compact curve}. A continuous function $\gamma$ from $(a,b)$ into $X$ is called a \textit{non-compact curve}. In this case, if  $\gamma$ extends continuously to a map $\br \gamma \colon [a,b]\to X$ then it is called an \textit{open curve}. The \textit{trace} of a curve $\gamma\colon I\to X$ is the set $\gamma(I)$ and is denoted by $|\gamma|$. A curve $\gamma\colon [a,b]\to X$ is \textit{closed} if $\gamma(a)=\gamma(b)$.

For $s\geq 0$ the \textit{$s$-dimensional Hausdorff measure} $\mathcal H^s(E)$ of a set $E$ in a metric space $X$ is defined by
$$\mathcal{H}^{s}(E)=\lim_{\delta \to 0} \mathcal{H}_\delta^{s}(E)=\sup_{\delta>0} \mathcal{H}_\delta^{s}(E),$$
where
$$
\mathcal{H}_\delta^{s}(E)=\inf \left\{ c(s)\sum_{j=1}^\infty \operatorname{diam}(U_j)^{s}: E \subset \bigcup_j U_j,\, \operatorname{diam}(U_j)<\delta \right\}
$$
for a normalizing constant $c(s)>0$ so that the $n$-dimensional Hausdorff measure agrees with Lebesgue measure in $\R^n$. Note that $c(1)=1$. We will use the notation $\mathcal H^s_d$ for the Hausdorff measure $\mathcal H^s$ if we wish to emphasize that the metric $d$ is used. We now state the co-area inequality for Lipschitz functions \cite[Prop.\ 3.1.5]{AmbrosioTilli:metric}.
\begin{prop}\label{proposition:coarea}
Let $L>0$ and $\psi\colon \widehat{\C}\to \R$ be an $L$-Lipschitz function. Then for each Borel function $\rho\colon \widehat{\C}\to [0,\infty]$ we have 
\begin{align*}
\int_{\R} \int_{\psi^{-1}(t)}\rho\, d\mathcal H^1 dt \leq \frac{4L}{\pi} \int_{\widehat{\C}} \rho \, d\Sigma. 
\end{align*}
\end{prop}

The cardinality of a set $E$ is denoted by $\#E$. For quantities $A$ and $B$ we write $A\lesssim B$ if there exists a constant $c>0$ such that $A\leq cB$. If the constant $c$ depends on another quantity $H$ that we wish to emphasize, then we write instead $A\leq c(H)B$ or $A\lesssim_H B$. Moreover, we use the notation $A\simeq B$ if $A\lesssim B$ and $B\lesssim A$. As previously, we write $A\simeq_H B$ to emphasize the dependence of the implicit constants on the quantity $H$. All constants in the statements are assumed to be positive even if this is not stated explicitly and the same letter may be used in different statements to denote a different constant.  

Let $E$ be a set in a metric space $X$. For $r>0$ we denote by $N_r(E)$ the open $r$-neighborhood of $E$. The Hausdorff distance of two sets $E,F\subset X$, denoted by $d_H(E,F)$, is defined as the infimum of all $r>0$ such that $E\subset N_r(F)$ and $F\subset N_r(E)$. We say that a sequence of sets $E_n\subset X$, $n\in \N$, converges to a set $E$ in the Hausdorff distance if $d_H(E_n,E)\to 0$ as $n\to\infty$. If the limiting set $E$ is closed, then it consists precisely of all limit points of sequences $x_n\in E_n$, $n\in \N$. 

A \textit{continuum} is a compact and connected set. An elementary property of Hausdorff convergence is that it preserves connectedness; namely, if a sequence of continua $E_n$, $n\in \N$, converges to a compact set $E$, then $E$ is also a continuum. See \cite[Section 7.3.1]{BuragoBuragoIvanov:metric} for more background.

\begin{lemma}\label{lemma:distance_convergence}
Let $X,Y$ be compact metric spaces and  $\pi\colon X\to Y$ be a continuous  map. 
\begin{enumerate}[\upshape(i)]
	\item\label{lemma:distance:i} Let $E\subset Y$ (resp.\ $E\subset X$) be a compact set and $\{E_n\}_{n\in \N}$ be a sequence of compact sets with the property that for each $r>0$ there exists $N\in \N$ such that $E_n\subset N_r(E)$ for all $n>N$. Then for each $r>0$ there exists $N\in \N$ such that 
	\begin{align*}
	\pi^{-1}(E_n)\subset N_r(\pi^{-1}(E))\quad  \textrm{( resp.\  $\pi(E_n)\subset N_r(\pi(E))$ )} 
	\end{align*}
	for all $n>N$.
	\item\label{lemma:distance:ii} If $\{E_n\}_{n\in \N}$, $\{F_n\}_{n\in \N}$ are sequences of compact sets in $Y$ converging in the Hausdorff sense to compact sets $E,F$, respectively, then
\begin{align*}
\dist(\pi^{-1}(E),\pi^{-1}(F))\leq \liminf_{n\to\infty} \dist( \pi^{-1}(E_n),\pi^{-1}(F_n)). 
\end{align*}
\end{enumerate}

\end{lemma}
\begin{proof}
The first part follows from compactness and continuity. For the second part, note that
\begin{align*}
\dist( \pi^{-1}(E),\pi^{-1}(F))=\lim_{r\to 0} \dist( N_r(\pi^{-1}(E)),N_r(\pi^{-1}(F))). 
\end{align*} 
By the first part, for each $r>0$, we have
$$\dist( N_r(\pi^{-1}(E)),N_r(\pi^{-1}(F))) \leq \dist( \pi^{-1}( E_n),\pi^{-1}(F_n))$$
for all sufficiently large $n\in \N$. This completes the proof.
\end{proof}

\subsection{Fat sets}\label{section:fat}

Let $\tau>0$.  A measurable set $K\subset \widehat{\C}$ is \textit{$\tau$-fat} if for each $x\in K$ and for each ball $B_{\sigma}(x,r)$ that does not contain $K$ we have $\Sigma( B_{\sigma}(x,r)\cap K) \geq \tau r^2$. A set is \textit{fat} if it is $\tau$-fat for some $\tau>0$. Note that points are automatically $\tau$-fat for every $\tau>0$.   A more modern terminology for fatness Ahlfors $2$-regularity, but we prefer to use the original terminology that was used by Schramm \cite{Schramm:transboundary} and Bonk \cite{Bonk:uniformization}. 

\begin{lemma}\label{lemma:fat_invariant}
Let $\tau>0$. If a connected set $K\subset \widehat{\C}$ is $\tau$-fat and $T$ is a M\"obius transformation, then $T(K)$ is $c(\tau)$-fat.
\end{lemma}

\begin{proof}
Schramm \cite[Theorem 2.1]{Schramm:transboundary} established the invariance of fatness under M\"obius transformations. However, Schramm's definition of fatness uses the Euclidean metric rather than the spherical one, requiring that $\mathcal H^2_{e}(B_{e}(x,r)\cap K)\geq \tau r^2$ for every $x\in K\cap \C$ and ball $B_e(x,r)$ that does not contain $K$. Hence, it suffices to show that fatness according to the spherical metric is equivalent to fatness according to the Euclidean metric, quantitatively.   We assume that $\C$ is embedded into $\widehat{\C}$ via stereographic projection.

Suppose that $K$ is $\tau$-fat according to the spherical metric and fix $x\in K\cap \C$ and a ball $B_{e}(x,r)$ that does not contain $K$. Consider the annuli $A_{e}(x; (n-1)r/5, nr/5)=B_e(x,nr/5)\setminus \br{B_e}(x,(n-1)r/5)$, $n\in \{2,\dots,5\}$. One of these annuli, say $A$, has the property that $\dist_{e} (A,0) \geq r/5$. Since $K$ is connected and not contained in $B_{e}(x,r)$, which contains $A$, there exists a point $y\in K$ lying in the circle that is equidistant from the boundary circles of $A$. The ball $B_{e}(y,r/10)$ is contained in $A\subset B_{e}(x,r)$. For all points  $z\in B_{e}(y,r/10)$ we have $|z|\simeq r$. Thus, for all $z,w\in B_{e} (y,r/10)$ we have  $\sigma(z,w) \simeq ({1+r^2})^{-1} |z-w|$. This implies that 
$$\frac{\mathcal H^2_{e}(B_{e}(x,r)\cap K)}{r^2} \gtrsim 
 \frac{\mathcal H^2_{e}(B_{e}(y,r/10)\cap K)}{\mathcal H^2_{e}( B_{e}(y,r/10))} \simeq \frac{\Sigma (B_{e}(y,r/10)\cap K)}{\Sigma( B_{e}(y,r/10))} \simeq_{\tau} 1,$$
since a Euclidean ball is also a spherical ball (of possibly different radius). Therefore, we conclude the fatness of $K$ according the Euclidean metric, quantitatively.

Conversely, suppose that $K$ is fat according to the Euclidean metric and fix $x\in K$ and a ball $B_{\sigma}(x,r)$ that does not contain $K$. We apply an isometry $P$ of $\widehat{\C}$ so that $P(x)=0$. The set $P(K)$ is also fat according to the Euclidean metric by Schramm's result. We have
\begin{align*}
{\Sigma(B_{\sigma}(x,r)\cap K)} \geq   {\Sigma(B_{\sigma}(x,r/4)\cap K)} = {\Sigma(B_{\sigma}(0,r/4)\cap P(K))}.
\end{align*}
Since $B_{\sigma}(0,r/4)$ is contained in the unit disk in the plane, the identity map from $B_{\sigma}(0,r/4)$ into $(\C,|\cdot|)$ is uniformly bi-Lipschitz and $B_{\sigma}(0,r/4)$ corresponds to a Euclidean ball $B_{e}(0,cr)$ for some constant $c\simeq 1$. Thus, 
$$\Sigma(B_{\sigma}(0,r/4)\cap P(K))\simeq \mathcal H^2_{e}(B_{e}(0,cr)\cap P(K))\gtrsim_{\tau} r^2.$$
This completes the proof.   
\end{proof}

\begin{lemma}\label{lemma:fat_hausdorff}
Let $\tau>0$ and $K_n\subset \widehat{\C}$, $n\in \N$, be a sequence of $\tau$-fat compact sets. Then every compact limit of $\{K_n\}_{n\in \N}$ in the Hausdorff sense is $\tau$-fat.
\end{lemma}
\begin{proof}
Let $K\subset \widehat{\C}$ be a compact set that is the Hausdorff limit of a subsequence of $K_n$, $n\in \N$, which we denote by $K_n$ for the sake of simplicity. If $K$ is a point, then $K$ is trivially $\tau$-fat, so without loss of generality, we assume that $\diam(K)>0$. Let $x\in K$ and $B(x,r)$ be a ball that does not contain $K$. Our goal is to show that $\Sigma(B(x,r)\cap K)\geq \tau r^2$. Let $\varepsilon>0$ and $U\supset K$ be an open set such that $\Sigma(B(x,r)\cap K) \geq \Sigma(B(x,r)\cap U)-\varepsilon$. For all sufficiently large $n\in \N$, we have $K_n\subset U$ by the Hausdorff convergence. Moreover, there exists a sequence $x_n\in K_n$ converging to $x$ such that for each $\delta>0$ we have $B(x_n,r-\delta)\subset B(x,r)$ for all sufficiently large $n\in \N$. Since $K\not\subset B(x,r)$, we have $K_n\not\subset B(x_n,r-\delta)$ for all sufficiently large $n\in \N$. Altogether, for all sufficiently large $n\in \N$ we have
\begin{align*}
\Sigma(B(x,r)\cap K) \geq \Sigma(B(x,r)\cap U)-\varepsilon \geq \Sigma(B(x_n,r-\delta) \cap K_n)-\varepsilon\geq \tau (r-\delta)^2-\varepsilon.
\end{align*}
We let $\delta\to 0$, and then $\varepsilon\to 0$ to obtain the desired conclusion.
\end{proof}

The next statement can be found \cite[Lemma 2.6 (iii)]{MaioNtalampekos:packings} in a slightly altered form.
\begin{lemma}\label{lemma:count}
Let $\tau>0$ and $\{p_i\}_{i\in \N}$ be a collection of disjoint $\tau$-fat continua in $\widehat{\C}$. For each compact set $E\subset \widehat{\C}$ and $a>0$ the set 
$$\{i: p_i\cap E\neq \emptyset  \,\,\, \textrm{and}\,\,\, \diam(p_i)\geq a\diam(E)\}$$
has at most $c(\tau,a)$ elements.
\end{lemma}

We also record an elementary consequence of fatness; see \cite[Property (F2), p.~154]{NtalampekosYounsi:rigidity} for a proof.
\begin{lemma}\label{lemma:radial}
Let $\tau>0$ and $K\subset \widehat{\C}$ be a connected $\tau$-fat set.  Then for each $x\in \widehat{\C}$ and $r>0$ we have
\begin{align*}
\mathcal H^1( \{s\in (0,r) : K\cap S(x,s)\neq \emptyset\})^2 \leq c(\tau) \Sigma(K\cap B(x,r)).  
\end{align*}
\end{lemma}

A metric measure space $(X,d,\mu)$ is \textit{doubling} if every ball in $X$ has positive and finite measure and there exists a constant $L>0$ such that 
$$\mu(B(x,2r))\leq L \mu(B(x,r))$$
for each $x\in X$ and $r>0$. In this case, we say that $X$ is $L$-doubling. 

\begin{lemma}\label{lemma:bojarski}
Let $(X,d,\mu)$ be an $L$-doubling metric measure space for some $L>0$. Let $p\geq 1$, $a\geq 1$, and $\{b_i\}_{i\in I}$ be a collection of non-negative numbers. Suppose that $\{D_i\}_{i\in I}$ is a family of measurable sets and $\{B_i=B(x_i,r_i)\}_{i\in I}$ is a family of balls in $X$ with the property that $D_i\subset B_i$ and $\mu(B_i)\leq a \mu(D_i)$ for each $i\in I$. Then
\begin{align*}
\left \| \sum_{i\in I} b_i\x_{B_i}\right \|_{L^p(X)} \leq c(L,p,a)  \left \| \sum_{i\in I} b_i\x_{D_i} \right \|_{L^p(X)}
\end{align*}
\end{lemma}
\begin{proof}Note that for $p=1$ the proof of the inequality is straightforward. Suppose that $p>1$.  For a non-negative measurable function $\phi$ on $X$ consider the centered maximal function $M\phi$. For each $i\in I$ and $x\in D_i\subset B_i$ we have
\begin{align*}
\frac{1}{\mu(B_i)}\int_{B_i}\phi \leq \frac{1}{\mu(B_i)} \int_{B(x,2r_i)} \phi \leq L  M\phi(x).
\end{align*}
Thus, 
\begin{align*}
\int_{D_i} M\phi \geq L^{-1 }\mu(D_i)  \frac{1}{\mu(B_i)}\int_{B_i}\phi \geq L^{-1}a^{-1} \int_{B_i} \phi.
\end{align*}

Now, let $f=\sum_{i\in I} b_i\x_{B_i}$ and  $\phi$ be an arbitrary non-negative function with $\|\phi\|_{L^q(X)}=1$, where $1/p+1/q=1$. Then, by the Hardy--Littlewood maximal inequality for doubling metric measure spaces \cite[Theorem 3.5.6, p.~92]{HeinonenKoskelaShanmugalingamTyson:Sobolev} we have
\begin{align*}
\int f \phi &=\sum_{i\in I} b_i \int_{B_i}\phi\leq c(L,a) \sum_{i\in I} b_i \int_{D_i} M\phi = c(L,a) \int \left(\sum_{i\in I} b_i\x_{D_i} \right)M\phi \\
&\leq   c(L,a)	\left \| \sum_{i\in I} b_i\x_{D_i} \right \|_{L^p(X)} \|M\phi\|_{L^q(X)}\leq c(L,p,a)\left \| \sum_{i\in I} b_i\x_{D_i} \right \|_{L^p(X)}.
\end{align*}
The duality between $L^p$ and $L^q$ shows the desired inequality. 
\end{proof}

\subsection{Topological preliminaries}\label{section:topological}

Let $\nu\colon X\to Y$ be a continuous map between topological spaces. The map $\nu$ is \textit{proper} if the preimage of each compact set is compact. The map $\nu$ is \textit{monotone} if the preimage of each point is a continuum.  Moreover, $\nu$ is \textit{cell-like} if the preimage of each point is a continuum that is contractible in all of its open neighborhoods. In $2$-manifolds without boundary cell-like continua coincide with sets that have a simply connected neighborhood that they do not separate. In the $2$-sphere, this condition is simply equivalent to the condition that the continuum is non-separating.

If $\nu$ is a continuous map from the $2$-sphere onto itself, then it is  monotone if and only if it is cell-like if and only if it is the uniform limit of homeomorphisms; see \cite{Youngs:monotone} and \cite[Theorem 6.3]{NtalampekosRomney:nonlength}. In analogy, for non-compact manifolds without boundary, we have the Armentrout--Quinn--Siebenmann approximation theorem \cite[Corollary 25.1A, p.~189]{Daverman:decompositions}. We formulate the theorem according to \cite{Siebenmann:approximation}.
\begin{theorem}\label{theorem:approximation}
A continuous, proper, and cell-like map $\nu\colon X\to Y$ between $2$-manifolds without boundary can be uniformly approximated by homeomorphisms in the following strong sense. For each continuous function $\epsilon\colon X\to (0,\infty)$ and for each metric $d$ on $Y$ compatible with the topology, there exists a homeomorphism $\widetilde \nu\colon X\to Y$ such that $$d( \nu(x),\widetilde \nu(x)) <\epsilon(x)$$ 
for each $x\in X$. 
\end{theorem}

\begin{lemma}\label{lemma:cell_like}
Let $\nu\colon X\to Y$ be as in Theorem \ref{theorem:approximation}. 
\begin{enumerate}[\upshape(i)]
	\item\label{lemma:cell_like:surjective} The map $\nu$ is surjective.
	\item\label{lemma:cell_like:open} For each open set $U\subset Y$ the set $\nu^{-1}(U)$ is homeomorphic to $U$.
	\item\label{lemma:cell_like:continuum} A compact set $E\subset Y$ is connected if and only if $\nu^{-1}(E)$ is connected.
	\item\label{lemma:cell_like:set} A compact set $E\subset Y$ is cell-like if and only if $\nu^{-1}(E)$ is cell-like. 
	\item\label{lemma:cell_like:boundary} Let $U\subset Y$ be an open set such that $\nu$ is injective on $\nu^{-1}(\partial U)$. Then $\nu^{-1}(\partial U)=\partial \nu^{-1}(U)$.
\end{enumerate}
\end{lemma}
\begin{proof}
For \ref{lemma:cell_like:surjective}, suppose that $y_0\in Y\setminus \nu(X)$. The properness of $\nu$ implies that there exists a ball $B(y_0,r)$ that is disjoint from $\nu(X)$. This contradicts the conclusion of Theorem \ref{theorem:approximation}. For \ref{lemma:cell_like:open}, let $U\subset Y$ be an open set. The map $\nu$ is a continuous, proper, and cell-like map from $\nu^{-1}(U)$ onto $U$. Theorem \ref{theorem:approximation} implies that $\nu$ is the uniform limit of homeomorphisms from $\nu^{-1}(U)$ onto $U$. This proves \ref{lemma:cell_like:open}. The non-trivial direction in \ref{lemma:cell_like:continuum} follows from the fact that $\nu|_{\nu^{-1}(E)}$ is a monotone map from the compact set $\nu^{-1}(E)$ onto $E$; by \cite[(2.2), Chap.\ VIII, p.~138]{Whyburn:topology}, in this setting, the preimage of a connected set is connected. For part \ref{lemma:cell_like:set}, if $E$ is cell-like, then there exists a simply connected neighborhood $U\supset E$ such that $U\setminus E$ is connected. Then by \ref{lemma:cell_like:open} $\nu^{-1}(U)$ is a simply connected neighborhood of $\nu^{-1}(E)$ that is not separated by $\nu^{-1}(E)$; hence $E$ is cell-like. Conversely, if $\nu^{-1}(E)$ is cell-like and $U$ is a neighborhood of $E$, then $\nu^{-1}(E)$ is contractible in $\nu^{-1}(U)$, so $E$ is contractible in $U$.  

For \ref{lemma:cell_like:boundary}, by continuity we have $\partial \nu^{-1}(U)\subset \nu^{-1}(\partial U)$. Moreover, for each $y\in \br U$ the set $\nu^{-1}(y)$ intersects $\br{\nu^{-1}(U)}$; indeed, if $y_n\in U$ and $y_n\to y$, then by properness there exists a sequence $x_n\in \nu^{-1}(y_n)\subset \nu^{-1}(U)$ converging to a point $x\in \nu^{-1}(y)$. The injectivity of $\nu$ on $\nu^{-1}(\partial U)$ implies that if $x\in \nu^{-1}(\partial U)$, then $x=\nu^{-1}(\nu(x))$. By the previous this point lies lies in $\br{\nu^{-1}(U)}$, and thus in $\partial \nu^{-1}(U)$. 
\end{proof}

We will also use Moore's theorem \cite{Moore:theorem}, which facilitates the study of planar domains, Sierpi\'nski packings, and Sierpi\'nski carpets.  Let $G$ be a partition of $\widehat{\C}$ into disjoint continua. We call $G$ a \textit{decomposition} of $\widehat{\C}$. We say that the decomposition $G$ is \textit{upper semicontinuous} if for each $g\in G$ and each open set $U\subset \widehat{\C}$ containing $g$, there exists an open set $V\subset\widehat{\C}$ containing $g$ such that if $g'\in G$ and $g'\cap V\neq \emptyset$, then $g'\subset U$. Equivalently, if $g_n\in G$, $n\in \N$, is a sequence that converges in the Hausdorff sense to a compact set $A$, then there exists $g\in G$ such that $A\subset G$. We now state Moore's theorem; see \cite[Theorem 25.1]{Daverman:decompositions} for a modern proof.

\begin{theorem}\label{theorem:moore:original}
If $G$ is an upper decomposition of $\widehat{\C}$ into non-separating continua, then $\widehat{\C}/G$  is homeomorphic to $\widehat{\C}$.  
\end{theorem}

\subsection{Sierpi\'nski packings}

We recall some definitions from the introduction. Let $\{p_i\}_{i\in \N}$ be a collection of pairwise disjoint, non-separating continua in $\widehat{\C}$ such that $\diam(p_i)\to 0$ as $i\to\infty$. The collection $\{p_i\}_{i\in \N}$ is called a Sierpi\'nski packing and the set  $X=\widehat{\C}\setminus \bigcup_{i\in \N} p_i$ is its {residual set}. When there is no confusion, we call $X$ a Sierpi\'nski packing and the underlying collection $\{p_i\}_{i\in \N}$ is implicitly understood. The continua $p_i$, $i\in \N$, are  called the {peripheral continua} of $X$. A Sierpi\'nski packing (resp.\ domain) is \textit{cofat} if there exists $\tau>0$ such that each of its peripheral continua (resp.\ complementary components) is $\tau$-fat.

Let $X=\widehat{\C}\setminus \bigcup_{i\in I} p_i$ be a Sierpi\'nski packing or a domain, where in the latter case the collection $\{p_i\}_{i\in I}$ is assumed to comprise the complementary components. We  consider the quotient space $\mathcal E(X)=\widehat{\C}/\{p_i\}_{i\in I}$, together with the natural projection map $\pi_X\colon \widehat{\C} \to \mathcal E(X)$. For a set $A\subset \widehat{\C}$ we denote $\pi_X(A)$ by $\widehat{A}$.  For a set $E\subset \widehat{\C}$, let $I_E=\{i\in I:p_i\cap E\neq \emptyset\}$.  We define $\mathcal E(X;E)=\widehat{\C}/\{p_i\}_{i\in I\setminus I_E}$; that is, the sets $p_i$ that intersect $E$ are not collapsed to points. 

If $X$ is a Sierpi\'nski packing or a domain, we note that the decomposition of $\widehat{\C}$ into the singleton points of $X$ and the continua $p_i$, $i\in I$, is always upper semicontinuous. In the case that $X$ is a Sierpi\'nski packing, the fact that $\diam(p_i)\to 0$ as $i\to\infty$ implies that for each set $E\subset \widehat \C$ the decomposition of $\widehat{\C}$ into the continua $p_i$, $i\in I\setminus I_E$, and the remaining singleton points is upper semicontinuous. Therefore, a consequence of Moore's theorem (Theorem \ref{theorem:moore:original}) is the following statement.

\begin{theorem}\label{theorem:moore_original}
Let $X$ be a Sierpi\'nski packing or a domain. Then $\mathcal E(X)$ is homeomorphic to $\widehat{\C}$. Moreover, if $X$ is a Sierpi\'nski packing and $E\subset \widehat{\C}$, then $\mathcal E(X;E)$ is homeomorphic to $\widehat{\C}$. 
\end{theorem}

A consequence of Lemma \ref{lemma:cell_like} that we will often use is that the preimages of continua under the projection maps $\pi_X$ and $\pi_{X;E}$ are continua.

\subsection{Transboundary modulus}
First, we give the definition of $2$-modulus on the sphere. Let $\Gamma$ be a family of curves in $\widehat{\C}$. We say that a Borel function $\rho\colon \widehat{\C}\to[0,\infty]$ is \textit{admissible} for the curve family $\Gamma$ if
$$\int_{\gamma}\rho\, ds\geq 1$$
for each locally rectifiable curve $\gamma\in \Gamma$. We then define the \textit{$2$-modulus}, or else \textit{conformal modulus}, of $\Gamma$ as
$$\md_2\Gamma=\inf_{\rho}\int \rho^2 \, d\Sigma,$$
the infimum taken over all admissible functions $\rho$. The next lemma is simple consequence of the co-area inequality; see \cite[Lemma 2.4.3]{Ntalampekos:CarpetsThesis} for an argument.

\begin{lemma}\label{lemma:lipschitz}
Let $f\colon \widehat{\C}\to \R$ be a Lipschitz function and $\Gamma_0$ be a family of curves in $\widehat\C$ with $\md_2\Gamma_0=0$. Then for a.e.\ $t\in \R$ every simple curve $\gamma$ whose trace is contained in $f^{-1}(t)$ lies outside $\Gamma_0$. 
\end{lemma}

Next, we define transboundary modulus, as introduced by Schramm \cite{Schramm:transboundary}. Let $X=\widehat{\C}\setminus \bigcup_{i\in I}p_i$ be a domain. Let $\rho\colon \mathcal E(X) \to [0,\infty]$ be a Borel function and $\gamma\colon [a,b]\to \mathcal E(X)$ be a curve. Then there exist countably many curves $\gamma_j$, $j\in J$, such that for each $j\in J$ we have $|\gamma_j|\subset \widehat{X}$  and $\gamma_j=\pi_X\circ \alpha_j$ for some possibly non-compact curve $\alpha_j$ in the domain $X$.  We define
\begin{align*}
\int_{\gamma} \rho \, ds= \sum_{j\in J}\int_{\alpha_j} \rho\circ \pi_X \, ds,
\end{align*}
where this is understood to be infinite if one of the curves $\alpha_j$ is not locally rectifiable. Let $\Gamma$ be a family of curves in $\mathcal E(X)$.  We say that a Borel function $\rho\colon \mathcal E(X)\to [0,\infty]$ is \textit{admissible} for $\Gamma$ if 
\begin{align*}
\int_{\gamma} \rho \, ds + \sum_{i:|\gamma|\cap \widehat p_i\neq \emptyset} \rho( \widehat p_i)\geq 1
\end{align*}
for each $\gamma\in \Gamma$. The \textit{transboundary modulus} of $\Gamma$ with respect to the domain $X$ is defined to be
\begin{align*}
\md_{X}\Gamma=  \inf_{\rho}\left\{ \int_{X}(\rho\circ \pi_X)^2 \, d\Sigma + \sum_{i\in \N}\rho(\widehat{p}_i)^2\right\},
\end{align*}
where the infimum is taken over all admissible functions $\rho$.

Let $X,Y$ be domains in $\widehat{\C}$ and $f\colon X\to Y$ be a conformal map. Then $f$ induces a homeomorphism $\widehat f\colon \mathcal E(X)\to \mathcal E(Y)$ such that $\widehat f=\pi_X \circ f\circ \pi_X^{-1}$ on $\widehat X$; see \cite[Section 3]{NtalampekosYounsi:rigidity} for a detailed discussion. It was observed by Schramm \cite{Schramm:transboundary} that transboundary modulus is invariant under conformal maps.

\begin{lemma}\label{lemma:transboundary_invariance}
Let $X,Y$ be domains in $\widehat{\C}$ and $f\colon X\to Y$ be a conformal map. Then for each curve family $\Gamma$ in $\mathcal E(X)$ we have
$$\md_X\Gamma=\md_Y \widehat f(\Gamma). $$
\end{lemma}

We also introduce the set function $f^*= \pi_{Y}^{-1}\circ \widehat{f}\circ \pi_X$ from the powerset of $\widehat{\C}$ into itself. In particular, $f^*=f$ on subsets of $X$ and if $A$ is contained in a boundary component of $X$, then $f^*(A)$ is the corresponding boundary component of $Y$.   Observe that if $g=f^{-1}$, then 
\begin{align}\label{inclusion:fgA}
g_n^*( f_n^*(A))\supset A
\end{align}
for each set $A\subset \widehat{\C}$ with equality if $A\subset X$. The next lemma is an implication of Carath\'eodory's kernel convergence theorem for multiply connected domains \cite[Theorem V.5.1, p.~228]{Goluzin:complex}.
\begin{lemma}\label{lemma:caratheodory_convergence}
Let $\Omega\subset \widehat{\C}$ be a domain and $f_n$, $n\in \N$, be a sequence of conformal maps in $\Omega$ that converges locally uniformly in $\Omega$ to a conformal map $f$. Then for each compact  $E\subset \widehat{\C}$ and for each compact limit $E^*$ of $\{f_n^*(E)\}_{n\in \N}$ in the Hausdorff sense we have
$$f^*(E) \supset E^*.$$
\end{lemma}

\begin{proof}
Suppose first that $E$ is a complementary component of $\Omega$. Since $f_n$ converges to $f$ locally uniformly, $f_n(\Omega)$ converges in the Carath\'eodory topology to $f(\Omega)$ (with respect to a point of $f(\Omega)$). This implies that each compact Hausdorff limit of $\widehat \C\setminus f_n(\Omega)$ is contained in $\widehat \C\setminus f(\Omega)$. Thus, if $E^*$ is a compact limit of  $f_n^*(E)$, which is a component of $\widehat{\C}\setminus f_n(\Omega)$, then $E^*$ is contained in a component of $\widehat{\C}\setminus f(\Omega)$. One can now see that this component has to be $f^*(E)$. 

For the general case, let $E$ be an arbitrary compact set in $\widehat{\C}$. By the definition of $f^*$, the set $f^*(E)$ is  the union of $f(\Omega\cap E)$ with the complementary components $f^*(B)$ of $f(\Omega)$, where $B$ is a complementary component of $\Omega$ such that $B\cap E\neq\emptyset$. The set $E^*$ consists of $f(\Omega\cap E)$ and Hausdorff limits of complementary components $f^*(B)$, where $B\cap E\neq \emptyset$. The previous case completes the proof.  
\end{proof}

\subsection{Packing-quasiconformal maps}\label{section:packing_qc}
For two Sierpi\'nski packings or domains $X,Y$, we introduce the notion of a packing-quasi\-conformal map between the associated topological spheres $\mathcal E(X),\mathcal E(Y)$. 

\begin{definition}\label{definition:packing_quasiconformal}
Let   $X=\widehat \C\setminus \bigcup_{i\in I} p_i$ and $Y=\widehat \C \setminus \bigcup_{i\in I} q_i$ be Sierpi\'nski packings or domains. Let $h\colon \mathcal E(X)\to\mathcal E(Y)$ a continuous, surjective, and monotone map such that $h(\widehat{p}_i)=\widehat q_i$ for each $i\in I$. We say that $h$ is \textit{packing-quasiconformal} if there exists $K\geq 1$ and a non-negative Borel function $\rho_h\in  L^2(\widehat{\C})$ with the following properties.
	\begin{itemize}
		\item \textit{(Transboundary upper gradient inequality)} There exists a curve family $\Gamma_0$ in $\widehat{\C}$ with $\md_2\Gamma_0=0$ such that for all curves $\gamma\colon [a,b]\to \widehat \C$ outside $\Gamma_0$ we have
	\begin{align*}
\dist( \pi_Y^{-1} \circ h \circ \pi_X(\gamma(a)), \pi_Y^{-1}\circ h\circ \pi_X(\gamma(b)))\leq \int_{\gamma}\rho_h\, ds + \sum_{i:p_i\cap |\gamma|\neq \emptyset} \diam(q_i)
\end{align*}
		\item \textit{(Quasiconformality)} For each Borel set $E\subset \widehat \C$ we have
	$$\int_{\pi_X^{-1}(h^{-1}(\pi_Y(E)))} \rho_h^2\, d\Sigma \leq K\Sigma  (E\cap Y).$$ 
\end{itemize}
In this case, we say that $h$ is \textit{packing-$K$-quasiconformal}. If $K=1$, then $h$ is called \textit{packing-conformal}.
\end{definition}

A Borel function $\rho_h$ satisfying the transboundary upper gradient inequality as above is called a \textit{transboundary weak upper gradient of $h$}. If the transboundary upper gradient inequality holds for all locally rectifiable curves in $\widehat{\C}$, without the need to exclude a family of conformal modulus zero, then we say that $\rho_h$ is a \textit{transboundary upper gradient} of $h$.

\begin{remark}\label{remark:packing_conformal}
Note that we are not requiring that $h^{-1}(\widehat q_i)=\widehat p_i$ and  $h^{-1}(\widehat q_i)$ could be much larger than $\widehat p_i$. The quasiconformality condition implies that if $E=q_i$, then $\rho_h=0$ a.e.\ on $\pi_X^{-1}(h^{-1}(\widehat q_i))$, $i\in \N$. Thus, $\rho_h$ is supported in the set $\pi_X^{-1}(h^{-1}(\widehat Y))$. In fact, we can set $\rho_h$ equal to $0$ \textit{everywhere} on $\pi_X^{-1}(h^{-1}(\widehat q_i))$, $i\in \N$, rather than almost everywhere. Indeed, line integrals are not affected by this change for $\md_2$-a.e.\ curve \cite[Lemma 5.2.16, p.~133]{HeinonenKoskelaShanmugalingamTyson:Sobolev}. Hence, by enlarging the exceptional curve family $\Gamma_0$, we may have that the transboundary upper gradient inequality also holds for the modified function $\rho_h$. 
\end{remark}

The following lemma is straightforward for finitely connected domains; it is also true for countably connected domains but we will not need that generality.

\begin{lemma}\label{lemma:packing_conformal}
Let $X,Y\subset\widehat{\C}$ be finitely connected domains and $f\colon X\to Y$ be a conformal map. Then the induced map $\widehat f \colon \mathcal E(X)\to \mathcal E(Y)$ is packing-conformal and the derivative of $f$ in the spherical metric is a transboundary upper gradient of $\widehat f$.
\end{lemma}

Here, the derivative of $f$ in the spherical metric at a point $z\in \widehat{\C}$  can be given by the following precise formula when $z,f(z)\in \C$.
\begin{align*}
|Df|(z)= \frac{1+|z|^2}{1+|f(z)|^2}|f'(z)|.
\end{align*}

\section{Transboundary modulus estimates}
Let $X=\widehat{\C}\setminus \bigcup_{i\in I} p_i$ be a Sierpi\'nski packing or a domain. Let $\widehat E,\widehat F\subset \mathcal E(X)$ be arbitrary sets  and $\Omega\subset \mathcal E(X)$ be an open set. We denote by $\Gamma(\widehat{E},\widehat{F};\Omega)$ the family of open paths in $\Omega$ joining $\widehat{E}$ and $\widehat{F}$. That is, $\Gamma(\widehat{E},\widehat{F};\Omega)$ contains precisely the open paths $\gamma\colon (a,b)\to \Omega$ such that $\br\gamma\colon[a,b]\to \br \Omega$ intersects both $\widehat{E}$ and $\widehat{F}$.  If $\Omega=\mathcal E(X)$, then we simply write $\Gamma(\widehat{E},\widehat{F})$.

Recall that for $E\subset \widehat{\C}$ we denote $I_E=\{i\in I: p_i\cap E\neq \emptyset\}$ and $\mathcal E(X;E)= \widehat{\C} / \{p_i\}_{i\in I\setminus I_E}$. Consider the natural projection $\pi_{X;E}\colon 
\widehat{\C} \to \mathcal E(X;E)$, which is injective on $X\cup E$. Recall that if $X$ is a Sierpi\'nski packing, and in particular $\diam(p_i)\to 0$ as $i\to \infty$, then the space $\mathcal E(X;E)$ is homeomorphic to the sphere $\widehat{\C}$ by Theorem \ref{theorem:moore_original}. 

The next lemma is one of the main technical ingredients of the proof of Theorem \ref{theorem:main}. 

\begin{lemma}[Non-degeneracy lemma]\label{lemma:l2collision_final}
Let $X=\widehat{\C}\setminus \bigcup_{i\in \N} p_i$ be a Sierpi\'nski packing such that the diameters of the peripheral continua lie in $\ell^2(\N)$. Let $\widehat{E}\subsetneq \mathcal E(X)$ be a continuum such that $E=\pi_X^{-1}(\widehat{E})$ is non-degenerate. For $\delta>0$ and a metric $d$ on $\mathcal E(X;E)$ that induces the quotient topology, let $F\subset \widehat{\C}\setminus E$ be a continuum such that $\diam_d( \pi_{X;E}(F))\geq \delta$.

\vspace{1em}

\noindent
For a finite set $J\subset \N$ consider the domain $Y=\widehat{\C}\setminus \bigcup_{i\in J} p_i$. Then for each $N\in \N\cup \{0\}$ and each set $J_0\subset J$ with $\#J_0\leq N$ we have  
\begin{align*}
\md_{Y} \Gamma\left(\pi_Y({E}),\pi_Y({F \setminus \bigcup_{i\in J_0} p_i}); \mathcal E(Y)\setminus \bigcup_{i\in J_0} \pi_{Y}(p_i) \right)\geq C( X , E,d,\delta, N ).
\end{align*}  
In particular, the lower bound does not depend on $F,Y$, and $J_0$. 
\end{lemma}

Surprisingly, this lemma gives uniform lower modulus bounds, although the packing $X$ does not have uniform geometry.  The reason is that we freeze a continuum $E$ and we consider curve families connecting relatively large continua $F$ to $E$. If one varies the continuum $E$ as well, then it is impossible to obtain uniform modulus bounds without some strong uniform geometric assumptions, as in \cite[Section 8]{Bonk:uniformization}. 

\begin{proof}
Note that $E\neq \widehat{\C}$, since $\widehat{E} \neq \mathcal E(X)$. The set $E$ is a continuum by Lemma \ref{lemma:cell_like} \ref{lemma:cell_like:continuum}.  The space  $Z=\mathcal E(X;E)$ is homeomorphic to $\widehat{\C}$. The projection $\pi=\pi_{X;E}\colon \widehat{\C}\to Z$ is injective on the set $E$, thus, $\pi(E)$ is a non-degenerate continuum in $Z$. We endow $Z$ with a metric $d$ inducing its topology, as in the statement of the lemma. We also fix $\delta>0$.

The set $Z\setminus \pi\left(E\cup \bigcup_{i\in \N}p_i\right)$ is  non-empty by Baire's theorem. Fix a point $x_{\infty}$ in that set and a ball $B_d(x_{\infty},r)\subset Z\setminus \pi(E)$, where $r<\delta/4$. Consider a  homeomorphism $\varphi$ from $Z\setminus \{x_{\infty}\}$ onto the plane $\C$ with the property that it maps the complement of the ball $B_d(x_{\infty},r)$ into the unit ball $B_e(0,1)$. Let $E'=\varphi( \pi(E))$, which is a non-degenerate continuum in $B_e(0,1)$. By uniform continuity, there exists a constant $c_0>0$ such that if $\widehat{F}\subset Z$ is a continuum with $\diam_d(\widehat{F})\geq \delta$, then $\varphi(\widehat{F}\setminus B_d(x_{\infty},r))$ contains a continuum $F'\subset B_e(0,1)$ with $\diam_e(F')\geq c_0$. We fix such a continuum $F'$.

Now, for $\varepsilon>0$ consider a grid of squares of side length $\varepsilon$ in the plane with sides parallel to the coordinate axes such that the $1$-skeleton avoids the countably many points $p_i'=\varphi (\pi(p_i))$, $i\in \N\setminus I_E$. Let $\mathcal G$ be the intersection of the $1$-skeleton with the square $[-2,2]^2$. Let $\mathcal B'$ be the collection of simple paths that are contained in $\mathcal G$ and connect all pairs of junction points (i.e., points where four edges meet) of $\mathcal G$. Note that $\mathcal B'$ contains a bounded number of paths, depending only on $\varepsilon$. Also, if two paths of $\mathcal B'$ do not intersect, then their distance is at least $\varepsilon$. For each path $\beta\in \mathcal B'$, consider a Jordan region  $U(\beta)$ such that $|\beta|\subset U(\beta)\subset N_{\varepsilon/2}(|\beta|)$ and $\partial U(\beta)$ avoids $\bigcup_{i\in \N\setminus I_E} p_i'$. In particular, observe that if $|\beta_1|\cap |\beta_2|=\emptyset$, then $U(\beta_1)\cap U(\beta_2)=\emptyset$.

Let $N\in \N\cup \{0\}$, as in the statement of the lemma. We claim that if the mesh of the grid is sufficiently small, i.e., $\varepsilon$ is sufficiently small, depending only on $N$, $E'$, and $c_0$, but not on $F'$, then $\mathcal B'$ contains $N+1$ disjoint paths $\beta_1,\dots,\beta_{N+1}$ that connect $E'$ and $F'$ with the additional feature that $\partial U(\beta_1),\dots,\partial U(\beta_{N+1})$ also connect $E'$ and $F'$. To see this, we consider separate cases. 

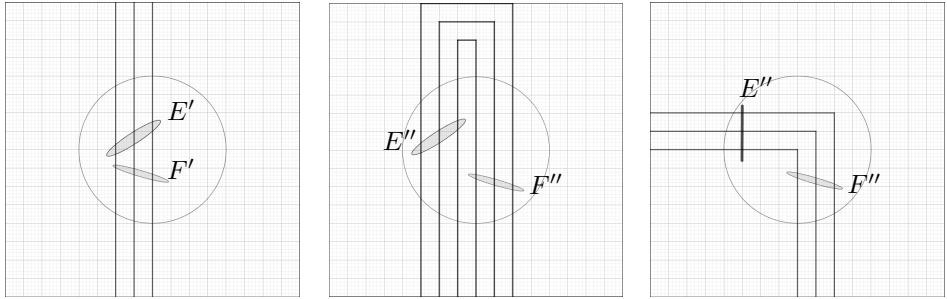
\begin{figure}
	\begin{overpic}[scale=.27]{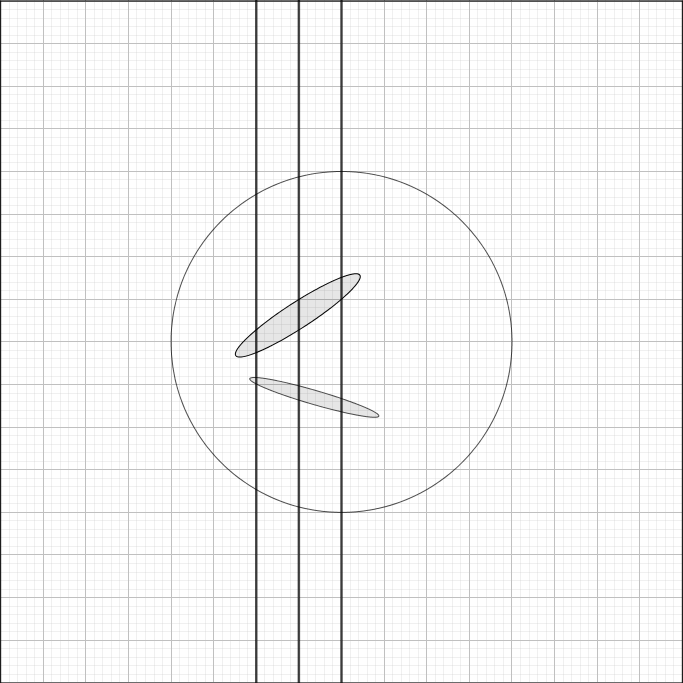}
		\put (55,60) {$E'$}
		\put (55,40) {$F'$}
	\end{overpic}
\hspace{.5em}	
	\begin{overpic}[scale=.25]{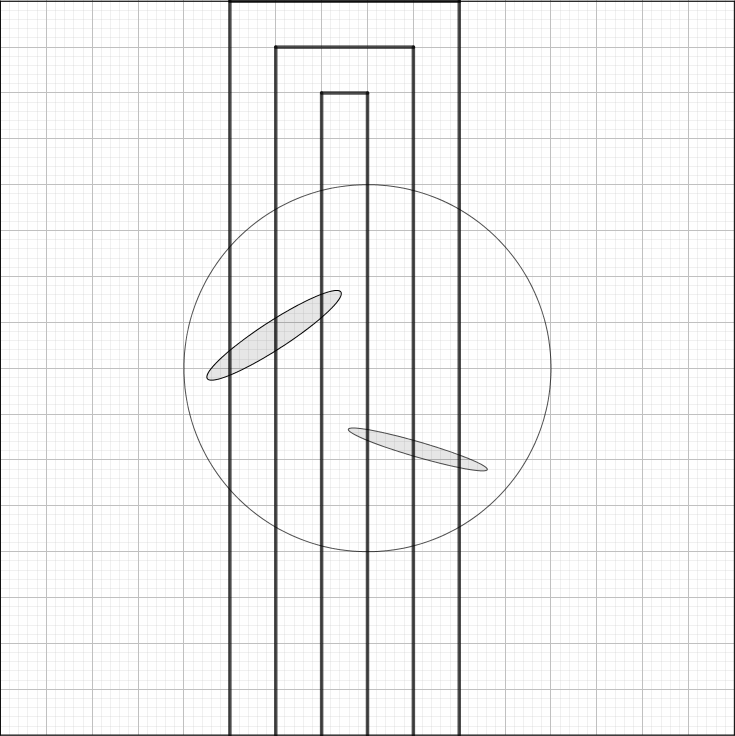}
		\put (18,50) {$E''$}
		\put (68,35) {$F''$}
	\end{overpic}
\hspace{.5em}
	\begin{overpic}[scale=.27]{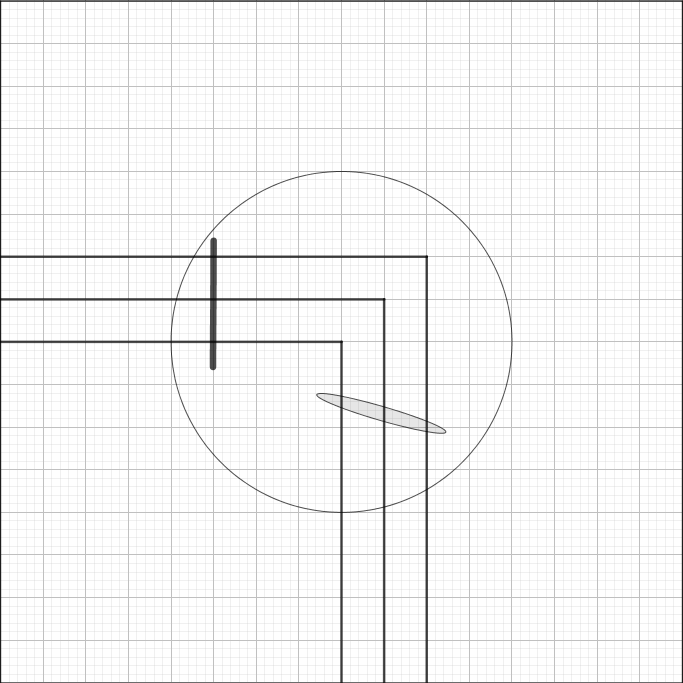}
		\put (30,68) {$E''$}
		\put (67,35) {$F''$}
	\end{overpic}
	\caption{Left: Case 1a. Middle: Case 1b. Right: Case 2.}\label{figure:grid}
\end{figure}

\smallskip
\noindent
\textit{Case 1.} Suppose that the projection of $E'$ to the $x$-axis has positive diameter equal to $c_1$ and the projection of $F'$ to the $x$-axis has diameter larger than $c_0/2$. We set $M=\min\{c_1,c_0/2\}$. 

\smallskip
\noindent
\textit{Case 1a.} Suppose that the overlap of the projections is larger than or equal to $M/2$. Then by choosing $\varepsilon< 2^{-1}M (N+3)^{-1}$ we may find $N+1$ disjoint vertical paths in $\mathcal B'$ intersecting both $E'$ and $F'$ so that their parallel translates by $\varepsilon$ to either side also have the same property. See Figure \ref{figure:grid}.

\smallskip
\noindent
\textit{Case 1b.} Suppose that the overlap of the projections is smaller than $M/2$. Then there exist subcontinua $E''$ and $F''$ of $E'$ and $F'$, respectively, such that the projections of $E''$ and $F''$ to the $x$-axis are disjoint and both have diameter $M/2$.  If we choose $\varepsilon< 2^{-1}M (N+3)^{-1}$, then there exist $N+1$ disjoint vertical paths in $\mathcal B'$ intersecting $F''$ but not $E''$ and $N+1$ disjoint vertical paths intersecting $E''$ but not $F''$ so that their parallel translates by $\varepsilon$ to either side also have the same property. We truncate these paths appropriately outside $B_{e}(0,1)$ and then connect them with disjoint horizontal segments in $\mathcal B'$ to obtain $\Pi$-shaped paths; see Figure \ref{figure:grid}. The resulting collection of paths has the desired properties. 

\smallskip
\noindent
\textit{Case 2.} Suppose that the projection of $E'$ to the $x$-axis has diameter  zero and the projection of $F'$ to the $x$-axis has diameter larger than $c_0/2$. Thus, $E'$ is a vertical line segment that projects to a segment of diameter $c_1$ in the $y$-axis. Our goal is to find appropriate subcontinua $E''$ and $F''$ of $E'$ and $F'$, respectively, whose diameters are bounded below depending on $c_0$ and $c_1$ and whose projections to both axes are disjoint. First, we consider a subcontinuum $F''$ of $F'$ with diameter $c_0/8$ whose projection to the $x$-axis is disjoint from the projection of $E'$. If the projection of $F''$ to the $y$-axis has diameter larger than or equal to $\min \{c_1/2,c_0/16\}$, then we can argue as in Case 1, replacing the $x$-axis with the $y$-axis and finding a value of $\varepsilon$ that depends on $c_0,c_1$, and $N$. Otherwise, we consider a subsegment $E''$ of $E'$ of length $c_1/4$ so that the projections of $E''$ and $F''$ to both axes are disjoint. Then by choosing a small enough $\varepsilon$, we may find $\Gamma$-shaped curves from the collection $\mathcal B'$ that join $E''$ and $F''$ and have the desired properties; see Figure \ref{figure:grid}. The remaining cases are symmetric to the ones we treated.

\smallskip

Note that $(\varphi\circ \pi)^{-1}$ is injective on the square grid $\mathcal G$ and on the boundaries $\partial U(\beta)$, $\beta\in \mathcal B'$, since these sets avoid the set $\bigcup_{i\in \N\setminus I_E}p_i'$. We pull back the collection $\mathcal B'$ and the regions $U(\beta)$, $\beta\in \mathcal B'$, to $\widehat{\C}$ under the proper and cell-like map $\varphi\circ \pi\colon \widehat{\C}\setminus \pi^{-1}(x_{\infty})\to \C$.  Using Lemma \ref{lemma:cell_like} \ref{lemma:cell_like:boundary} we obtain a collection $\mathcal B$ of simple curves and Jordan regions $U(\beta)$, $\beta\in \mathcal B$, in $\widehat{\C}$ with the following properties.
\begin{enumerate}[\upshape(i)]
	\item\label{lemma:nondegeneracy:beta} For each $\beta\in \mathcal B$ the Jordan region $U(\beta)$ contains $|\beta|$ and if $\beta_1,\beta_2\in \mathcal B$ are disjoint curves, then $U(\beta_1)$ and $U(\beta_2)$ are disjoint and intersect disjoint collections of sets $p_i$, $i\in \N\setminus I_E$. 
	\item\label{lemma:nondegeneracy:F} Whenever $F\subset \widehat{\C}\setminus E$ is a continuum with $\widehat{F}=\pi_{X;E}(F)$ and  $\diam_d(\widehat{F})\geq \delta$, there exist disjoint paths $\beta_1,\dots,\beta_{N+1}\in \mathcal B$ such that $\beta_i$ and $\partial U(\beta_i)$ intersect $E$ and $F$ for each $i\in \{1,\dots,N+1\}$.
\end{enumerate}
We let $\eta=\min\{\dist( |\beta|,\partial U(\beta)): \beta\in \mathcal B\}$, which is positive since $\mathcal B$ is a finite collection. 

Fix a continuum  $F\subset \widehat{\C}\setminus E$  and paths $\beta_1,\dots,\beta_{N+1}\in \mathcal B$ as in \ref{lemma:nondegeneracy:F}. If $J_0\subset J$ is a set with $\#J_0\leq N$ as in the statement of the lemma, then by \ref{lemma:nondegeneracy:beta} and  the pigeonhole principle there exists $k\in \{1,\dots,N+1\}$ such that $U(\beta_k)$ does not intersect $\bigcup_{i\in J_0\setminus I_E} p_i$. {Note that $F\cap U(\beta_k)$ does not intersect $\bigcup_{i\in J_0}p_i$, but this is not necessarily true for $E$, since it intersects $p_i$, whenever $i\in I_E\cap J_0$.}  Let $\psi(x)=\dist(x,|\beta_k|)$, which is a $1$-Lipschitz function on $\widehat{\C}$. For a.e.\ $t\in (0,\eta)$ the components of $\psi^{-1}(t)$ are points, Jordan curves, or Jordan arcs; see \cite[Theorem 1.5]{Ntalampekos:monotone} for a general statement in metric surfaces or \cite{Brown:distancesets} for a planar version. Hence, for a.e.\ $t\in (0,\eta)$, the set $\psi^{-1}(t)$ contains a Jordan curve separating $|\beta_k|$ from $\partial U(\beta_k)$. We fix such a $t\in (0,\eta)$. Since $E$ and $F\setminus \bigcup_{i\in J_0}p_i$ connect $\beta_k$ and $\partial U(\beta_k)$, we conclude they intersect this Jordan curve. Thus, $\psi^{-1}(t)$ contains an open curve connecting $E$ and $F\setminus \bigcup_{i\in J_0}p_i$ and avoiding $\bigcup_{i\in J_0}p_i$. It follows that $\pi_Y(\psi^{-1}(t))$ contains a curve in $\Gamma=\Gamma\left(\pi_Y({E}),\pi_Y({F} \setminus \bigcup_{i\in J_0} p_i); \mathcal E(Y)\setminus \bigcup_{i\in J_0} \pi_{Y}(p_i) \right)$. Thus, if $\rho\colon \mathcal E(Y)\to [0,\infty]$ is admissible for $\Gamma$, then
\begin{align*}
\int_{\psi^{-1}(t)\cap Y} \rho\circ \pi_Y  \, d\mathcal H^1 + \sum_{\substack{i: \psi^{-1}(t) \cap p_i\neq \emptyset \\ i\in J}} \rho(\widehat p_i) \geq 1
\end{align*} 
for a.e.\ $t\in (0,\eta)$. Integrating over $t\in (0,\eta)$, and using the fact that $\psi$ is $1$-Lipschitz, by the co-area inequality (Proposition \ref{proposition:coarea}) we obtain
\begin{align*}
C\eta&\leq  \int_{N_\eta(|\beta_k|)\cap Y} \rho \circ \pi_Y \,d\Sigma+ \sum_{i\in J} \rho(\widehat p_i)\diam(p_i)\\
&\leq    \left(\int_{Y}(\rho\circ \pi_Y)^2 \, d\Sigma + \sum_{i\in J}\rho(\widehat{p}_i)^2 \right)^{1/2}  \left( \Sigma(\widehat{\C})+ \sum_{i\in \N} \diam(p_i)^2 \right)^{1/2}.
\end{align*}
Infimizing over $\rho$ gives
\begin{align*}
\md_{Y}\Gamma \geq C'\eta^2 \left( 1+ \sum_{i\in \N} \diam(p_i)^2 \right)^{-1}.
\end{align*}
This completes the proof.
\end{proof}

The next lemma is established in \cite{Bonk:uniformization}. The \textit{relative distance} of two non-degenerate subsets $E,F$ of a metric space is defined as
$$\Delta(E,F)= \frac{\dist(E,F)}{\min\{\diam(E),\diam(F)\}}.$$

\begin{lemma}[{\cite[Proposition 8.7]{Bonk:uniformization}}]\label{lemma:bonkcollisions}
For each $\tau>0$ there exists a number $N_0>0$ and  a function $\psi\colon (0,\infty)\to (0,\infty)$ with $\lim_{t\to\infty}\psi(t)=0$ such that the following is true. Let $X=\widehat{\C}\setminus \bigcup_{i\in I} p_i$ be a $\tau$-cofat finitely connected domain.  Let $\widehat{E},\widehat F\subset \mathcal E(X)$ be disjoint sets and let $E=\pi_X^{-1}(\widehat E)$, $F=\pi_X^{-1}(\widehat F)$. If $\Delta(E,F)\geq 12$, then there exists a set $I_0\subset I$ with $\#I_0\leq N_0$ such that 
	\begin{align*}
	\md_X \Gamma\left(\widehat{E},\widehat{F}; \mathcal E(X)\setminus \bigcup_{i\in I_0} \widehat{p_i} \right) \leq \psi(\Delta(E,F)). 
	\end{align*}
\end{lemma}

The statement here is slightly different from \cite{Bonk:uniformization}, but the proof remains unchanged. We point out the main differences. First, Bonk uses continua $E,F$, while we use arbitrary sets. Note that upper modulus bounds are not affected by this generalization; it is instead that lower modulus bounds require continua so that the curve family that connects them is rich enough. Second, in \cite{Bonk:uniformization} the continua $E,F$ are chosen in $\br X$, but we choose sets $\widehat E,\widehat F$ in $\mathcal E(X)$ and then take their preimages in $\widehat{\C}$. Third, we use path families in $\mathcal E(X)$ rather than in $\widehat{\C}$. Each path in $\mathcal E(X)$ corresponds to countably many possibly non-compact paths in the domain $X$ and thus the considerations in \cite{Bonk:uniformization}, which use line integrals in $\widehat{\C}$, are applicable here as well.

\section{Uniformization of Sierpi\'nski packings}
Our goal in this section is to prove the following theorem at the heart of the paper.

\begin{theorem}\label{theorem:uniformization_full}
Let $Y=\widehat{\C}\setminus \bigcup_{i\in \N} q_i$ be a Sierpi\'nski packing such that $\{\diam(q_i)\}_{i\in \N}$ lies in $\ell^2(\N)$. Let $\tau>0$ and $\zeta_\infty,\zeta_0,\zeta_1\in Y$, and for each $n\in \N$, let $f_n$ be a conformal map from the domain $Y_n= \widehat{\C} \setminus \bigcup_{i=1}^n q_i$ onto a $\tau$-cofat domain $X_n=\widehat{\C}\setminus \bigcup_{i=1}^n p_{i,n}$  such that $f_n(\zeta_\infty)=\infty$, $f_n(\zeta_0)=0$, $|f_n(\zeta_1)|=1$, and $f_n^*(q_i)=p_{i,n}$ for each $i\in \{1,\dots,n\}$. Suppose, in addition, that for each $i\in \N$ all Hausdorff limits of the sequence $\{p_{i,n}\}_{n\geq i}$ are non-separating. Then there exists a sequence $\{k_n\}_{n\in \N}$ increasing to $\infty$ with the following properties.
\begin{enumerate}[\upshape(i)]
	\item\label{item:fatness} For each $i\in \N$, $p_{i,k_n}$ converges as $n\to\infty$  to a $\tau$-fat non-separating continuum $p_{i}$ in the Hausdorff sense.
	\item\label{item:packing} The set $X=\widehat{\C}\setminus \bigcup_{i\in \N} p_i$ is a $\tau$-cofat Sierpi\'nski packing.
	\item\label{item:nondegenerate} For each $i\in \N$, $p_i$ is non-degenerate if and only if  $q_i$ is non-degenerate.
	\item\label{item:convergence} The sequence $g_{k_n}=f_{k_n}^{-1}\colon X_{k_n}\to Y_{k_n}$ gives rise to a sequence of set functions $h_n= \pi_Y \circ g_{k_n}^* \circ \pi_X^{-1}$ from subsets of $\mathcal E(X)$ to subsets of $\mathcal E(Y)$ that converges uniformly to a packing-conformal map $h\colon \mathcal E(X)\to \mathcal E(Y)$.
\end{enumerate}
\end{theorem}

The proof of the theorem is given in Sections \ref{section:limiting_packing}--\ref{section:limiting_mapping_regularity}. We now record two immediate corollaries. If the sets $p_{i,n}$ are chosen to be geometric disks and points, and the maps $f_n$ are given by Koebe's uniformization theorem for finitely connected domains \cite{Koebe:FiniteUniformization}, then we obtain the following corollary.

\begin{corollary}\label{corollary:uniformization_disk}
Let $Y=\widehat{\C}\setminus \bigcup_{i\in \N} q_i$ be a Sierpi\'nski packing such that $\{\diam(q_i)\}_{i\in \N}$ lies in $\ell^2(\N)$. Then there exists a round Sierpi\'nski packing $X=\widehat{\C}\setminus \bigcup_{i\in \N} p_i$ and a packing-conformal map $h\colon \mathcal E(X)\to \mathcal E(Y)$. Moreover, for each $i\in \N$, the disk $p_i$ is  non-degenerate if and only if $q_i$ is non-degenerate.
\end{corollary}

More generally, one can consider a collection $\{P_i\}_{i\in \N}$ of non-degenerate $\tau$-fat sets in the plane $\C$, which is regarded as a subset of $\widehat{\C}$.  According to the Brandt--Harrington unformization theorem \cites{ Brandt:conformal, Harrington:conformal}  there exists a conformal map $f_n$ from the finitely connected domain $Y_n=\widehat{\C} \setminus \bigcup_{i=1}^n q_i$ onto a domain $X_n=\widehat{\C} \setminus \bigcup_{i=1}^n p_{i,n}$, where $p_{i,n}$ is either a point or is homothetic to $P_i$; that is, $p_{i,n}$ is the image of $P_i$ under a transformation $z\mapsto az+b$, $a>0$, $b\in \C$. By postcomposing $f_n$ with a homothetic transformation, we may assume that it satisfies the normalizations of Theorem \ref{theorem:uniformization_full}. By Lemma \ref{lemma:fat_invariant}, the sets $p_{i,n}$ are $c(\tau)$-fat. Also, note that if $p_{i,n}$ is homothetic to $P_i$, then each non-degenerate Hausdoff limit of $p_{i,n}$ as $n\to\infty$ is also homothetic to $P_i$, provided that it is not the entire sphere. Therefore, Theorem \ref{theorem:uniformization_full} gives the following corollary.

\begin{corollary}\label{corollary:uniformization_cofat_general}
Let $Y=\widehat{\C}\setminus \bigcup_{i\in \N} q_i$ be a Sierpi\'nski packing such that $\{\diam(q_i)\}_{i\in \N}$ lies in $\ell^2(\N)$. For $\tau>0$ consider a collection $\{P_i\}_{i\in \N}$ of non-degenerate $\tau$-fat sets in the plane $\C$. Then there exists a Sierpi\'nski packing $X=\widehat{\C}\setminus \bigcup_{i\in \N} p_i$, where $p_i$ is a point when $q_i$ is a point and $p_i$ is homothetic to $P_i$ when $q_i$ is non-degenerate, and there exists a packing-conformal map $h\colon \mathcal E(X)\to \mathcal E(Y)$. 
\end{corollary}

\subsection{Existence of limiting Sierpi\'nski packing}\label{section:limiting_packing}
We now initiate the proof of Theorem \ref{theorem:uniformization_full}.  Recall that the conformal map $f_n\colon Y_n\to X_n$ gives rise to a set function $f_n^*= \pi_{X_n}^{-1}\circ \widehat{f}_n \circ \pi_{Y_n}$ from the powerset of $\widehat \C$ into itself. Namely, $f_n^*=f_n$ in $Y_n$ and $f_n^*(A)=p_{i,n}$ whenever $A\subset q_{i}$. By Lemma \ref{lemma:cell_like} \ref{lemma:cell_like:continuum}, $f_n^*$ maps continua to continua. 

\begin{lemma}[Non-degeneracy]\label{lemma:nondegeneracy}
Let $\widehat{E}\subset \mathcal E(Y)$ be a continuum such that $E=\pi_{Y}^{-1}(\widehat{E})$ is non-degenerate. Then  
$$\liminf_{n\to\infty}\diam(f_n^*(E))>0.$$
\end{lemma}

\begin{proof}
Without loss of generality $\widehat{E}\neq \mathcal E(Y)$, since in that case we have $f_n^*(E)=\widehat{\C}$ for all $n\in \N$. The set $E$ is a continuum and has the property that it contains all peripheral continua of $Y$ that intersect $E$. We set $E_n'= f_n^*(E)$, which is a non-degenerate continuum by the conformality of $f_n$, and suppose that $\diam(E_n')\to 0$ along a subsequence. We wish to derive a contradiction. After passing to a further subsequence, we assume that $E_n'$ converges to a point $x_0\in \widehat{\C}$ in the Hausdorff sense. Consider the great circle through the points $0=f_n(\zeta_0)$, $f_n(\zeta_1)$, and $\infty=f_n(\zeta_{\infty})$. 
The point $x_0$ does not lie on at least one of the arcs (of the great circle) from $0$ to $f_n(\zeta_1)$, $f_n(\zeta_1)$ to $\infty$, or $\infty$ to $0$. Without loss of generality, after passing to a subsequence, suppose that $E_n'$ stays away from the arc $\alpha_n$ joining $0$ to $f_n(\zeta_1)$. That is, there exists $c>0$ such that $\dist(E_n', |\alpha_n|)\geq c$ for all $n\in \N$. Let $F_n'= \pi_{X_n}^{-1} (\pi_{X_n}(|\alpha_n|))$, i.e., the union of the arc $|\alpha_n|$, together with all the sets $p_{j,n}$, $j\in \{1,\dots,n\}$, that intersect $|\alpha_n|$. Note that $F_n'$ might be very close to $E_n'$, although they are disjoint, because they both contain the complementary components of $X_n$ that they intersect. 

The $\tau$-cofatness of $X_n$ and Lemma \ref{lemma:count} imply that there exists a constant $N_0\in \N$, depending only on $\tau$ and $c$, such that at most $N_0$ of the sets $p_{i,n}$, $i\in \{1,\dots,n\}$, with diameter larger than $c/2$  intersect the arc $|\alpha_n|$. Thus, there exists a set $J_n\subset \{1,\dots,n\}$ with $\# J_n\leq N_0$ such that 
$$\dist(E_n', F_n'\setminus \bigcup_{i\in J_n} p_{i,n} ) \geq c/2$$
for each $n\in \N$. Since $\diam(F_n' \setminus \bigcup_{i\in J_n} p_{i,n})\geq \diam(\{0,f_n(\zeta_1)\})$, we have
$$\Delta (E_n', F_n'\setminus \bigcup_{i\in J_n} p_{i,n} ) \to \infty$$
as $n\to \infty$.  Finally, we observe that the sets $E_n'$ and $F_n'\setminus  \bigcup_{i\in J_n} p_{i,n}$ are invariant under the set function $\pi_{X_n}^{-1}\circ \pi_{X_n}$ because they contain all complementary components of $X_n$ that they intersect.  By Lemma \ref{lemma:bonkcollisions} there exists a constant $N_0'\in \N$ that depends only on $\tau$ and a set $J_n'\subset \{1,\dots,n\}$ with $\#J_n'\leq N_0'$ such that
\begin{align*}
\md_{X_n} \Gamma \left( \pi_{X_n}(E_n'), \pi_{X_n} (F_n'\setminus \bigcup_{i\in J_n}p_{i,n} ) ; \mathcal E(X_n)\setminus \bigcup_{i\in J_n'} \pi_{X_n}(p_{i,n})  \right) \to 0
\end{align*}
as $n\to\infty$. Note that by the monotonicity of modulus, if we set $J_n''=J_n\cup J_n'$, then we also obtain
\begin{align*}
\md_{X_n} \Gamma \left( \pi_{X_n}(E_n'), \pi_{X_n} (F_n'\setminus \bigcup_{i\in J_n''}p_{i,n} ) ; \mathcal E(X_n)\setminus \bigcup_{i\in J_n''} \pi_{X_n}(p_{i,n})  \right) \to 0
\end{align*}
as $n\to \infty$. 

Consider the set ${F_n}=(f_n^{-1})^*(F_n')$; that is, $F_n$ is the union of $f_n^{-1}(|\alpha_n|\cap X_n)$ together with the sets $q_j$, $j\in \{1,\dots,n\}$, that intersect its closure. Note that in the space $\mathcal E(Y;E)$ (endowed with any fixed metric that induces its topology) the projection of $F_n$ has diameter uniformly bounded from below, since it connects the projections of the points $\zeta_0$ and $\zeta_1$, which do not lie on $E$. Also, note that $\#J_n''\leq N_0+N_0'$. Thus, by Lemma \ref{lemma:l2collision_final}, we have
\begin{align*}
\md_{Y_n} \Gamma \left( \pi_{Y_n}(E), \pi_{Y_n} (F_n\setminus \bigcup_{i\in J_n''}q_i ) ; \mathcal E(Y_n)\setminus \bigcup_{i\in J_n''} \pi_{Y_n}(q_i)  \right) \geq C
\end{align*}
for all $n\in \N$. The conformal invariance of transboundary modulus (Lemma \ref{lemma:transboundary_invariance}) leads to a contradiction. 
\end{proof}

\begin{corollary}\label{cor:nondegeneracy}
For each $i\in \N$, the following statements are true.
\begin{enumerate}[\upshape(i)]
	\item If $q_i$ is non-degenerate, then the sequence of sets $\{p_{i,n}\}_{n\geq i}$ does not degenerate to a point as $n\to\infty$. 
	\item The set $q_i$ is a point if and only if $p_{i,n}$ is a point for all $n\geq i$.
\end{enumerate} 
\end{corollary}
The second part of the lemma follows from the fact that $f_n$ is a conformal map. After taking a diagonal sequence, assume that for each $i\in \N$ the sequence $p_{i,n}$ converges in the Hausdorff sense as $n\to\infty$ to a compact set $p_i$ that is connected, does not separate the sphere by the assumption in Theorem \ref{theorem:uniformization_full}, and is $\tau$-fat by Lemma \ref{lemma:fat_hausdorff}. We have already established parts \ref{item:fatness} and \ref{item:nondegenerate} of Theorem \ref{theorem:uniformization_full}.

Next, we wish to show that the sets $p_i$, $i\in \N$, are pairwise disjoint. First we establish a preliminary lemma.

\begin{lemma}[No clustering]\label{lemma:noclustering}
There exists $N(\tau)>0$ such that for each $x\in \widehat{\C}$ there are at most $N(\tau)$ non-degenerate sets $p_i$, $i\in \N$, containing $x$. 
\end{lemma} 
\begin{proof}
Let $N\in \N$ and suppose that the sets $p_{i_1},\dots,p_{i_N}$ are non-degenerate and contain $x$. Fix $r>0$ be such that $p_{i_l}$ is not contained in $B(x,r)$ for each $l\in \{1,\dots,N\}$ and consider points $x_{i_l,n}\in p_{i_l,n}$ converging to $x$ as $n\to\infty$. For fixed $\delta\in (0,r)$ note that the sets $B(x_{i_l,n},r-\delta)\cap p_{i_l,n}$, $l\in \{1,\dots,N\}$, are pairwise disjoint and are contained in $B(x,r)$ for all sufficiently large $n\in \N$. Moreover, $p_{i_l,n}$ is not contained in $B(x_{i_l,n},r-\delta)$ for all sufficiently large $n\in \N$. By the $\tau$-fatness of $p_{i_l,n}$, for large $n\in \N$ we have 
$$\Sigma(B(x,r)) \geq  \sum_{l=1}^N \Sigma(B(x_{i_l,n},r-\delta)\cap p_{i_l,n}) \geq  N\tau (r-\delta)^2.$$
Letting $\delta\to 0$, gives $N\leq \tau^{-1} \Sigma(B(x,r)) r^{-2} \leq C\tau^{-1}$. 
\end{proof}

\begin{lemma}[No collisions]\label{lemma:nocollisions}
Let $\widehat{F}_1,\widehat{F}_2\subset  \mathcal E(Y)$ be disjoint continua and let $F_1=\pi_{Y}^{-1}(\widehat{F}_1)$, $F_2=\pi_{Y}^{-1}(\widehat{F}_2)$. Then 
\begin{align*}
\liminf_{n\to\infty} \dist( f_n^*(F_1),f_n^*(F_2))>0.
\end{align*} 
\end{lemma}

\begin{proof}
We may assume that $\widehat{F}_1$ and $\widehat{F}_2$ are, in addition, non-separating continua, after replacing them if necessary with larger continua that are disjoint and non-separating. Suppose that $\dist( f_n^*(F_1),f_n^*(F_2)) \to 0$ along a subsequence. Let $F_i^*$ be a Hausdorff limit of $f_n^*(F_i)$, $i=1,2$. Note that the sets $F_i$, $i=1,2$, contain all of the peripheral continua of $Y$ that they intersect. We will consider two cases.

\smallskip

\noindent
\textit{Case 1.} The set $q_i$ is non-degenerate for finitely many $i\in  \N$. Let $J\subset \N$ be the set of those indices.  The conformal map $f_n\colon Y_n\to X_n$ extends conformally to the isolated point $q_i$ whenever $i\in \{1,\dots,n\}\setminus J$. Since $F_1$ and $F_2$ are disjoint, non-separating, and contain all peripheral continua that they intersect, the open set $\Omega=\widehat{\C}\setminus (F_1\cup F_2 \cup  \bigcup_{i\in J} q_i)$ is connected and the sets $F_1,F_2$ are complementary components of $\Omega$. The three-point normalization of $f_n$ implies that, after passing to a subsequence, $f_n$ converges locally uniformly in $\widehat{\C}\setminus (F_1\cup F_2 \cup  \bigcup_{i\in J} q_i)$ to a map $f$ that is either constant or conformal. By Lemma \ref{lemma:nondegeneracy}, $f$ is non-constant, so it is conformal.  Lemma \ref{lemma:caratheodory_convergence} implies that $f^*(F_i) \supset  F_i^*$ for $i=1,2$. This leads to a contradiction, since the sets $f^*(F_1),f^*(F_2)$ are disjoint, but the sets $F_1^*, F_2^*$ are not, by assumption.

\smallskip

\noindent
\textit{Case 2.} There are infinitely many non-degenerate sets $q_i$, $i\in \N$. By Corollary \ref{cor:nondegeneracy}, $p_i$ is non-degenerate if and only if $q_i$ has this property.  Let $x\in F_1^*\cap F_2^*$. By Lemma \ref{lemma:noclustering}, there exists $N(\tau)>0$ such that at most $N(\tau)$ non-degenerate sets $p_i$, $i\in \N$, contain the point $x$. Thus, there exists $i_0\in \N$ such that $E=q_{i_0}$ is non-degenerate and $x\notin p_{i_0}$. The sets $E_n'=f_n^*(E)$ converge to $p_{i_0}$ so they have diameters uniformly bounded away from $0$ and they stay away from the point $x$. Let $\alpha_n$ be a curve in $\widehat{\C}$ joining $f_n^*(F_1)$ and $f_n^*(F_2)$ such that $|\alpha_n|$ converges to $x$ in the Hausdorff sense. Let $G_n'$ be the union of $|\alpha_n|$ together with the sets $p_{j,n}$, $j\in \{1,\dots,n\}$, that intersect $|\alpha_n|$. Note that the diameter of $G_n'$ might be large. 

We now argue as in the proof of Lemma \ref{lemma:nondegeneracy}. The cofatness of $X_n$ and Lemma \ref{lemma:count} imply that if $J_n$ is the set of indices $j$ such that $p_{j,n}\cap |\alpha_n|\neq \emptyset$ and $\diam(p_{j,n})\geq \diam(|\alpha_n|)$, then $\# J_n$ is uniformly bounded, depending only on $\tau$.  We conclude that $\Delta(E_n', G_n' \setminus \bigcup_{i\in J_n} p_{i,n} )\to \infty$ as $n\to\infty$. Lemma \ref{lemma:bonkcollisions} implies that there exists a set of natural numbers $J_n''\supset J_n$ with uniformly bounded cardinality such that
\begin{align*}
\md_{X_n} \Gamma \left( \pi_{X_n}(E_n'), \pi_{X_n} (G_n'\setminus \bigcup_{i\in J_n''}p_{i,n} ) ; \mathcal E(X_n)\setminus \bigcup_{i\in J_n''} \pi_{X_n}(p_{i,n})  \right) \to 0
\end{align*}
as $n\to \infty$. 

On the other hand, we consider the sets $E$ and $G_n=(f_n^{-1})^*(G_n')$. We note that the continuum $G_n$ joins $F_1$ and $F_2$, so its projection to $\mathcal E(Y;E)$ has diameter uniformly bounded below away from $0$, in any given metric. By Lemma \ref{lemma:l2collision_final} we have
\begin{align*}
\md_{Y_n} \Gamma \left( \pi_{Y_n}(E), \pi_{Y_n} (G_n\setminus \bigcup_{i\in J_n''}q_i ) ; \mathcal E(Y_n)\setminus \bigcup_{i\in J_n''} \pi_{Y_n}(q_i)  \right) \geq C
\end{align*}
for each $n\in \N$. This contradicts the conformal invariance of transboundary modulus. 
\end{proof}

\begin{corollary}\label{cor:nocollisions}
The sets $p_i$, $i\in \N$, are pairwise disjoint. Moreover, the set $X=\widehat{\C} \setminus \bigcup_{i=1}^\infty p_i$ is a Sierpi\'nski packing.
\end{corollary}
\begin{proof}
The sets $p_i$, $i\in \N$, are pairwise disjoint by Lemma \ref{lemma:nocollisions}. As we have discussed, each of the sets $p_i$, $i\in \N$, is a non-separating continuum and is $\tau$-fat. A trivial consequence of Lemma \ref{lemma:count} is that $\diam(p_i)\to 0$ as $i\to \infty$. These facts imply that $X$ is a Sierpi\'nski packing. 
\end{proof}
Thus, we have established part \ref{item:packing} of Theorem \ref{theorem:uniformization_full}. In the following sections we prove the existence of the limiting map of part \ref{item:convergence} and its regularity. 

\subsection{Existence of limiting map}\label{section:limiting_mapping}

Fix metrics on the spaces $\mathcal E(X),\mathcal E(Y)$ that induce their topology. The particular metrics are not of importance. We set $g_n=f_n^{-1}$ and consider the sequence of set functions 
$$h_n=\pi_Y \circ g_n^*\circ \pi_{X}^{-1} =\pi_Y \circ \pi_{Y_n}^{-1} \circ \widehat g_n \circ \pi_{X_n}  \circ \pi_{X}^{-1} ,\,\, n\in \N$$
from subsets of $\mathcal E(X)$ to subsets of $\mathcal E(Y)$.

\begin{lemma}[Equicontinuity]\label{lemma:equicontinuity}
For each $\varepsilon>0$ there exists $\delta>0$ and $N\in \N$ such that if $\widehat E'$ is a set in $\mathcal E(X)$ with $\diam(\widehat E')<\delta$,  then $\diam(h_n(\widehat E'))<\varepsilon$ for all $n>N$.
\end{lemma}
Roughly speaking, this lemma says that if a set has small diameter in $\mathcal E(X)$, then its image under $g_n^*$ projects to a set that has small diameter in $\mathcal E(Y)$. We note that the diameter need not be small if we do not project to $\mathcal E(Y)$. As an analogy, consider a conformal map $f$ from a simply connected domain $\mathcal Y$ whose boundary is not locally connected onto the unit disk $\mathcal X$. Then small sets near $\partial \mathcal X$ need not be mapped to small sets in $\mathcal Y$ under $g=f^{-1}$.

\begin{proof}
By the local connectivity of $\mathcal E(X)\simeq \widehat{\C}$, it suffices to show the statement for continua $\widehat{E}'$, rather than arbitrary sets. We argue by contradiction, assuming that there exists $\varepsilon_0>0$ and a sequence of positive integers $\{k_n\}_{n\in \N}$ with $k_n\to \infty$ such that there exists a sequence of continua $\widehat{E}_n'\subset \mathcal E(X)$ converging to a point $\widehat x \in \mathcal E(X)$ as $n\to\infty$, but $\diam( \pi_{Y}(E_n))\geq \varepsilon_0$ for all $n\in \N$, where $E_n=g_{k_n}^*( \pi_{X}^{-1}( \widehat{E}'_n))$. After passing to a subsequence, we assume that $E_n$ converges to a  continuum  $E \subset \widehat{\C}$, and by continuity $\pi_Y(E_n)$ converges to a continuum $\pi_Y(E)$ in $\mathcal E(Y)$ with $\diam(\pi_Y(E))\geq \varepsilon_0$. 

Consider distinct points $\widehat y,\widehat z\in \pi_Y(E) \setminus \bigcup_{i\in \N}\pi_Y(q_i)$ and let $y =\pi_Y^{-1}(\widehat y)$,  $z=\pi_Y^{-1}(z)\in E\cap Y$. Consider disjoint closed Jordan regions $U_y$, $U_z$ containing $y,z$ in their interior, whose boundaries lie in $Y$; these regions can be obtained as preimages under $\pi_Y$ of appropriate Jordan regions in $\mathcal E(Y)$. By Lemma \ref{lemma:nocollisions}, the distance of the images $f_{k_n}^*(U_y)$,  $f_{k_n}^*(U_z)$ is uniformly bounded away from zero.  By the Hausdorff convergence of $E_n$ to $E$, for all sufficiently large $n\in \N$ the set  $E_n$ intersects both $U_y$ and $U_z$. Since $y,z\in Y$, there exist points $y_n,z_n\in E_n\cap Y$ converging to $y,z$, respectively. The points $y_n'=f_{k_n}(y_n), z_n'=f_{k_n}(z_n)$ lie in  $\pi_{X}^{-1}( \widehat E_n')$ and stay away from each other as $n\to\infty$. Since the points $\pi_X(y_n')$ and $\pi_X(z_n')$ converge to the point $\widehat{x}$, but the points $y_n'$ and $z_n'$ stay away from each other, by Lemma \ref{lemma:distance_convergence} we must have $\widehat{x}=\pi_X(p_{i_0})$ for some $i_0\in \N$. We conclude that $y_n',z_n'$ accumulate at $p_{i_0}$ as $n\to\infty$. 

Fix a closed Jordan region $U$ containing $y$ in its interior such that $U$ is disjoint from $q_{i_0}$ and $\partial U\subset Y$; this is possible because $y\in Y$. We have $y_n\in U$ for all sufficiently large $n\in \N$. On the other hand, the sets $f_{k_n}^*(U)$ and $f_{k_n}^*(q_{i_0})=p_{i_0,k_n}$ come arbitrarily close to each other. This contradicts Lemma \ref{lemma:nocollisions}.
\end{proof}

The following statement is a version of the Arzel\`a--Ascoli theorem for equicontinuous families of set functions. The proof is a straightforward adaptation of the classical argument and the experienced reader may safely skip it.

\begin{lemma}[Compactness]\label{lemma:arzela_ascoli}
Let $\mathcal X, \mathcal Y$ be compact metric spaces and $ \mathpzc{h}_n\colon \mathcal P(\mathcal X)\to \mathcal P(\mathcal Y)$, $n\in \N$, be a sequence of set functions with the following properties:
\begin{enumerate}[\upshape(i)]
	\item Surjectivity: for each $n\in \N$, $\mathpzc h_n( \mathcal X)= \mathcal Y$.
	\item Setwise monotonicity: for each $n\in \N$, if $A,B\subset \mathcal X$ and $A\subset B$, then $\mathpzc h_n(A)\subset \mathpzc h_n(B)$. 
	\item Inverse image property: for each $n\in \N$, if $A\subset \mathcal X$ and $y\in \mathpzc h_n(A)$, then there exists $x\in A$ such that $y\in \mathpzc h_n(x)$.  
	\item  Equicontinuity: for each $\varepsilon>0$ there exists $\delta>0$ and $N\in \N$ such that for each set $E\subset \mathcal X$ with $\diam(E)<\delta$ we have $\diam(\mathpzc h_n(E))<\varepsilon$ for all $n>N$.
\end{enumerate}
Then there exists a subsequence $\{\mathpzc h_{k_n}\}_{n\in \N}$ of $\{\mathpzc h_n\}_{n\in \N}$ that converges uniformly to a continuous and surjective map $\mathpzc h\colon \mathcal X\to \mathcal Y$ in the following sense: for each $\varepsilon>0$ there exists $N\in \N$ such that for each set $E\subset \mathcal X$ we have
$$d_H(\mathpzc h_{k_n}(E), \mathpzc h(E) )<\varepsilon$$
for all $n>N$.
\end{lemma} 
\begin{proof}
Note that the space of compact subsets of a compact metric space is compact in the Hausdorff metric \cite[Theorem 7.3.8, p~253]{BuragoBuragoIvanov:metric}. Consider a countable dense set $\{x_l\}_{l\in \N}$ in $\mathcal X$ and a diagonal subsequence $\mathpzc h_{k_n}$ such that for each $l\in \N$ the sequence of sets ${\mathpzc h_{k_n}(x_l)}$, $n\in \N$, converges, and thus is a  Cauchy sequence in the Hausdorff metric. For simplicity, we denote $\mathpzc h_{k_n}$ by $\mathpzc h_n$. 

We fix $\varepsilon>0$. For each $l\in \N$ and $\varepsilon>0$ there exists $N(\varepsilon,l)>0$ such that for all $n,m> N(\varepsilon,l)$ we have
$$ d_H(\mathpzc h_n(y_l), \mathpzc h_m(y_l)) <\varepsilon/3.$$
By equicontinuity, there exists $\delta>0$ and $N\in \N$ such that if $E\subset \mathcal X$ is a set with $\diam(E)<\delta$, then  $\diam(\mathpzc h_n(E))<\varepsilon/3$ for $n>N$. By the compactness of $\mathcal X$, we may find $M\in \N$ such that every point $x\in \mathcal X$ is within distance $\delta$ from the set $\{x_1,\dots,x_M\}$. We define $N_0=\max\{N, N(\varepsilon,1),\dots,N(\varepsilon,M)\}$.

Let $x\in \mathcal X$ be arbitrary and consider $l\in \{1,\dots,M\}$ such that $d(x,x_{l})<\delta$. Let $E=\{x,x_l\}$, so $\diam(\mathpzc h_n(E))<\varepsilon/3$ for all $n>N$.   For $n,m>N_0$  we now have
\begin{align*}
d_H(\mathpzc h_n(x),\mathpzc h_m(x)) &\leq d_H( \mathpzc h_n(x),\mathpzc h_n(x_l))+ d_H(\mathpzc h_n(x_l),\mathpzc h_m( x_l)) +d_H(\mathpzc h_m(x_l),\mathpzc h_m(x))\\
&\leq \diam(\mathpzc h_n(E)) + \varepsilon/3 +\diam(\mathpzc h_m(E))<\varepsilon.
\end{align*}
Here we used the monotonicity property of $\mathpzc h_n$, which implies that $\mathpzc h_n(x)\cup \mathpzc h_n(x_l)\subset \mathpzc h_n(E)$. Since the space of compact subsets of $\mathcal X$ is complete with the Hausdorff metric, we conclude that $\mathpzc h_n(x)$ (as well as, its closure) converges in the Hausdorff sense to a compact subset of $\mathcal X$. By the equicontinuity condition, we have $\diam(\mathpzc h_n(x))\to 0$, so the limit has to be a point $\mathpzc h(x)$. 

Summarizing, we have shown that for each $\varepsilon>0$ there exists  $N\in \N$ such that for each $x\in \mathcal X$ we have $d_H(\mathpzc h_n(x),\mathpzc h(x))<\varepsilon$ for all $n>N$.  The equicontinuity of $\mathpzc h_n$ also implies that $\mathpzc h$ is continuous.  

For the surjectivity, we use the surjectivity of $\mathpzc h_n$ and the inverse image property. Note that for each $y\in \mathcal Y=\mathpzc h_n(\mathcal X)$ there exists $\widetilde x_n\in \mathcal X$ such that $y\in \mathpzc h_n(\widetilde x_n)$. By the uniform convergence, we have $d_H (\mathpzc h_n(\widetilde x_n), \mathpzc h(\widetilde x_n))\to 0$ as $n\to\infty$. After passing to a subsequence, we assume that $\widetilde x_n\to x\in \mathcal X$. Thus, $d_H (\mathpzc h_n(\widetilde x_n), \mathpzc h(x))\to 0$. Since $y\in \mathpzc h_n(\widetilde x_n)$, we have $\mathpzc h(x)=y$. 

Finally, we prove the uniform convergence for images of sets as in the end of the statement of the lemma. Let $\varepsilon>0$ and $N\in \N$ be such that for each $x\in \mathcal X$ we have $d_H(\mathpzc h_n(x),\mathpzc h(x))<\varepsilon$ for all $n>N$. Let $E\subset \mathcal X$ be any set and fix $n>N$. We will show that $d_H(\mathpzc h_n(E),\mathpzc h(E))<\varepsilon$. First, we show that $\mathpzc h_{n}(E)\subset N_{\varepsilon}(\mathpzc h(E))$. Let $y\in \mathpzc h_{n}(E)$, so by the inverse image property there exists $\widetilde x_n\in E$ such that $y\in \mathpzc h_n(\widetilde x_n) \subset B( \mathpzc h(\widetilde x_n),\varepsilon) \subset N_{\varepsilon}(\mathpzc h(E))$. Conversely, we will show that $\mathpzc h(E)\subset N_{\varepsilon}(\mathpzc h_n(E))$. Let $y\in \mathpzc h(E)$. By the surjectivity of $\mathpzc h$, there exists $x\in E$ such that $y=\mathpzc h(x)\subset N_{\varepsilon}(\mathpzc h_n(x)) \subset N_{\varepsilon}(\mathpzc h_n(E))$; here we also used the setwise monotonicity.   
\end{proof}

\begin{lemma}\label{lemma:set_properties}
The sequence of set functions $h_n=\pi_Y\circ g_n^*\circ \pi_X^{-1}$, $n\in \N$,  satisfies the assumptions of Lemma \ref{lemma:arzela_ascoli}. In particular, after passing to a subsequence, $h_n$ converges uniformly as $n\to\infty$ to a continuous and surjective map $h\colon \mathcal E(X)\to \mathcal E(Y)$. 
\end{lemma}
\begin{proof}
Recall that $h_n= \pi_Y\circ \pi_{Y_n}^{-1} \circ \widehat{g}_n \circ \pi_{X_n}\circ \pi_X^{-1}$. The surjectivity is immediate. The equicontinuity follows from Lemma \ref{lemma:equicontinuity}. In general, if $\phi$ is a \textit{function} between any sets, then the induced \textit{set functions} $\phi$ and $\phi^{-1}$ have the monotonicity and inverse image properties. Also, note that if $\phi$ and $\psi$ are \textit{set functions} that have the monotonicity and inverse image properties, then $\phi\circ \psi$ also does so. This shows that $h_n$ has these properties.
\end{proof}

\begin{lemma}[Topological properties]\label{lemma:topological_properties}
The map $h\colon \mathcal E(X)\to \mathcal E(Y)$ is monotone and $h(\widehat{p}_i)=\widehat q_i$ for each $i\in \N$. 
\end{lemma}
\begin{proof}
For the monotonicity we argue by contradiction. Let $\widehat{y}\in \mathcal E(Y)$ and suppose that $h^{-1}(\widehat{y})$ is disconnected. Let $\widehat{x},\widehat{z}$ be points lying on distinct components of $h^{-1}(\widehat{y})$ and let $\widehat E'$ be a continuum in $\mathcal E(X)\setminus h^{-1}(\widehat{y})$ that separates them. Consider the continuum $E'=\pi_X^{-1}( \widehat{E}')$, which separates the sets $x=\pi_X^{-1}(\widehat{x})$ and $z=\pi_X^{-1}(\widehat{z})$. 

In $\mathcal E(Y)$ consider the continuum $\widehat{E}=h(\widehat{E}')$, which is disjoint from $\widehat{y}$. By the uniform convergence, the sets $h_n(\widehat{E}')$ converge to $\widehat E$ in the Hausdorff sense, so there exists a closed Jordan region $\widehat{U}$, disjoint from $\widehat{y}$, containing $h_n(\widehat{E}')$ for all sufficiently large $n\in \N$. Since $h_n(\widehat{x})$ and $h_n(\widehat{z})$ converge to $\widehat{y}$, there exists a closed Jordan region $\widehat{V}$, disjoint from $\widehat{U}$, containing these sequences for all large $n\in \N$. 

Let $U,V$ be the preimages of $\widehat{U}, \widehat{V}$ under $\pi_Y$, respectively. So, $U,V$ are disjoint continua in $\widehat{\C}$. Using \eqref{inclusion:fgA}, we have
$$f_n^*(U)\supset f_n^*( \pi_Y^{-1}( h_n(\widehat E') ))= f_n^*( \pi_Y^{-1}( \pi_Y (g_n^*( E'))))\supset f_n^* (g_n^*(E')) \supset E'$$
and similarly $f_n^*(V)$ contains $x,z$ for all large $n\in \N$. Since $E'$ separates $x,z$ and $f_n^*(V)$ joins them, we conclude that $f_n^*(U)\cap f_n^*(V)\neq \emptyset$ for all large $n\in \N$.   However, by Lemma \ref{lemma:nocollisions}, $f_n^*(U)$ and $f_n^*(V)$ have distance uniformly bounded below away from $0$ as $n\to\infty$. This is a contradiction. Therefore $h^{-1}(\widehat{y})$ is connected.

Next we show that $h(\widehat{p}_i)=\widehat q_i$ for each $i\in \N$. Recall that $p_{i,n}\to p_i$ as $n\to\infty$. This implies that the sequence of sets $E_n=\pi_X(p_{i,n})$ converges to the point $\pi_{X}(p_i)=\widehat{p}_i$. By uniform convergence, $h(\widehat{p}_i)$ is precisely the limit of $h_n(E_n)$. Note that $\pi_{X}^{-1}(E_n)$ contains $p_{i,n}$. Thus, $g_n^* (\pi_X^{-1}(E_n))$ contains $q_i$ and $h_n(E_n)$ contains $\widehat{q}_i$ for all sufficiently large $n\in \N$. It follows that the limit of $h_n(E_n)$, which is the point $h(\widehat{p}_i)$, is precisely equal to $\widehat q_i$.  
\end{proof}

This completes the proof of the existence and of the topological properties of the limiting map $h\colon \mathcal E(X)\to \mathcal E(Y)$.  In the next section we establish the analytic properties of $h$ as required in the definition of a packing-conformal map, completing the proof of Theorem \ref{theorem:uniformization_full}.

\subsection{Regularity of limiting map}\label{section:limiting_mapping_regularity}

For each conformal map $g_n\colon X_n\to Y_n$ consider the derivative (with respect to the spherical metrics) $|Dg_n|$, which satisfies the following relations by Lemma \ref{lemma:packing_conformal}. First, we have the transboundary upper gradient inequality:
\begin{align}\label{regularity:upper_gradinet_gn}
\dist(g_n^*(\gamma(a)),g_n^*(\gamma(b))) \leq \int_{\gamma}|Dg_n|\, ds+ \sum_{i:p_{i,n}\cap |\gamma|\neq \emptyset} \diam(q_i)
\end{align}
for all locally rectifiable paths $\gamma\colon [a,b] \to \widehat{\C}$. Second, for all Borel sets $E\subset Y_n$ we have
\begin{align}\label{regularity:conformality_gn}
\int_{g_n^{-1}(E)} |Dg_n|^2 \,  d\Sigma = \Sigma (E).
\end{align}
We extend $|Dg_n|$ to $\widehat{\C}$ by setting it to be zero in $\bigcup_{i=1}^n p_{i,n}$. We first establish a preliminary result. 

\begin{lemma}\label{lemma:limit_sum}
Let $\{\lambda_i\}_{i\in \N}$ be a non-negative sequence in $\ell^2(\N)$. Then
\begin{align*}
\limsup_{n\to\infty} \sum_{i:p_{i,n}\cap |\gamma|\neq \emptyset} \lambda_i \leq \sum_{i:p_i\cap |\gamma|\neq \emptyset} \lambda_i.
\end{align*}
for all compact curves $\gamma$ in $\widehat \C$ outside a curve family $\Gamma_0$ with $\md_2\Gamma_0=0$. 
\end{lemma}
The proof relies on the fact that for each $n\in \N$ the sets $p_{i,n}$, $i\in \{1,\dots,n\}$, are pairwise disjoint and $\tau$-fat, and they converge in the Hausdorff sense as $n\to \infty$ to the sets $p_i$, $i\in \N$, which are also pairwise disjoint and $\tau$-fat.

\begin{proof}
Suppose that $\gamma$ is a compact path so $|\gamma|$ is compact. For each $i\in \N$, if $p_i\cap |\gamma|=\emptyset$, then $p_i$ has a positive distance from $|\gamma|$ by compactness. Thus, $p_{i,n}\cap |\gamma|=\emptyset$ for all sufficiently large $n\in \N$. It follows that for each $M\in \N$ we have
\begin{align}\label{lemma:limit_sum_M}
\limsup_{n\to\infty} \sum_{\substack{i\in \{1,\dots,M\} \\p_{i,n}\cap |\gamma|\neq \emptyset}} \lambda_i \leq \sum_{i:p_i\cap |\gamma|\neq \emptyset} \lambda_i.
\end{align}
We will show that there exists a curve family $\Gamma_0$ with $\md_2\Gamma_0=0$ such that if $\gamma$ is path outside $\Gamma_0$, then 
\begin{align}\label{lemma:limit_sum:claim}
\lim_{M,n\to\infty}  \sum_{\substack{i>M \\p_{i,n}\cap |\gamma|\neq \emptyset}} \lambda_i =0.
\end{align}
Combined with \eqref{lemma:limit_sum_M}, this gives the desired conclusion.

Let $J$ be the set of indices $i\in \N$ such that $\diam(p_i)>0$. By Corollary \ref{cor:nondegeneracy}, if $i\in J$ then $p_{i,n}$ is non-degenerate for all $n\geq i$, and if $i\notin J$ then $p_{i,n}$ is a point for all $n\geq i$. For each $n\in \N$ and $i\in \{1,\dots,n\}\cap J$ consider points $x_{i,n}\in p_{i,n}$ and $x_i\in p_i$ such that $x_{i,n}\to x_i$ as $n\to\infty$. We let $B_{i,n}=B(x_{i,n},\diam(p_{i,n}))$ and $B_{i}=B(x_i, \diam(p_i))$. Note that $\Sigma(2B_{i,n}) \leq c(\tau) \Sigma(p_{i,n})$  and  $\Sigma(2B_i)\leq c(\tau)\Sigma(p_i)$ by the $\tau$-fatness, and observe that for each $i\in J$ the characteristic functions $\x_{2B_{i,n}}$ converge as $n\to\infty $ pointwise a.e.\ to $\x_{2B_i}$.  We define the functions
\begin{align*}
\phi(x)&= \sum_{i\in J}  \frac{\lambda_i}{\diam(p_{i})} \x_{2B_{i}}(x),\\
\psi_{i,n}(x)&= \left|\frac{\lambda_i\x_{\{1,\dots,n\}}(i)}{\diam(p_{i,n})} \x_{2B_{i,n}}(x) -\frac{\lambda_i}{\diam(p_{i})} \x_{2B_{i}}(x) \right|,\,\,\, i\in J,\,\, n\in \N,\\
\phi_{n}(x)&= \sum_{i\in J} \psi_{i,n}(x),\,\,\,  n\in \N.
\end{align*}
Using Lemma \ref{lemma:bojarski} and the fact that the sets $p_i$, $i\in  J$, are pairwise disjoint and $\tau$-fat, we obtain
\begin{align*}
\int  \phi^2\leq c(\tau) \int  \sum_{i\in J} \frac{\lambda_i^2}{\diam(p_{i})^2} \x_{p_i}\leq c'(\tau) \sum_{i=1}^\infty \lambda_i^2
\end{align*}
and similarly, for each $n,M\in \N$ we have
\begin{align}\label{lemma:limit_sum_psi}
\int \left (\sum_{i>M, i\in J} \psi_{i,n}(x)\right)^2 \leq c(\tau) \sum_{i>M} \lambda_i^2. 
\end{align}
Note that for each $i\in J$ the sequence of functions $\{\psi_{i,n}\}_{n\in \N}$ is uniformly bounded and converges to $0$ pointwise a.e. The dominated convergence theorem implies that for each fixed $M\in \N$, the sequence $\{\sum_{i\leq M, i\in J}\psi_{i,n}\}_{n\in \N}$ converges to $0$ in $L^2(\widehat \C)$. This fact, combined with \eqref{lemma:limit_sum_psi} gives that $\{\phi_n\}_{n\in \N}$ converges to $0$ in $L^2(\widehat \C)$. 

By Fuglede's lemma \cite[p.~131]{HeinonenKoskelaShanmugalingamTyson:Sobolev}, there exists a curve family $\Gamma_1$ with $\md_2\Gamma_1=0$ such that for all curves $\gamma$ outside $\Gamma_1$ we have  
\begin{align*}
\lim_{n\to\infty }\int_{\gamma} \phi_n \, ds=0.
\end{align*}
Moreover, since $\phi\in L^2(\widehat{\C})$, there exists a curve family $\Gamma_2$ with $\md_2\Gamma_2=0$ such that $\int_{\gamma}\phi\, ds<\infty $ for each $\gamma\notin \Gamma_2$. Finally, there exists a family $\Gamma_3$ with $\md_2\Gamma_3=0$ that contains all non-constant curves intersecting the countable collection of points $p_i$, $i\notin J$, and $p_{i,n}$, $i\notin J$, $n\geq i$; see \cite[\S 7.9, p.~23]{Vaisala:quasiconformal}.  We define $\Gamma_0=\Gamma_1\cup \Gamma_2\cup \Gamma_3$, which satisfies $\md_2\Gamma_0=0$, and fix a curve $\gamma\notin \Gamma_0$. Note that if $\gamma$ is a constant curve, then the claim \eqref{lemma:limit_sum:claim} holds trivially, since the sum contains at most one term and $\lambda_i\to 0$ as $i\to\infty$. Thus, we assume that $\gamma$ is non-constant.

Since $\gamma\notin \Gamma_1$, for each $\varepsilon>0$ there exists $N_0\in \N$ such that
\begin{align}\label{lemma:limit_sum_epsilon}
 \int_{\gamma }\sum_{i\in J}\left|\frac{\lambda_i\x_{\{1,\dots,n\}}(i)}{\diam(p_{i,n})} \x_{2B_{i,n}}(x) -\frac{\lambda_i}{\diam(p_{i})} \x_{2B_{i}} \right|\, ds <\varepsilon
\end{align}
for all $n>N_0$. By Lemma \ref{lemma:count}, there exists a number $N_1\in \N$, depending only on $\tau$, such that for each $n\in \N$ there exists a set $I_n\subset \N$ with $\# I_n\leq N_1$ that contains precisely the indices $i\in \N$ with the property that $p_{i,n}\cap |\gamma|\neq \emptyset $ and $\diam(2B_{i,n})\geq \diam(|\gamma|)$. Hence, $\diam(|\gamma|) >\diam(2B_{i,n})$ for each $i\in \{1,\dots,n\}\cap J\setminus I_n$ with $p_{i,n}\cap |\gamma|\neq \emptyset$.  By the properties of the set $I_n$, if $p_{i,n}\cap |\gamma|\neq \emptyset$ and $i\in \{1,\dots,n\}\cap J\setminus I_n$, then 
$$\lambda_i \leq \int_{\gamma} \frac{\lambda_i}{\diam(p_{i,n})} \x_{2B_{i,n}} \, ds.$$
Thus, for each $M\in \N$ and $n>N_0$, using \eqref{lemma:limit_sum_epsilon}, we have
\begin{align*}
\sum_{\substack{i>M,\,\, i\in J\setminus I_n \\ p_{i,n}\cap |\gamma|\neq \emptyset}} \lambda_i &\leq  \int_{\gamma}\sum_{\substack{i>M,\,\, i\in J\setminus  I_n \\ p_{i,n}\cap |\gamma|\neq \emptyset}}\frac{\lambda_i\x_{\{1,\dots,n\}}(i)}{\diam(p_{i,n})} \x_{2B_{i,n}} \, ds   \\
&< \varepsilon+ \int_{\gamma}\sum_{i>M, i\in J} \frac{\lambda_i}{\diam(p_{i})} \x_{2B_{i}} \, ds.
\end{align*}  
Since $\gamma\notin \Gamma_2$, we have $\int_{\gamma}\phi\, ds<\infty$; thus, by the dominated convergence theorem, for all sufficiently large $M\in \N$ the latter line integral term in the above inequalities is less than $\varepsilon$. Therefore, for all sufficiently large $M\in \N$ and for $n>N_0$ we have
\begin{align*}
\sum_{\substack{i>M,\,\, i\in J\setminus I_n \\ p_{i,n}\cap |\gamma|\neq \emptyset}} \lambda_i < 2\varepsilon.
\end{align*} 
Since $\lambda_i\to 0$ as $i\to \infty$, we also have
\begin{align*}
\sum_{i>M,\,\, i\in I_n} \lambda_i\leq N_1\cdot \max \{ \lambda_i:i>M\}< \varepsilon
\end{align*}
for all sufficiently large $M\in \N$. Altogether,
\begin{align*}
\sum_{\substack{i>M, i\in J \\ p_{i,n}\cap |\gamma|\neq \emptyset}} \lambda_i < 3\varepsilon
\end{align*} 
for all sufficiently large $M\in \N$ and $n>N_0$. Finally, since $\gamma\notin \Gamma_3$, we have $i\in J$ whenever $p_{i,n}\cap |\gamma|\neq \emptyset$. This implies that we may remove the restriction $i\in J$ in the summation range of the latter sum. This completes the proof of \eqref{lemma:limit_sum:claim}.  
\end{proof}

\begin{lemma}[Upper gradient]\label{lemma:upper_gradient}
The sequence $\{|Dg_n|\}_{n\in \N}$ has a subsequence that converges weakly in $L^2(\widehat{\C})$ to a function $\rho_h$ with the property that 
\begin{align*}
\dist( \pi_Y^{-1} \circ h \circ \pi_X(\gamma(a)), \pi_Y^{-1}\circ h\circ \pi_X(\gamma(b)))\leq \int_{\gamma}\rho_h\, ds + \sum_{i:p_i\cap |\gamma|\neq \emptyset} \diam(q_i)
\end{align*}
for all curves $\gamma\colon [a,b]\to \widehat{\C}$ outside a curve family $\Gamma_0$ with $\md_2\Gamma_0=0$. 
\end{lemma}

\begin{proof}
Recall that $h_n=\pi_Y\circ g_n^*\circ \pi_X^{-1}$, so for each set $A\subset \widehat{\C}$ we have
\begin{align}\label{lemma:upper_gradient_inclusion}
\pi_Y^{-1}\circ h_n\circ \pi_X(A) \supset g_n^*(A). 
\end{align}
Let $\gamma\colon[a,b]\to \widehat{\C}$ be an arbitrary locally rectifiable path. If we set $\alpha=\pi_X(\gamma(a))$ and $\beta=\pi_X(\gamma(b))$, the transboundary upper gradient inequality of $g_n$, as stated in \eqref{regularity:upper_gradinet_gn}, and \eqref{lemma:upper_gradient_inclusion} imply that 
\begin{align}\label{lemma:upper_gradient_ineq}
\dist( \pi_Y^{-1}\circ h_n (\alpha),\pi_Y^{-1}\circ h_n (\beta) ) \leq \int_{\gamma}|Dg_n|\, ds+ \sum_{i:p_{i,n}\cap |\gamma|\neq \emptyset} \diam(q_i).
\end{align}
Our goal is to show that we can take limits in this expression and derive the claimed upper gradient inequality of $h$ for all curves $\gamma$ outside an exceptional family $\Gamma_0$ with $\md_2\Gamma_0=0$.  

First, we treat the line integral terms. By \eqref{regularity:conformality_gn}, the sequence $|Dg_n|$ is uniformly bounded in $L^2(\widehat{\C})$. Consider a weak limit $\rho_h\in L^2(\widehat{\C})$ of $|Dg_n|$, given by the Banach--Alaoglu theorem \cite[Theorem 2.4.1]{HeinonenKoskelaShanmugalingamTyson:Sobolev}. By Mazur's lemma \cite[p.~19]{HeinonenKoskelaShanmugalingamTyson:Sobolev}, there exist convex combinations of $|Dg_n|$ that converge strongly in $L^2(\widehat \C)$ to $\rho_h$. Specifically, these convex combinations have the form
\begin{align*}
\rho_n= \sum_{i=n}^{M_n} \lambda_{i,n} |Dg_i|
\end{align*}
for some  $M_n>n$ and $0\leq \lambda_{i,n}\leq 1$, $i\in \{n,\dots,M_n\}$, where $\sum_{i=n}^{M_n}\lambda_{i,n}=1$. By Fuglede's lemma \cite[p.~131]{HeinonenKoskelaShanmugalingamTyson:Sobolev}, there exists a curve family $\Gamma_1$ with $\md_2\Gamma_1=0$ such that for all curves $\gamma$ outside $\Gamma_1$ we have
\begin{align}\label{lemma:upper_gradient_integral}
\lim_{n\to\infty}\int_{\gamma}\rho_n\,ds=\int_{\gamma}\rho_h\, ds.
\end{align}

Next, we treat the left-hand side of \eqref{lemma:upper_gradient_ineq}. Observe that as $n\to\infty$, the sets $h_n(\alpha)$ and $h_n(\beta)$ converge to the points $h(\alpha)$ and $h(\beta)$ respectively; recall Lemma \ref{lemma:set_properties}. By Lemma \ref{lemma:distance_convergence} \ref{lemma:distance:ii}, this implies that
\begin{align*}
\dist( \pi_Y^{-1}\circ h(\alpha) , \pi_Y^{-1}\circ h(\beta) ) \leq \liminf_{n\to\infty} \dist( \pi_Y^{-1}\circ h_n (\alpha),\pi_Y^{-1}\circ h_n (\beta) ).
\end{align*}
Note that if we set $d_n=\dist( \pi_Y^{-1}\circ h_n (\alpha),\pi_Y^{-1}\circ h_n (\beta) )$ and consider the convex combinations
$$\widetilde d_n= \sum_{i=n}^{M_n} \lambda_{i,n}d_i $$
then we obtain immediately
\begin{align}\label{lemma:upper_gradient_dist}
\dist( \pi_Y^{-1}\circ h(\alpha) , \pi_Y^{-1}\circ h(\beta) )\leq \liminf_{n\to\infty}\widetilde d_n. 
\end{align}

Finally, we treat the summation term in \eqref{lemma:upper_gradient_ineq}. We apply Lemma \ref{lemma:limit_sum} to conclude that there exists a curve family $\Gamma_2$ with $\md_2\Gamma_2=0$ such that for all compact curves $\gamma$ outside $\Gamma_2$ we have
\begin{align*}
\limsup_{n\to\infty} \sum_{i:p_{i,n}\cap |\gamma|\neq \emptyset} \diam(q_i) \leq \sum_{i:p_i\cap |\gamma|\neq \emptyset} \diam(q_i).
\end{align*}
We set 
\begin{align*}
s_n= \sum_{i:p_{i,n}\cap |\gamma|\neq \emptyset} \diam(q_i) \quad \textrm{and} \quad \widetilde s_n= \sum_{i=n}^{M_n}\lambda_{i,n}s_i 
\end{align*}
and observe that
\begin{align}\label{lemma:upper_gradient_sums}
\limsup_{n\to\infty}\widetilde s_n \leq \sum_{i:p_i\cap |\gamma|\neq \emptyset} \diam(q_i).
\end{align}

We now define $\Gamma_0=\Gamma_1\cup \Gamma_2$ and let $\gamma\notin \Gamma_0$ be a compact curve. Taking convex combinations in \eqref{lemma:upper_gradient_ineq}, we obtain
\begin{align*}
\widetilde d_n \leq \int_{\gamma} \rho_n\, ds +\widetilde s_n.
\end{align*}
By \eqref{lemma:upper_gradient_integral}, \eqref{lemma:upper_gradient_dist}, and \eqref{lemma:upper_gradient_sums}, we can take limits to obtain  
\begin{align*}
\dist( \pi_Y^{-1}\circ h(\alpha) , \pi_Y^{-1}\circ h(\beta) )\leq \int_{\gamma}\rho_h\, ds+ \sum_{i:p_i\cap |\gamma|\neq \emptyset} \diam(q_i).
\end{align*}
This completes the proof.
\end{proof}

\begin{lemma}[Conformality]\label{lemma:conformality}
For all Borel sets $E\subset  Y$ we have
$$\int_{\pi_X^{-1}(h^{-1}(\pi_Y(E)))} \rho_h^2\, d\Sigma\leq \Sigma(E).$$
\end{lemma}
\begin{proof}
Let $E\subset Y$ be a Borel set. Since $\pi_Y$ is injective on $Y$, by the Lusin--Souslin theorem \cite[Theorem 15.1, p.~89]{Kechris:descriptive}  $\pi_Y(E)$ is also a Borel set. By continuity, $\pi_X^{-1}(h^{-1}(\pi_Y(E)))$ is a Borel set. Let $K$ be a compact subset of $\pi_X^{-1}(h^{-1}(\pi_Y(E)))$. Since $\pi_Y$ is injective on $E$, we have $\pi^{-1}_Y(h(\pi_X(K)))\subset E$. We will show that
\begin{align*}
\int_K \rho_h^2 \, d\Sigma \leq \Sigma( \pi^{-1}_Y(h(\pi_X(K))) ) \leq \Sigma(E).
\end{align*}
Since $K$ is an arbitrary compact subset of $\pi_X^{-1}(h^{-1}(\pi_Y(E)))$, the proof will be completed by the inner regularity of $\Sigma$. 

Let $\widehat{K}=\pi_X(K)$. By Lemma \ref{lemma:distance_convergence} \ref{lemma:distance:i}, for each $\delta>0$, there exists an open neighborhood $\widehat{V}$ of $h(\widehat{K})$ such that the open set $\pi_Y^{-1}(\widehat{V})$ contains $\pi_Y^{-1}( h(\widehat{K}))$ and is contained the open $\delta$-neighborhood of $\pi_Y^{-1}(h(\widehat{K}))$. In particular, by the compactness of $\pi_Y^{-1}(h(\widehat{K}))$, for each $\varepsilon>0$ we may find such an open set $\widehat{V}$ with the additional property that
\begin{align}\label{lemma:conformality_upper_bound}
\Sigma( \pi_Y^{-1}(\widehat V)) \leq \Sigma(\pi_Y^{-1}( h(\widehat{K})))+\varepsilon.
\end{align}

Since $h_n(  \widehat K )$ converges in the Hausdorff sense to $h(\widehat{K})$, we have $h_n( \widehat{K})\subset \widehat V$ for all sufficiently large $n\in \N$. This implies that $g_n^*( \pi_X^{-1}(\widehat{K})) \subset \pi_Y^{-1}(\widehat{V})$ for all sufficiently large $n\in \N$.  This inclusion, the conformality of $g_n$, and \eqref{lemma:conformality_upper_bound}, give
\begin{align*}
\int_{K} |Dg_n|^2\, d\Sigma &=\int_{K\cap X_n} |Dg_n|^2\, d\Sigma \leq \int_{\pi_X^{-1}(\widehat{K})\cap X_n} |Dg_n|^2\, d\Sigma  \\
&= \Sigma( g_n(\pi_X^{-1}(\widehat K)\cap X_n))\leq \Sigma( \pi_{Y}^{-1}(\widehat V))\leq \Sigma(\pi_Y^{-1}( h(\widehat{K})))+\varepsilon
\end{align*}
for all sufficiently large $n\in \N$. Since $|Dg_n|$ converges weakly to $\rho_h$ in $L^2(\widehat{\C})$, we see that $|Dg_n|\chi_{K}$ also converges weakly to $\rho_h\chi_K$. Thus,
\begin{align*}
\int_{K} \rho_h^2\, d\Sigma \leq \liminf_{n\to\infty} \int_{K} |Dg_n|^2\, d\Sigma\leq  \Sigma(\pi_Y^{-1}( h(\widehat{K})))+\varepsilon.
\end{align*}
Finally, we let $\varepsilon\to 0$.
\end{proof}

\begin{lemma}\label{lemma:support_rho}
The function $\rho_h$ is supported in the set $\pi_X^{-1}( h^{-1}( \pi_Y(Y)))$.  
\end{lemma}
\begin{proof}It suffices to show that $\rho_h=0$ on $\pi_X^{-1}(h^{-1}(\pi_Y(q_i)))$ for each $i\in \N$. The argument is similar to the one used in the previous lemma. Let $K$ be a non-empty compact subset of $\pi_X^{-1}(h^{-1}(\pi_Y(q_i)))$. It suffices to show that $\rho_h=0$ a.e.\ on $K$.  We set $\widehat{K}=\pi_X(K)$ and note that $h(\widehat{K})$ is the singleton $\pi_Y(q_i)$. By Lemma \ref{lemma:distance_convergence}, for each $\delta>0$, there exists an open neighborhood $\widehat{V}$ of $h(\widehat{K})$ such that the open set $\pi_Y^{-1}(\widehat{V})$ contains $\pi_Y^{-1}( h(\widehat{K}))$ and is contained the open $\delta$-neighborhood of $\pi_Y^{-1}(h(\widehat{K}))=q_i$. Therefore, for each $\varepsilon>0$,  we may find such an open set $\widehat{V}$ with the additional property that
\begin{align*}
\Sigma( \pi_Y^{-1}(\widehat V)\setminus q_i) < \varepsilon.
\end{align*}
As in the proof of Lemma \ref{lemma:conformality}, $g_n^*(\pi_X^{-1}(\widehat{K}))\subset \pi_Y^{-1}(\widehat{V})$ for all sufficiently large $n\in \N$, and particularly, $g_n^*(\pi_X^{-1}(\widehat{K}) \setminus p_{i,n})\subset \pi_Y^{-1}(\widehat{V})\setminus q_i$. We now obtain
\begin{align*}
\int_{K} |Dg_n|^2\, d\Sigma &=\int_{(K\setminus p_{i,n})\cap X_n} |Dg_n|^2\, d\Sigma  \leq \int_{(\pi_X^{-1}(\widehat K)\setminus p_{i,n})\cap X_n} |Dg_n|^2\, d\Sigma \\
&= \Sigma( g_n( (\pi_X^{-1}(\widehat K)\setminus p_{i,n})\cap X_n))\leq \Sigma( \pi_{Y}^{-1}(\widehat V)\setminus q_i)<\varepsilon.
\end{align*}
Taking limits and using the weak convergence of $|Dg_n|\chi_K$ to $\rho_h\chi_K$, we obtain
$$\int_{K}\rho_h^2\leq \varepsilon.$$
We let $\varepsilon\to 0$, so $\rho_h=0$ a.e.\ on $K$. 
\end{proof}

A consequence of Lemma \ref{lemma:conformality} and Lemma \ref{lemma:support_rho} is the following statement, which concludes the proof that $h$ is packing-conformal and the proof of Theorem \ref{theorem:uniformization_full}.
\begin{corollary}
For all Borel sets $E\subset  \widehat{\C}$ we have
$$\int_{\pi_X^{-1}(h^{-1}(\pi_Y(E)))} \rho_h^2\, d\Sigma\leq \Sigma(E\cap Y).$$
\end{corollary}

\section{Topology of planar maps}

Our goal in this section is to study continuous, proper, and cell-like maps between simply connected domains in the sphere. In view of the approximation theorem, Theorem \ref{theorem:approximation}, these maps behave like homeomorphisms. Specifically, we wish to understand when these maps have an extension to the boundaries and to the whole sphere satisfying certain properties. This section can be read independently of the other sections. The results are used in the proof of Theorem \ref{theorem:continuous_extension} in the next section.

\subsection{Conditions for continuous extension}
Let $f\colon \Omega \to D$ be a map between domains in $\widehat{\C}$. For a point $z_0\in \partial \Omega$ we define the \textit{cluster set} $\clu(f,z_0)$ to be the set of accumulation points of $\{f(z_n)\}_{n\in \N}$ over all sequences $\{z_n\}_{n\in \N}$ in $\Omega$ converging to $z_0$. Recall from Section \ref{section:topological} that a continuous map between open subsets of the sphere is cell-like if the preimage of each point is a non-separating continuum.

\begin{lemma}\label{lemma:cluster}
Let $\Omega,D\subset \widehat{\C}$ be simply connected regions such that $\partial D$ is a Peano continuum and let $f\colon \Omega\to D$ be a continuous, proper, and cell-like map. 
\begin{enumerate}[\upshape(i)]
	\item\label{lemma:cluster_bound} For each $z_0\in \partial \Omega$  and $\varepsilon>0$ there exists $\delta>0$ such that if $\gamma$ is a closed curve in $B(z_0,\delta)\setminus \{z_0\}$ that is not null-homotopic,  then $$\diam(f(|\gamma|\cap \Omega))\geq \diam(\clu(f,z_0))-\varepsilon.$$
	\item\label{lemma:cluster_extension} The map $f$ extends to a continuous map on $\br \Omega$ if and only if  for each $z_0\in \partial \Omega$ and $z_{\infty}\in \widehat{\C}\setminus \{z_0\}$ there exists a sequence of curves $\gamma_n$ in  $\widehat{\C}\setminus \{z_0,z_\infty\}$ that are not null-homotopic and converge to $z_0$ such that 
$$\lim_{n\to\infty}\diam(f(|\gamma_n|\cap \Omega))=0.$$
In this case $f(\br \Omega)=\br D$ and $f^{-1}(\partial D)=\partial \Omega$. 
\end{enumerate}
\end{lemma}

\begin{proof}
This first part of the lemma is established in \cite{NtalampekosYounsi:rigidity}, as Lemma 3.10, under the assumption that $f$ is a homeomorphism and the complements of $\Omega,D$ are points or closed Jordan regions; see the first two paragraphs in \cite[p.~152]{NtalampekosYounsi:rigidity}. The proof applies with few changes to this more general setting so we omit it, but we make instead a few remarks.  One can reduce the statement to homeomorphisms via the approximation theorem, Theorem \ref{theorem:approximation}. In particular, one can replace $f$ by a homeomorphism without altering the cluster sets and so that $\diam(f(|\gamma|\cap \Omega))$ is altered very slightly. The assumption that $\Omega$ is a Jordan region or a point is not used in \cite{NtalampekosYounsi:rigidity}.  Finally, for $D$ all we need is that it has the property that any two points that are close to each other can be connected with a path in $D$ that is small in diameter; this is true since $\partial D$ is a Peano continuum \cite[Theorem (4.2), p.~112]{Whyburn:topology}.

We prove part \ref{lemma:cluster_extension}. If there is a continuous extension then the conclusion holds trivially for any sequence of curves $\gamma_n$ as in the statement. Conversely, suppose that for each $z_0$ there exists a sequence of curves $\gamma_n$ as in the statement. By \ref{lemma:cluster_bound} we see that $\clu(f,z_0)$ contains only one point. Thus, $f$ extends continuously to $\br \Omega$. The surjectivity of $f\colon \Omega\to D$ from Lemma \ref{lemma:cell_like} \ref{lemma:cell_like:surjective} implies the surjectivity of the extension onto $\br D$. The properness of $f$ implies that $\clu(f,z_0)\subset \partial D$ for each $z_0\in \partial \Omega$. Hence $f^{-1}(\partial D)=\partial \Omega$.
\end{proof}

\subsection{Implications of continuous extension}
We state the main result of the section. 

\begin{theorem}\label{theorem:cell_like_extension}
Let $\Omega\subset \widehat{\C}$ be a simply connected region and $Y,D\subset \widehat{\C}$ be Jordan regions such that $\Omega\subset Y$. Assume that $f\colon \Omega\to D$ is a continuous, proper, and cell-like map that extends to a continuous map from $\br \Omega$ onto $\br D$ with the property that for each component $V$ of $Y\setminus \br \Omega$, $f|_{\partial V\setminus \partial Y}$ is constant. Then there exists a unique extension of $f$ to a continuous and cell-like map  $\widetilde f\colon  \br Y\to \br D$. The map $\widetilde f$ is constant in each continuum $E\subset Y\setminus \Omega$ and can be further extended to a continuous and cell-like map $\widetilde f\colon \widehat{\C}\to \widehat{\C}$ such that  $\widetilde f^{-1}(\widehat{\C}\setminus \br D)= \widehat{\C}\setminus \br Y$. 
\end{theorem}

The proof will be completed in several steps. We establish some preliminary statements.

\begin{lemma}\label{lemma:dense}
Let $U\subsetneq \widehat{\C}$ be an open set with $\partial \br U=\partial U$ and let $P\subset U$ be a Jordan region. Then for each component $V$ of $U\setminus \br P$ the set $\partial V\setminus \partial P$ is a non-empty dense subset of $\partial V\cap \partial U$. 
\end{lemma}

The assumption $\partial \br U=\partial U$ holds automatically if $\Omega$ is an open set and $U$ is a connected component of $\widehat{\C}\setminus \br \Omega$.

\begin{proof}
Using the Schoenflies theorem, we may assume that $\partial P$ is a circle. Suppose that $\partial V\setminus \partial P=\emptyset$, so $\partial V\subset \partial P$. The set $\partial V$ is a non-empty closed subset of $\partial P$. Let $x\in \partial V$ and consider a ball $B(x,r)$ that does not contain $P$. Then $\partial P$ separates the ball $B(x,r)$ into two connected open sets, one contained in $P$, and one outside $\br P$. The latter set does not intersect $\partial V$ and contains points of $V$, since $x\in \partial V$. By the connectedness of $V$, we have $B(x,r)\setminus \br P\subset V$. This implies that $B(x,r)\cap \partial P\subset \partial V$, so $\partial V$ is open in the relative topology of $\partial P$. We conclude that $\partial V=\partial P$. The Jordan curve theorem implies that $V= \widehat{\C}\setminus \br P$. Thus, $\widehat{\C}=V\cup \br P \subset U$, which contradicts the assumption that $U\subsetneq \widehat{\C}$.  Therefore, $\partial V\setminus \partial P\neq \emptyset$. 

We show that $\partial V\setminus \partial P\subset \partial V\cap \partial  U$. If not, there exists $x\in \partial V\setminus \partial P\subset  \partial V\subset \br U$ with $x\in U$. Since $x\notin \br P$, there exists a ball $B(x,r)\subset U\setminus \br P$. This implies that $V\cup B(x,r)$ is a connected open subset of $U\setminus \br P$, which contradicts the assumption that $V$ is a connected component of $U\setminus \br P$. 

It remains to show that each point of $\partial V\cap \partial U\cap \partial P$ can be approximated by points of $\partial V\setminus \partial P$.  For the sake of contradiction, suppose that there exists $x\in \partial V\cap \partial U\cap \partial P$ and an open ball $B(x,r)$ that does not intersect $\partial V\setminus \partial P$ and does not contain $\partial P$ (upon choosing a small enough $r>0$). Thus, $B(x,r)\cap \partial V\subset \partial P$. Consider the arc $\alpha=B(x,r)\cap \partial P$.   The arc $\alpha$ separates the ball $B(x,r)$ into precisely two regions: one region contained in $P$ and hence not intersecting $V$, and one region outside $\br P$ that does not intersect $\partial V$. The latter region contains points of $V$ near $x$, since $x\in \partial V$.  By the connectedness of $V$, this region is contained in $V$. Therefore, $\alpha\subset \partial V$. We conclude that $B(x,r)$ is contained in $P\cup \br V\subset  \br U$. This contradicts that $x\in \partial U =\partial \br U$. 
\end{proof}

\begin{lemma}\label{lemma:cell_like_extension}
Let $E\subset \widehat{\C}$ be a compact set, $D\subset \widehat{\C}$ be a Jordan region, and $g\colon E\to \br D$ be a continuous, surjective, and cell-like map. Then $E\neq \widehat{\C}$, $E$ is a non-separating continuum, and $g(\partial E)=\partial D$. Moreover, there exists an extension of $g$ to a continuous and cell-like map $\widetilde g\colon \widehat{\C}\to \widehat{\C}$ such that $\widetilde g^{-1}(\widehat{\C}\setminus \br D) = \widehat{\C}\setminus E$.
\end{lemma}

\begin{proof}
We consider a decomposition of $\widehat{\C}$ into the non-separating continua $g^{-1}(z)$, $z\in \br D$, and the remaining singleton points.  This decomposition is upper semicontinuous, as follows from the continuity of $g$.  By Moore's theorem (Theorem \ref{theorem:moore:original}), $\widehat{\C}/\sim$ is a topological $2$-sphere. Consider the projection $\pi\colon \widehat{\C}\to \widehat{\C}/\sim$. We define the map $G=g\circ \pi^{-1}$ on $\pi(E)$ by mapping each point $\pi(g^{-1}(z))$ to $z$, where $z\in \br D$. Trivially, this map is injective.  Using Lemma \ref{lemma:distance_convergence} and the continuity of $g$ one can show that $G$ is continuous from $\pi(E)$ onto $\br D$. Therefore, $G$ is a homeomorphism, which implies that $\pi( E)$ is a closed Jordan region, whose interior we denote by $W$; moreover, $G(\partial W)=\partial D$. The region $\widehat{\C}\setminus E$  projects under $\pi$ homeomorphically onto the complement of $\br W$. In particular, $\widehat{\C}\setminus E\neq \emptyset$ and $\pi(\partial E)=\pi(\partial (\widehat{\C}\setminus E))=\partial W$. This implies that 
$g(\partial E)=G(\pi(\partial E))=\partial D$. Finally, consider an arbitrary extension of $G$ to a homeomorphism of the sphere $\widehat{\C}/\sim$. Then $\widetilde g=G\circ \pi$ gives the desired extension of $g$. 
\end{proof}

In the next lemmas the standing assumptions are that $\Omega\subset \widehat{\C}$ is a simply connected region, $D\subset \widehat{\C}$ is a Jordan region, and $f\colon \Omega\to D$ is a continuous, proper, and cell-like map that has an extension to a continuous map from $\br \Omega$ onto $\br D$, which we also denote by $f$. 

\begin{lemma}\label{lemma:monotone}
There exists a unique component $U_0$ of $\widehat \C\setminus {\br \Omega}$ such that $f( \partial U_0)=\partial D$. Moreover, there exists an extension of $f$ to a  continuous and cell-like map $\widetilde f\colon \widehat{\C}\setminus U_0 \to \br D$ such that $\widetilde f|_{\br U}$ is constant for each component $U\neq U_0$ of $\widehat{\C} \setminus \br \Omega$.
\end{lemma}

See Figure \ref{figure:monotone} for an illustration of the conclusions of the lemma. 

\begin{figure}
\input{monotone.tikz}
	\caption{The region $U_0$ is the unbounded component of $\widehat{\C}\setminus \br \Omega$. The map $\widetilde f$ is constant on each component $U\neq U_0$ of $\widehat{\C}\setminus \br \Omega$. Here the set $A_z$ is the closure of the shaded regions.}\label{figure:monotone}
\end{figure}
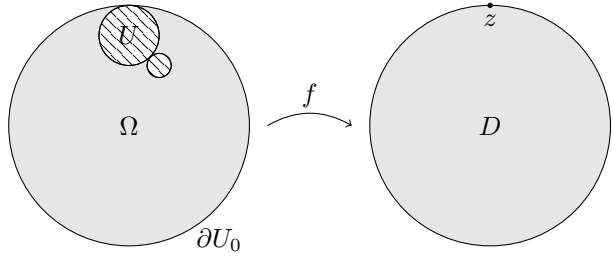

\begin{proof} 
By postcomposing $f$ with a homeomorphism of $\widehat{\C}$ via the Schoenflies theorem, we assume that $D$ is the unit disk. For each $z\in \partial D$ the set $f^{-1}(z)$ is a subset of $\partial \Omega$  by Lemma \ref{lemma:cluster} \ref{lemma:cluster_extension}. Moreover, we have 
\begin{align*}
f^{-1}(z)= \bigcap_{r>0} \br{f^{-1}(B(z,r)\cap D)}
\end{align*}
by the continuity of $f$. Since $f$ is cell-like  and the set $B(z,r)\cap D$ is connected, by Lemma \ref{lemma:cell_like} \ref{lemma:cell_like:open} the same is true for $f^{-1}(B(z,r)\cap D)$, and thus for $\br{f^{-1}(B(z,r)\cap D)}$. We conclude that $f^{-1}(z)$ is a connected subset of $\partial \Omega$.   

For each $z\in D$, the set $A_z=f^{-1}(z)$ is a non-separating continuum in $\Omega$, since $f$ is cell-like and $f^{-1}(\partial D)=\partial \Omega$. For $z\in \partial D$ the set $f^{-1}(z)$ is a connected subset of $\partial \Omega$. If $f^{-1}(z)$ does not separate the sphere, then we set  $A_z=f^{-1}(z)$.  If $f^{-1}(z)$ separates the sphere, then $\widehat{\C}\setminus f^{-1}(z)$ contains at least two components. Precisely one of the components of $\widehat{\C}\setminus f^{-1}(z)$ contains the connected set $\Omega$. We define $A_z$ to be the complement of that component; see Figure \ref{figure:monotone}. Then $A_z$ is non-separating by its definition and has the property that
\begin{align}\label{lemma:monotone:az1}
\emptyset\neq  \partial A_z\subset f^{-1}(z)\subset \partial \Omega.
\end{align}
The set $\inter(A_z)$ is disjoint from $\br \Omega$, so
\begin{align}\label{lemma:monotone:az2}
A_z\cap \partial \Omega = \partial A_z\cap \partial \Omega \subset f^{-1}(z).
\end{align}
We claim that $A_z\cap A_w=\emptyset$ for each $w\neq z$, $z,w\in \br D$. If $z\in D$, then $A_z=f^{-1}(z)\subset \Omega$ and the claim is trivial. Suppose that $z,w\in \partial D$.  By \eqref{lemma:monotone:az2}, we have 
$$A_z\cap f^{-1}(w)= A_z\cap f^{-1}(w)\cap \partial \Omega \subset f^{-1}(z)\cap f^{-1}(w)=\emptyset.$$
Thus, the connected set $A_z$ is contained in a component of $\widehat{\C}\setminus f^{-1}(w)$. The definition of $A_w$ implies that $A_z\subset A_w$ or $A_z\cap A_w=\emptyset$. If $A_z\subset A_w$, then then by reversing the roles of $z$ and $w$ one obtains $A_w\subset A_z$, so $A_z=A_w$. This contradicts \eqref{lemma:monotone:az1}.

Consider the set $\mathcal A=\bigcup_{z\in \br D} A_z$, which contains $\br \Omega$. We observe that this set is closed in $\widehat{\C}$. Indeed, if $x_n$ is a sequence in  $A_{z_n}$, where $z_n\in \br D$ are distinct points, then there exists $x_n'\in \partial A_{z_n} \subset \br \Omega$ such that $\sigma(x_n,x_n')\to 0$. Thus, $x_n'$, and hence $x_n$ as well, must accumulate at points of $\br \Omega \subset \mathcal A$. Another observation is that 
\begin{align}\label{lemma:monotone:a}
\partial \mathcal A\subset \partial \Omega.
\end{align}
Indeed, if $x\in \partial\mathcal A\subset\mathcal A$, then $x\in \partial A_z$ for some $z\in \br D$. If $z\in D$, then $A_z\subset \Omega\subset \inter(\mathcal A)$, so we must have $z\in \partial D$. In this case, $\partial A_z\subset \partial \Omega$ by \eqref{lemma:monotone:az1}, as desired. 

We extend $f$ to $\mathcal A=\bigcup_{z\in \br D} A_z$ by defining $\widetilde f(x)=z$ for $x\in A_z$. Observe that if $x\in A_z$, then for each point $x'\in \partial A_z\subset f^{-1}(z)$ we have $\widetilde f(x)=z=f(x')$.  We claim that the extension is continuous. If not, there exists $\varepsilon>0$ and a sequence $x_n\to x$, such that $\sigma(\widetilde f(x_n),\widetilde f(x))\geq \varepsilon$ for each $n\in \N$. Note that the point $x$ cannot lie in the interior of $A_z$ for any $z\in \br D$, since $\widetilde f$ is continuous there. Thus, we necessarily have $x\in \partial A_z$ for some $z\in \br D$. By \eqref{lemma:monotone:az1}, $x\in \br \Omega$, so $\widetilde f(x)=f(x)$.   Assume that $x_n\in A_{z_n}$, $n\in \N$. For each $n\in \N$ we connect $x_n$ to $x$ with a small path in $\widehat{\C}$ and we see that there exists a point $x_n'\in \partial A_{z_n}\subset f^{-1}(z_n)$ such that $\widetilde f(x_n)=f(x_n')$ and $x_n'\to x$. This contradicts the continuity of $f$ on $\br \Omega$. Therefore, $\widetilde f$ is continuous. By construction, $\widetilde f^{-1}(z)=A_z$ is a non-separating continuum for each $z\in \br D$; thus $\widetilde f$ is also cell-like.

By Lemma \ref{lemma:cell_like_extension}, $\widehat{\C}\setminus \mathcal A\neq \emptyset$, $\mathcal A$ is a non-separating continuum, and $\widetilde f(\partial \mathcal A)=\partial D$. We define $U_0=\widehat{\C}\setminus \mathcal A$, so $\widetilde f(\partial U_0)=\partial D$. The connected set $U_0$ is contained in a component of $\widehat{\C}\setminus \br \Omega$. By \eqref{lemma:monotone:a} we have $\partial U_0=\partial \mathcal A\subset \partial \Omega$, which implies that $U_0$ is a component of $\widehat{\C}\setminus \br\Omega$.  Let $U$ be a  component of $\widehat{\C}\setminus \br \Omega$ that is different from $U_0$. Then $U\subset \mathcal A$, so $U\cap A_z\neq \emptyset$ for some $z\in \br D$; in fact, $z\in \partial D$, since $U\cap \Omega=\emptyset$. Note that $U$ is a connected set that is disjoint from $f^{-1}(z)$. By the definition of $A_z$ we have $U\subset A_z$. In particular, $\widetilde f$ is constant on $\br U$. 
\end{proof}

\begin{lemma}\label{lemma:countably_many_values}
Let $Y\subset \widehat{\C}$ be a Jordan region such that $\Omega\subset Y$.  Assume that for each component $V$ of $Y\setminus \br \Omega$ the restriction $f|_{\partial V\setminus \partial Y}$ is constant. Then there exists a unique extension of $f$ to a continuous and cell-like map $\widetilde f\colon \br Y \to \br D$. The map $\widetilde f|_{\br Y\setminus \br \Omega}$ attains countably many values and $\widetilde f^{-1}(D)=\Omega$.  
\end{lemma}
\begin{proof}
Consider the region $U_0$ and the extension $\widetilde f\colon \widehat \C\setminus U_0\to \br D$ as provided by Lemma \ref{lemma:monotone}. Let $Y\supset \Omega$ as in the statement. The connected region $\widehat{\C}\setminus \br Y$ is contained in a component of $\widehat{\C}\setminus \br \Omega$. If $\widehat{\C}\setminus \br Y$ is contained in a component $U\neq U_0$ of $\widehat{\C}\setminus \br \Omega$, then $\widetilde f|_{\br U}$ is constant; moreover, $U_0\subset Y$ and $U_0$ is a component of $Y\setminus \br \Omega$.  By assumption, $f|_{\partial U_0\setminus \partial Y}$ is constant. Since $\partial U_0\cap \partial Y\subset \partial U_0\cap \br U$, $f$ is also constant on $\partial U_0\cap \partial Y$. This contradicts the surjectivity of $f|_{\partial U_0}$ onto $\partial D$ from Lemma \ref{lemma:monotone}. Therefore $\widehat{\C}\setminus \br Y\subset U_0$ and $\partial Y\subset \br {U_0}$.

Next, we extend $\widetilde f$ to $\br Y\setminus \br \Omega$ as follows. If $V$ is a component of $Y\setminus \br \Omega$, then either $V$ is a component of $\widehat{\C}\setminus \br \Omega$ that is disjoint from $U_0$, or $V$ is  a component of $Y\cap U_0$.  If $V$ is disjoint from $U_0$, then $\widetilde f$ has already been defined in $\br V$ and is constant with value equal to $f(\partial V)\subset f(\partial \Omega)\subset \partial D$. Suppose that $V$ is a component of  $Y\cap U_0$.  By Lemma \ref{lemma:dense} (applied to ${U_0}$ and $P=\widehat{\C}\setminus \br Y$), the set $\partial V\setminus \partial Y$ is a non-empty dense subset of $\partial V\cap \partial U_0$. By assumption, $f|_{\partial V\setminus \partial Y}$ is constant, so by continuity $f|_{\partial V\cap \partial U_0}$ is constant.  We define 
$$\widetilde f|_{\br V} = f|_{\partial V\cap \partial U_0}.$$
Since $f(\partial U_0)\subset \partial D$, we see that $\widetilde f(\br V)$ is a point of $\partial D$. 

First, we ensure that this gives a well-defined map. Let $V_1,V_2$ be distinct components of $Y\cap U_0$ such that $\br{V_1}\cap \br{V_2}\neq \emptyset$. If $\br{V_1}\cap \br{V_2}$ contains a point $x\in U_0$, then that point would have to lie on $\partial Y$; however, $\partial Y$ has exactly two ``sides" near $x$, one contained in the region  $Y$, and one contained in precisely one of $V_1,V_2$. This is a contradiction. Therefore, $\br{V_1}\cap \br{V_2}\subset \partial U_0$. Since $f$ is constant on $\partial V_i\cap \partial U_0$, $i=1,2$, we conclude that $\widetilde f|_{\br{V_1}}=\widetilde f|_{\br{V_2}}$, as desired. 

Next, we note that 
$$\br Y= \br \Omega \cup \bigcup_{V\subset Y\setminus \br \Omega}\br V,$$
where the union is over all components $V$ of $Y\setminus \br \Omega$. Hence, the above definition provides an extension of $\widetilde f$ to all of $\br Y$. Indeed, if $x\in Y\setminus \br \Omega$, then $x$ lies in a component $V$ of $Y\setminus \br \Omega$. If $x\in \partial Y \setminus \br \Omega$, then  $x\in \partial Y\cap U_0$, so $x$ lies in the boundary of a component $V$ of $Y\setminus \br \Omega$, because $\partial Y$ is ``two-sided" near $x$. 

Since there are countably many components $V$ of $Y\setminus \br \Omega$, we observe that $\widetilde f$ attains countably many values in $\br Y\setminus \br \Omega$, as required in the conclusion of Lemma \ref{lemma:countably_many_values}. Moreover, by the definition of the extension, we have $\widetilde f^{-1}(D)=\Omega$. 

\begin{figure}
\input{graph.tikz}
	\caption{A portion of $\partial \Omega$ and $\partial Y$. }\label{figure:graph}
\end{figure}
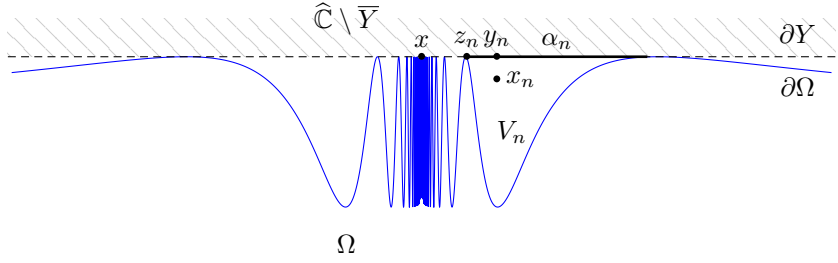

Now, we show that $\widetilde f$ is continuous. Since $\widetilde f$ is already continuous in the complement of $U_0$ by Lemma \ref{lemma:monotone}, it suffices to show that if $x_n$ is a sequence in $\br Y\cap U_0$ converging to some point $x$, then $\widetilde f(x_n)$ converges to $\widetilde f(x)$.  For the sake of contradiction, suppose that there exists $\varepsilon>0$ such that $\sigma(\widetilde f(x_n),f(x))\geq \varepsilon$ for all $n\in \N$.  If $x\in U_0$, then $x,x_n\in \br V$ for all sufficiently large $n\in \N$, for some component $V$ of $Y\cap U_0$; thus, $\widetilde f(x)=\widetilde f(x_n)$ for all large $n\in \N$ and we obtain a contradiction. Suppose $x\in \partial U_0$.  We have $x_n\in \br {V_n}$ for a sequence of components  $V_n$ of $Y\cap U_0$. If $\br {V_n}$ is a specific component $\br V$ for infinitely many $n\in \N$, then $x\in \br V$, so $\widetilde f(x_n)=\widetilde f(x)$ for infinitely many $n\in \N$, a contradiction. Thus, we may assume that the components $V_n$, $n\in \N$, are distinct. One may find points $y_n\in \partial V_n$ with $\sigma(y_n,x_n)\to 0$ and $\widetilde f(y_n)=\widetilde f(x_n)$. By \cite[Prop.\ 2.13]{Pommerenke:conformal} for each $V_n$ there exists an open arc $\alpha_n=\partial V_n\cap  U_0$ that is a component of $\partial Y\cap U_0$, so that $\partial V_n$ is the union of $\alpha_n$ and $\partial V_n\cap \partial U_0$; see Figure \ref{figure:graph}. Moreover, for distinct sets $V_n$ the arcs $\alpha_n$ are disjoint, since $\partial Y$ is ``two-sided". The local connectivity of $\partial Y$ implies that  $\diam(\alpha_n)\to 0$ as $n\to\infty$. If $y_n\in \partial V_n\cap \partial U_0$, we set $z_n=y_n$ and if $y_n\in \alpha_n$, we set $z_n$ to be an endpoint of $\alpha_n$ (Figure \ref{figure:graph}). Therefore, one may find points $z_n\in \partial V_n\cap \partial U_0$  such that $\widetilde f(y_n)=\widetilde f(z_n)=f(z_n)$ and  $\sigma(y_n,z_n)\to 0$. Since $z_n\to x$, the continuity of $f$ on $\partial U_0\subset \br \Omega$ implies that $f(z_n)\to f(x)$.  This contradiction completes the proof of the continuity of $\widetilde f$.

Next, we show that for each $z\in \br D$, the set $\widetilde f^{-1}(z)$ is a non-separating continuum. If $z\in D$, then $\widetilde f^{-1}(z)=f^{-1}(z)\subset \Omega$, so $f^{-1}(z)$ is a non-separating continuum by the cell-likeness of $f$. Next, suppose that $z\in \partial D$.  By Lemma \ref{lemma:monotone}, $\widetilde f|_{\widehat{\C}\setminus U_0}$ is cell-like, so the set $E=\widetilde f^{-1}(z)\cap (\widehat{\C}\setminus U_0)$ is a non-separating continuum. By the definition of $\widetilde f$ on $\br Y$, $\widetilde f^{-1}(z)$ is the union of $E$ with the closures of components $V$ of $Y\cap U_0$ such that $\widetilde f|_{\br V}\equiv z$. By the previous, for each such component we have $\partial V\cap \partial U_0\neq \emptyset$, so $\br V\cap E\neq \emptyset $.  This shows that $\widetilde f^{-1}(z)$ is connected. 

In order to show that $\widetilde f^{-1}(z)$ is non-separating, we observe that  $\widehat{\C}\setminus \widetilde f^{-1}(z)$ is the union of $\br \Omega\setminus \widetilde f^{-1}(z)$ with $\widehat \C\setminus \br{Y}$ and  with the sets $\br V$ where $V$ is a component of $Y\setminus \br \Omega$  with $\br V\cap \widetilde f^{-1}(z)=\emptyset$.  Let $W$ be the component of $\widehat{\C}\setminus \widetilde f^{-1}(z)$ that contains the connected set $\Omega$ and note that   $\br \Omega\setminus \widetilde f^{-1}(z)$ is connected, so  it is contained in $W$.  The sets $V$ as above are either components of $\widehat{\C}\setminus \br \Omega$ or components of $Y\cap U_0$. In the first case we have $\partial V\subset \partial \Omega$ so $\br V\cap (\br \Omega \setminus \widetilde f^{-1}(z))\neq \emptyset$, while in the second case,  $\partial V\cap \partial U_0\neq \emptyset$,  so  $\br V \cap ( \br \Omega \setminus \widetilde f^{-1}(z))\neq \emptyset$. Therefore, in both cases, $\br V\subset W$.  Since $\widetilde f|_{\partial U_0}$ is surjective onto $\partial D$, there exists a point $x\in \partial U_0\setminus \widetilde f^{-1}(z)\subset \br \Omega\setminus \widetilde f^{-1}(z)$. If $x\in \partial Y$, then this shows that $\widehat{\C}\setminus \br{Y}$ lies in $W$. If $x\in \partial V$ for some $V$ as above, then $\br V\subset W$ by the previous and $\partial V$ contains an arc of $\partial Y$. Thus, $\widehat{\C}\setminus \br Y\subset W$. This completes the proof that $W=\widehat{\C}\setminus \widetilde f^{-1}(z)$, and thus $\widetilde f^{-1}(z)$ is non-separating.

Now, we show the uniqueness statement. Suppose that $\widetilde f_1$ is another extension of $f$ to a continuous and cell-like map from $\br Y$ onto $\br D$. We first observe that $\widetilde f_1^{-1}(D)=\Omega$; indeed, if $\widetilde f_1^{-1}(z)$ intersects the complement of $\Omega$ for some $z\in D$, then $\widetilde f_1^{-1}(z)$ is a continuum that intersects both $\Omega$ and its complement, so it intersects $\partial \Omega$. However, $\widetilde f_1(\partial \Omega)=f(\partial \Omega)=\partial D$, a contradiction. 

Let $V$ be a component of $Y\setminus \br \Omega$, such that $\widetilde f|_{\br V}\equiv z$ for some $z\in \partial D$. Our goal is to show that $\widetilde f_1\equiv z$ on $\br V$. Suppose that $\widetilde f_1^{-1}(w)\cap \br V\neq \emptyset$ for some $w\in \br D$; then we necessarily have $w\in \partial D$ by the previous paragraph. The continuum $E=\widetilde f_1^{-1}(w)$ intersects $\partial U_0$, since $\widetilde f_1|_{\partial U_0}=f|_{\partial U_0}\colon \partial U_0\to \partial D$ is surjective. This implies that $\partial V\cap E\neq \emptyset$.  

If $V\cap U_0=\emptyset$, then $V$ is a component of $\widehat{\C}\setminus \br \Omega$, so $\partial V\subset \partial \Omega$. Hence, $w=\widetilde f_1|_{\partial V\cap E}=f|_{\partial V\cap E}= \widetilde f|_{\br V}= z$. Now, suppose that $V\subset U_0$, so $V$ is a component of $Y\cap U_0$.  We claim that $E\cap \partial V\cap \partial U_0\neq \emptyset$. Suppose that $E\cap \partial V\cap \partial U_0=\emptyset$. Note that $E\cap \br V$ is a closed subset of $E$. We will show that $E\cap \br V$ is open in the relative topology of $E$. It suffices to show that each $x\in E\cap \partial V$ has an open neighborhood in the topology of $E$ that is contained in $E\cap \br V$. By assumption, we have $$x\in E\cap \partial V= E\cap \partial V \cap (\partial Y\cup \partial U_0)= E\cap \partial V \cap \partial Y.$$
Since $Y$ is ``two-sided" and $E\subset \br Y$, all points of $E$ near $x$ are contained in $\br V$, as desired. Since $E\cap \br V$ is clopen in $E$, the connectedness of $E$ implies now that $E\cap \br V=E$. Thus, 
$$E \cap \partial  V\cap \partial U_0 = E\cap \br V \cap \partial U_0= E\cap \partial U_0\neq \emptyset,$$
a contradiction. Therefore, 
$E\cap \partial V\cap \partial U_0\neq \emptyset$, which implies that $w=\widetilde f_1|_{E\cap \partial V\cap \partial U_0}=f|_{\partial V\cap \partial U_0}= \widetilde f|_{\br V}=z$. Hence, we conclude that $w=z$ in all cases, as desired. 
\end{proof}

\begin{proof}[Proof of Theorem \ref{theorem:cell_like_extension}]
Consider the unique extension of $f$ to a continuous and cell-like  map $\widetilde f\colon \br Y \to \br D$ as provided by Lemma \ref{lemma:countably_many_values}. By Lemma \ref{lemma:cell_like_extension} the map $\widetilde f$ can be extended to a continuous and cell-like map of the sphere such that $\widetilde f^{-1}(\widehat{\C}\setminus \br D)=\widehat{\C} \setminus \br Y$, as required in the theorem. Moreover, $\widetilde f(\partial Y)=\partial D$. By Lemma \ref{lemma:countably_many_values},  $\widetilde f$ takes countably many values in $\br Y\setminus \br \Omega$. If $\partial Y\cap \partial \Omega=\emptyset$, then $\partial Y\subset  \br Y\setminus \br \Omega$, so $\widetilde f$ is constant on the connected set $\partial Y$. This contradicts the fact that $\widetilde f(\partial Y)=\partial D$. Therefore, $\partial Y\cap \partial \Omega\neq \emptyset$. 

Let $E$ be a continuum in $Y\setminus \Omega$ as in the statement of the theorem. Consider the component $A$ of $\widehat \C\setminus E$ that contains the connected set $\widehat \C \setminus  Y$. Since $\partial Y\cap \partial \Omega\neq \emptyset$ and $A$ is open, we see that $\Omega$ is also contained in $A$. Let $E'$ be the complement of $A$. Then $E'\supset E$ and $E'$ is a non-separating continuum disjoint from $\Omega$ and $\widehat{\C}\setminus Y$. We claim that $\widetilde f$ is constant on $E'$.   We argue by contradiction, assuming that $\widetilde f(E')$ is a non-degenerate continuum. By Lemma \ref{lemma:countably_many_values}, $\widetilde f^{-1}(D)=\Omega$, so $f(E')$ is an  arc of $\partial D$. Then there exist distinct points $z_1,z_2\in \partial D$ such that $\widetilde f^{-1}(z_i)\cap E'\neq \emptyset$, $i\in \{1,2\}$. Since $\widetilde f(\partial Y)=\partial D$, we see that $\widetilde f^{-1}(z_i)\cap \partial Y \neq \emptyset$ for $i\in \{1,2\}$.

We collapse topologically the non-separating continua $K=\widehat{\C}\setminus Y$ and $E'$ to points (e.g.\ via Moore's theorem) and consider a continuous projection map $\lambda\colon \widehat{\C}\to \widehat{\C}$  such that $\lambda(K)=w_1$, $\lambda(E')=w_2$ for some points $w_1\neq w_2$, and $\lambda$ is a homeomorphism from $\widehat\C\setminus (K\cup E')$ onto $\widehat\C\setminus \{w_1,w_2\}$. The continua $G_i=\lambda(\widetilde f^{-1}(z_i))$, $i\in \{1,2\}$, meet at the points $w_1$ and $w_2$, but otherwise they are disjoint. For $i\in \{1,2\}$ there exists a continuum $G_i'\subset G_i$ that is \textit{irreducible} between the points $w_1$ and $w_2$; that is, there is no proper subcontinuum of $G_i'$ that contains both $w_1$ and $w_2$. See \cite[Theorem 28.4]{Willard:topology} for the existence. 

A result of Moore \cite[Theorems 1 and 2]{Moore:separation} implies that $\widehat \C \setminus (G_1'\cup G_2')$ is the union of precisely two domains and the boundary of each is $G_1'\cup G_2'$. Let $Z$ be the component of $\widehat{\C}\setminus (G_1'\cup G_2')$ that is disjoint from $\lambda(\Omega)$ and thus from $\lambda (\br \Omega)$.  Since $\widetilde f$ attains countably many values in $\br Y\setminus\br \Omega$, we conclude that $\widetilde f\circ \lambda^{-1}$ is constant in $Z$.  However, the boundary values of $\widetilde f\circ \lambda^{-1}|_Z$  are non-constant, a contradiction.
\end{proof}

\section{Continuous extension}

In this section we provide sufficient conditions so that a packing-quasiconformal map between the topological spheres associated to two Sierpi\'nski packings as in Definition \ref{definition:packing_quasiconformal} can be lifted to a map between the actual packings. 

\begin{theorem}\label{theorem:continuous_extension}
Let $X=\widehat{\C}\setminus \bigcup_{i\in \N}p_i$ and $Y=\widehat{\C}\setminus \bigcup_{i\in \N} q_i$ be Sierpi\'nski packings such that the peripheral continua of $X$ are uniformly fat closed Jordan regions or points and the peripheral continua of $Y$ are closed Jordan regions or points with diameters lying in $\ell^2(\N)$. Let $ h\colon \mathcal E(X)\to \mathcal E(Y)$ be a packing-$K$-quasiconformal map for some $K\geq 1$.  Then there exists a continuous, surjective, and monotone map $H\colon \widehat \C \to \widehat \C$ such that $\pi_Y\circ H=h\circ \pi_X$ and $H^{-1}(\inter(q_i))=\inter(p_i)$ for each $i\in \N$.  Moreover, there exists a non-negative Borel function $\rho_H\in L^2(\widehat \C)$ with the following properties.
	\begin{itemize}
		\item  There exists a curve family $\Gamma_0$ in $\widehat{\C}$ with $\md_2\Gamma_0=0$ such that for all curves $\gamma\colon[a,b]\to \widehat\C$ outside $\Gamma_0$ we have
	\begin{align*}
\sigma( H(\gamma(a)), H(\gamma(b)) )\leq \int_{\gamma}\rho_H\, ds + \sum_{i:p_i\cap |\gamma|\neq \emptyset} \diam(q_i)
\end{align*}
		\item For each Borel set $E\subset  \widehat \C$ we have
	$$\int_{H^{-1}(E)} \rho_H^2\, d\Sigma \leq K \Sigma(E\cap Y).$$ 
	\end{itemize}
\end{theorem}

Moreover, we show that any map $H$ as in the theorem has a further topological property.

\begin{prop}\label{proposition:constant}
Let $H$ be a map satisfying the conclusions of Theorem \ref{theorem:continuous_extension}. Then for each $i\in \N$ and for each continuum $E\subset H^{-1}(q_i)\setminus p_i$ the map $H|_{E}$ is constant.   
\end{prop}

\begin{remark}
The assumptions of Theorem \ref{theorem:continuous_extension} can be relaxed, allowing the possibility that finitely many peripheral continua of $X$ are not fat. See Remark \ref{remark:cofatness_unnecessary} below for more details. 
\end{remark}

The proofs of both statements are given in Section \ref{section:patching}. We first establish several preliminary statements.

\subsection{Preliminaries}
\begin{lemma}\label{lemma:fubini}
Let $\tau>0$ and $X=\widehat{\C} \setminus \bigcup_{i\in \N} p_i$ be a $\tau$-cofat Sierpi\'nski packing. Let $\rho\colon \widehat{\C}\to [0,\infty]$ be a Borel function in $L^2(\widehat{\C})$ and $\{\lambda_i\}_{i\in \N}$ be a non-negative sequence in $\ell^2(\N)$. For each $x,y\in \widehat \C$ and $0<r_0<\diam(\widehat{\C})$ we have
\begin{align*}
\essinf_{r\in (r_0/2,r_0)} &\left(\int_{S(x,r)} \rho \, ds+\int_{S(y,r)} \rho \, ds+ \sum_{i:p_i\cap S(x,r)\neq \emptyset} \lambda_i +\sum_{i:p_i\cap S(y,r)\neq \emptyset} \lambda_i \right)\\
&\leq c(\tau) \left( \int_{B(x,r_0)\cup B(y,r_0)} \rho^2\, d\Sigma+\sum_{i:p_i\cap B(x,r_0)\neq \emptyset }\lambda_i^2+\sum_{i:p_i\cap B(y,r_0)\neq \emptyset }\lambda_i^2\right )^{1/2}.
\end{align*}
\end{lemma}
It is straightforward to obtain the conclusion of the lemma if $\lambda_i\equiv 0$ by integrating over $r\in (r_0/2,r_0)$ and using the co-area inequality of Proposition \ref{proposition:coarea}. For the general statement one  uses the cofatness of $X$ as well, in order to treat the summation terms.  The proof of the lemma follows from the argument used in \cite[Lemma 2.4.7 (a)]{Ntalampekos:CarpetsThesis} and we omit it.

\begin{lemma}\label{lemma:avoid}
Let $\Gamma_0$ be a path family in $\widehat \C$ with $\md_2\Gamma_0=0$ and let $x,y\in \widehat{\C}$. Define $\Gamma_1$ to be the family of paths $\gamma$ such that for a set of $r>0$ of positive Lebesgue measure there exists a simple path $\gamma_r\in \Gamma_0$ whose trace is contained in $|\gamma|\cup S(x,r)\cup S(y,r)$. Then $\md_2\Gamma_1=0$.
\end{lemma}
Alternatively, one can say that for $\md_2$-a.e.\ path $\gamma$ and for a.e.\ $r>0$, every simple path $\gamma_r$ whose trace is contained in $|\gamma|\cup S(x,r)\cup S(y,r)$ lies outside $\Gamma_0$. 
\begin{proof}
Since $\md_2\Gamma_0=0$, there exists a non-negative Borel function  $\rho\in L^2(\widehat{\C})$  such that $\int_{\gamma}\rho\, ds=\infty$ for each $\gamma\in \Gamma_0$ (see \cite[Lemma 5.2.8, p.~129]{HeinonenKoskelaShanmugalingamTyson:Sobolev}). Let $\gamma\in \Gamma_1$. Then for a set of $r>0$ of positive Lebesgue measure there exists a simple path $\gamma_r\in \Gamma_0$ whose trace is contained in $|\gamma|\cup S(x,r)\cup S(y,r)$. Thus,
\begin{align*}
\infty= \int_{\gamma_r} \rho\,ds \leq \int_{\gamma}\rho\, ds + \int_{S(x,r)}\rho\, ds+ \int_{S(y,r)}\rho\, ds.
\end{align*}
By the co-area inequality of Proposition \ref{proposition:coarea}, the latter two line integrals are finite for a.e.\ $r>0$. Thus, we conclude that $\int_{\gamma}\rho\, ds=\infty$. This implies that $\md_2\Gamma_1=0$. 
\end{proof}

\begin{remark}\label{remark:modulus}
We will use the following standard fact about modulus. If $\md_2\Gamma_0=0$, then the family of paths that have a subpath in $\Gamma_0$ also has modulus zero \cite[Theorem 6.4, p.~17]{Vaisala:quasiconformal}. Thus, we may apply the transboundary upper gradient inequality of packing-quasiconformal maps (see Definition \ref{definition:packing_quasiconformal}) not only for paths outside a family of modulus zero, but also for all subpaths of such paths, upon enlarging the exceptional curve family $\Gamma_0$.
\end{remark}

\subsection{Continuous extension to each peripheral continuum}\label{section:extension_individual}

Let $X,Y$ be as in Theorem \ref{theorem:continuous_extension} and $h\colon \mathcal E(X)\to \mathcal E(Y)$ be a packing-quasiconformal map as in Definition \ref{definition:packing_quasiconformal}. The proof of Theorem \ref{theorem:continuous_extension} will be completed in several steps. Consider the map $g=h\circ \pi_X\colon \widehat{\C}\to \mathcal E(Y)$, which is continuous and cell-like; indeed, the preimage of each point under $h$ is a non-separating continuum and the preimage of each non-separating continuum under $\pi_X$ is a non-separating continuum by Lemma \ref{lemma:cell_like} \ref{lemma:cell_like:set}. Recall that
$$\mathcal E(Y) = \pi_Y(\widehat{\C}) = \pi_Y(Y)\cup \bigcup_{i\in \N}\pi_Y(q_i).$$
Let $\widetilde X= g^{-1}(\pi_Y(Y))$ and note that $\widetilde X\subset X$ since $g(p_i)=h( \pi_X(p_i))= \pi_Y(q_i)$ for each $i\in \N$,  by the definition of a packing-quasiconformal map. If we set $\widetilde p_i= g^{-1}(\pi_Y(q_i))$, we see that $\widetilde X= \widehat{\C}\setminus \bigcup_{i\in \N} \widetilde p_i$. Note that $\widetilde p_i$ is non-separating, since $g$ is cell-like. 
 
\begin{lemma}\label{lemma:H_definition}
The map $H=\pi_Y^{-1}\circ g\colon \widetilde X\to Y$ is continuous, surjective, and cell-like.
\end{lemma}
\begin{proof}
Note that $g(\widetilde X)=\pi_Y(Y)$ and that each point of $\pi_Y(Y)$ has a unique preimage under $\pi_Y$, so $H=\pi_Y^{-1}\circ g$ gives a well-defined map from $\widetilde X$ onto $Y$. The continuity of the maps $\pi_Y,g$, together with the injectivity of $\pi_Y$ on $Y$, imply that $H$ is continuous on $\widetilde X$. Finally, the fact that $g$ is cell-like implies immediately that $H$ is cell-like. 
\end{proof}

We fix $i\in \N$ and consider the space $\mathcal E(Y;q_i)$. Recall that for a set $E\subset \widehat{\C}$ the space $\mathcal E(Y;E)$ is defined by collapsing to points all peripheral continua of $Y$ that do not intersect $E$. The space $\mathcal E(Y;q_i)$ is a topological $2$-sphere by Moore's theorem (Theorem \ref{theorem:moore_original}). The projection $\pi_{Y;q_i}\colon \widehat{\C}\to \mathcal E(Y;q_i)$ maps each $q_j$, $j\neq i$, to a point and is injective otherwise. 

\begin{lemma}\label{lemma:h_definition}
For each $i\in \N$ the open set  $\Omega_i=\widehat{\C}\setminus \widetilde p_i$ is simply connected and $D_i=\mathcal E(Y;q_i)\setminus \pi_{Y;q_i}(q_i)$ is a Jordan region or the complement of a point. Moreover, the map
$$ g_i=\pi_{Y;q_{i}}\circ \pi_Y^{-1}\circ g\colon \Omega_i\to D_i, $$
defined alternatively by $g_i =\pi_{Y;q_{i}}\circ H$ on $\widetilde X$
and $g_i(\widetilde p_j)= \pi_{Y;q_i}(q_j)$ for $j\neq i$, is continuous, proper, and cell-like. 
\end{lemma}
\begin{proof}
For simplicity, we set $p=p_i$, $\widetilde p=\widetilde p_i$, $q=q_i$, and $\pi_i=\pi_{Y;q_i}$. Since $\widetilde p$ is a non-separating continuum, the domain  $\Omega= \widehat{\C}\setminus \widetilde p$ is simply connected. Next, since $\pi_i$ is injective on $q$, we see that $\pi_i(q)$ is a closed Jordan region or a point in $\mathcal E(Y;q)$. Thus, $D=\mathcal E(Y;q)\setminus \pi_i(q)$ is a Jordan region or the complement of a point.

Observe that $\pi_{Y;q_i}(\pi_Y^{-1}(g(x)))$ is a point of $D$ for each $x\in \Omega$.  We show that the map $g_i=\pi_{i}\circ \pi_Y^{-1}\circ g$ is continuous on $\Omega$. Let $x\in \Omega$ and let $x_n\in \Omega$ be a sequence converging to $x$. Then $g(x_n)$ converges to  $g(x)$. By Lemma \ref{lemma:distance_convergence} \ref{lemma:distance:i}, $\pi_Y^{-1}(g(x_n))$ is contained in arbitrarily small neighborhoods of $\pi_Y^{-1}(g(x))$ as $n\to\infty$. By the continuity of $\pi_i$, $\pi_i(\pi_Y^{-1}(g(x_n)))$ converges to the point $\pi_i(\pi_Y^{-1}(g(x)))$, as desired. 

If a sequence $x_n\in \Omega$ accumulates at $\partial \Omega\subset \widetilde p$, then $g(x_n)$ accumulates at the point $\pi_Y(q_i)$ by continuity. By Lemma \ref{lemma:distance_convergence} \ref{lemma:distance:i}, $\pi_Y^{-1}(g(x_n))$ accumulates at $q_i$, so $g_i(x_n)$ accumulates at $\pi_i(q_i)\supset \partial D$.  Thus, $g_i$ is proper. Finally,  $g_i$ is cell-like, directly from its definition and from the cell-likeness of $H$ from Lemma \ref{lemma:H_definition}.
\end{proof}

In the proofs of the remaining statements of Section \ref{section:extension_individual}, we adopt the same notational conventions as in the first paragraph of the proof of Lemma \ref{lemma:h_definition}. 

\begin{lemma}\label{lemma:h_continuous}
For each $i\in \N$ the map $g_i\colon \Omega_i\to D_i$ has an extension to a continuous map from $\br {\Omega_i}$ onto $\br{D_i}$. 
\end{lemma}
\begin{proof}
We will show that for each $a\in \partial \Omega=\partial \widetilde{p}$ there exists a sequence of closed curves $C_n$ surrounding $a$ and shrinking to $a$ so that $\diam( H( C_n\cap \widetilde X))\to 0$. Suppose that this is the case. The continuity of $\pi_i=\pi_{Y;q_i}$ then implies that $\diam(\pi_i( H(C_n\cap \widetilde X) ))\to 0$ so the relation $g_i=\pi_i\circ H$ from Lemma \ref{lemma:h_definition} yields
$$\diam( g_i( C_n\cap \widetilde X))\to 0.$$
The set $g_i( C_n\cap \widetilde X)$ is dense in $g_i (C_n \cap \Omega)$, so
$$\diam(g_i(C_n \cap \Omega))\to 0.$$
Lemma \ref{lemma:cluster} \ref{lemma:cluster_extension} now implies that $g_i$ extends to a continuous map from $\br \Omega$ onto $\br D$, as desired. 

Now we show the original claim. Suppose first that $a\in \partial \widetilde p\setminus  p$ or that $p$ is a point and $a=p$. Fix $r_0>0$ so that $(B(a,r_0)\setminus \{a\})\cap p=\emptyset$. Let $C_r=S(a,r)$, $r>0$. For $r\in (0,r_0)$ the circle $C_r$ does not intersect $p=p_i$. We now apply the transboundary upper gradient inequality of $h$, which holds along the circle $C_r$ for a.e.\ $r\in (r_0/2,r_0)$; this is a consequence of Lemma \ref{lemma:lipschitz}. By the relation $H=\pi_Y^{-1}\circ h\circ \pi_X$ on $\widetilde X$, we have
\begin{align*}
\diam(H(C_r\cap \widetilde X)) \leq \int_{C_r} \rho_h \, ds+ \sum_{\substack{j:p_j\cap C_r\neq \emptyset \\ j\neq i}} \diam(q_j).
\end{align*}
We apply Lemma \ref{lemma:fubini} with $x=y=a$, $\lambda_j=\diam(q_j)$ for $j\neq i$ and $\lambda_{i}=0$, so we have
\begin{align*}
\essinf_{r\in (r_0/2,r_0)}\diam(H(C_r\cap \widetilde X))  \lesssim_{\tau}\left( \int_{B(a,r_0)} \rho_h^2\, d\Sigma+\sum_{\substack{j:p_j\cap B(a,r_0)\neq \emptyset \\ j\neq i}}\diam(q_j)^2\right )^{1/2}.
\end{align*}
Note that as $r_0\to 0$ the right-hand side tends to $0$, since $a\notin \bigcup_{j\neq i}p_j$. Thus, we can find a sequence $r_n\to 0$ such that 
$\diam(H(C_{r_n}\cap \widetilde X))\to 0$
as claimed. 

Next, suppose that $p$ is a closed Jordan region and that $a\in \partial \widetilde p\cap p$, so $a\in \partial p$. Let $r_0>0$ such that $B(a,r_0)$ does not contain $p$. Consider a circle centered at $a$ of radius $r<r_0$. Since $\partial p$ is a Jordan curve, there exists a unique arc $\alpha$ of the circle whose endpoints lie in $\partial p$ that separates (in $\widehat \C\setminus \inter(p)$) the point $a$ from any point outside $B(a,r_0)$. We let $C_r$ be the Jordan curve formed by the circular arc $\alpha$, together with a simple arc inside the interior of $p$ that connects the endpoints of $\alpha$. Then $C_r$ is a Jordan curve that surrounds the point $a$. If the non-circular part of $C_r$ is chosen appropriately, then we may have that $C_r$ shrinks to the point $a$ as $r\to 0$.  Now, the transboundary upper gradient inequality, applied to the circular part of $C_r$, gives
\begin{align*}
\diam(H(C_r\cap \widetilde X)) \leq \int_{C_r} \rho_h \, ds+ \sum_{\substack{j:p_j\cap C_r\neq \emptyset \\j\neq i}} \diam(q_j)
\end{align*}
for a.e.\ $r\in(r_0/2,r_0)$. We apply Lemma \ref{lemma:fubini} as above with $\lambda_j=\diam(q_j)$ for $j\neq i$ and $\lambda_{i}=0$ and the conclusion follows.  
\end{proof}

\begin{remark}\label{remark:cofatness_unnecessary}
The cofatness of $X$ is needed only for the application of Lemma \ref{lemma:fubini}. In fact, it suffices to apply the lemma to the cofat Sierpi\'nski packing generated by continua $p_j$, $j\neq i$, lying in a neighborhood of $p_i$. In particular, the fatness of $p_i$ is not needed. 
\end{remark}

\begin{lemma}\label{lemma:h_constant}
For each $i\in \N$ and for each component $V$ of $\inter(\widetilde p_i\setminus p_i)$ the map $g_i$ is constant on the set $\partial V\setminus \partial p_i$. 
\end{lemma}

See the shaded regions in Figure \ref{figure:branches}  for a depiction of the components of $\inter(\widetilde p_i\setminus p_i)$. 

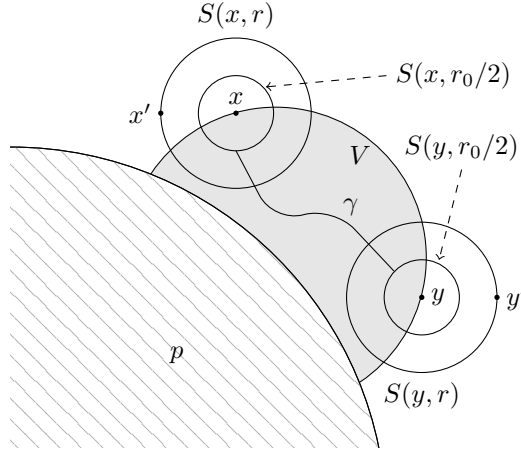
\begin{figure}
\input{component.tikz}
\caption{Illustration of the proof of Lemma \ref{lemma:h_constant}.}\label{figure:component}
\end{figure}

\begin{proof}
Note that if $x\in \partial \widetilde p$, then the circle $S(x,r)$ intersects $\widehat{\C}\setminus \widetilde p$, and thus $\widetilde X$, for all sufficiently small $r>0$. Indeed, since $\widetilde p$ is non-separating and $x\in \partial \widetilde p$, if $S(x,r)\subset \widetilde p$, then the complementary component of $S(x,r)$ that does not contain $x$ has to be contained in $\widetilde p$. However, if this occurs for a sequence of $r\to 0$, then we see that $\widetilde p=\widehat{\C}$, a contradiction. 

Let $V$ be a component of $\inter(\widetilde p\setminus p)$ and observe that $\partial V \setminus \partial p\subset  \partial \widetilde p$. We fix $x,y\in \partial V \setminus  \partial p$ and $\varepsilon>0$. Let $r_0>0$ such that $S(x,r)$ and $S(y,r)$ are disjoint from each other, intersect both $V$ and $\widetilde X$, and are disjoint from $p$ for all $r<r_0$. In Lemma \ref{lemma:fubini} we set $\rho=\rho_h$, $\lambda_j=\diam(q_j)$ for $j\neq i$, and $\lambda_i=0$.  If $r_0$ is sufficiently small, then there exists a set $S\subset (r_0/2,r_0)$ that has positive Lebesgue measure such that for all $r\in S$ we have
\begin{align}\label{lemma:h_constant_fubini}
\int_{S(x,r)\cup S(y,r)} \rho_h \, ds+ \sum_{\substack{j:p_j\cap S(x,r)\neq \emptyset\\ j\neq i}} \diam(q_j)+\sum_{\substack{j:p_j\cap S(y,r)\neq \emptyset\\ j\neq i}} \diam(q_j) <\varepsilon.
\end{align}
This uses the fact that $x,y\in \partial \widetilde p$, so $x,y\notin \bigcup_{j\neq i}p_j$ (see also the proof of Lemma \ref{lemma:h_continuous}). 

Consider the curve family $\Gamma_0$ that has $2$-modulus zero as in the transboundary upper gradient inequality of $h$. Also, consider a curve family $\Gamma_1$ with $\md_2\Gamma_1=0$ as in  Lemma \ref{lemma:avoid}. Then we may find (e.g.\ by invoking Lemma \ref{lemma:lipschitz}) a curve $\gamma\notin \Gamma_0\cup \Gamma_1$ contained in the connected open set $V$ and joining the circles $S(x,r_0/2)$ and $S(y,r_0/2)$, so that for a.e.\ $r\in (r_0/2,r_0)$, every simple path $\gamma_r$ whose trace is contained in $|\gamma|\cup S(x,r)\cup S(y,r)$ lies outside $\Gamma_0$. In particular, we may find such a path $\gamma_r$ with $r\in S$ such that $\gamma_r$ connects a point $x'\in S(x,r)\cap \widetilde X$ to a point $y'\in S(y,r)\cap \widetilde X$; see Figure \ref{figure:component}.

By the transboundary upper gradient inequality, we have 
\begin{align*}
\sigma(H(x'),H(y')) \leq \int_{\gamma_r} \rho_h\, ds  + \sum_{\substack{j: p_j\cap |\gamma_r|\neq \emptyset}} \diam(q_j).
\end{align*}
Since $\rho_h=0$ in $V$ (see Remark \ref{remark:packing_conformal}), by \eqref{lemma:h_constant_fubini} we have $\sigma(H(x'),H(y'))<\varepsilon$. If we let $r_0\to 0$ and then $\varepsilon\to 0$, we see that there exist sequences $x_n,y_n\in \widetilde X$ converging to $x,y$, respectively, such that $\sigma(H(x_n),H(y_n))\to 0$.  Therefore 
\begin{align*}
d(g_i(x),g_i(y))&= \lim_{n\to\infty} d( g_i(x_n),g_i(y_n) )\\&=\lim_{n\to\infty} d(\pi_i(H(x_n)) , \pi_i(H(x_n)))=0.
\end{align*}
This shows that $g_i(x)=g_i(y)$ as desired. 
\end{proof}

\begin{corollary}\label{corollary:h_extension}
For each $i\in \N$ the map $g_i\colon \Omega_i\to D_i$ has a unique extension to a continuous and cell-like map from $\widehat \C \setminus \inter(p_i)$ onto $\mathcal E(Y;q_i)\setminus \pi_{Y;q_i}(\inter(q_i))$. The extension is constant on each continuum $E\subset \widetilde p_i \setminus p_i$ and can be further extended to a continuous and cell-like map $g_i\colon \widehat{\C}\to \mathcal E(Y;q_i)$ such that $g_i^{-1}(\pi_{Y;q_i}(\inter(q_i)) )=\inter(p_i)$.
\end{corollary}

\begin{proof}
If $q$ is a point then the extension is trivial, so assume that $q$ is a closed Jordan region. By Lemma \ref{lemma:h_continuous}, $g_i$ extends to a continuous map from $\br\Omega$ onto $\br D$. By Lemma \ref{lemma:h_constant}, for each component $V$ of 
$$\inter(\widetilde p\setminus p)= \inter(\widetilde p)\setminus p= (\widehat \C  \setminus p) \setminus \br \Omega $$
the map $g_i|_{\partial V\setminus \partial p}$ is constant. We now apply  Theorem \ref{theorem:cell_like_extension} (with $Y= \widehat \C\setminus p$) to obtain the desired conclusions.
\end{proof}

\subsection{Patching continuous extensions}\label{section:patching}

\begin{corollary}\label{corollary:extension}
The map $H\colon \widetilde X\to Y$ extends to a continuous, surjective, and monotone map $H\colon \widehat \C\to  \widehat \C$ such that $\pi_Y\circ H=h\circ \pi_X$ and $H^{-1}(\inter(q_i))=\inter(p_i)$ for each $i\in \N$. The extension is constant on each continuum $E\subset \widetilde p_i\setminus p_i$, $i\in \N$.
\end{corollary}

\begin{proof}
We apply Corollary \ref{corollary:h_extension} so we obtain appropriate extensions of $g_i$ for each $i\in \N$. For each $i\in \N$, we extend $H$ to $\widetilde p_{i}$ by setting $H=\pi_{Y;q_{i}}^{-1} \circ g_i$; note that $\pi_{Y;q_i}^{-1}$ is injective on $g_i(\widetilde p_{i}) = \pi_{Y;q_i}(q_i)$. We observe that the relation $\pi_{Y;q_{i}}\circ H=  g_i$ (which is already valid on $\widetilde X$ by the definition of $g_i$) holds on all of $\widehat{\C}$. 

This extension is surjective onto $\widehat{\C}$. Moreover, the preimage of each point is a continuum; this follows from the facts that the maps $H|_{\widetilde X}\colon \widetilde X\to Y$ and $g_i|_{\widetilde p_i}\colon \widetilde p_i\to  {q_i}$ are cell-like. If $E\subset \widetilde p_i\setminus p_i$ is a continuum, then $g_i|_{E}$ is constant. Thus, $H$ is constant on $E$, as claimed. The relation $\pi_Y\circ H=h\circ \pi_X$, which holds on $\widetilde X$ extends automatically to $\widehat{\C}$, since $H( \widetilde p_i) = q_i$ for each $i\in \N$.

We finally argue that $H$ is continuous. Suppose first that $x\in \widetilde X$. If $x_n\in \widehat{\C}$ is a sequence converging to $x$, then $\pi_Y(H(x_n))= h(\pi_X(x_n))$ converges to $h(\pi_X(x))$ by continuity. By Lemma \ref{lemma:distance_convergence}, the point $H(x_n)\in \pi_Y^{-1} (\pi_Y(H(x_n)))$ is contained in arbitrarily small neighborhoods of $\pi_Y^{-1}(h(\pi_X(x)))=H(x)$ as $n\to\infty$.  Next, suppose that $x\in \widetilde p_{i}$ for some $i\in \N$ and consider a sequence $x_n\in \widehat{\C}$ converging to $x$. The sequence $\pi_{Y;q_{i}}( H(x_n))=  g_i(x_n)$ converges by continuity to $g_i(x)$. Using again Lemma \ref{lemma:distance_convergence}, we see that $H(x_n)$ is contained in arbitrarily small neighborhoods of $\pi_{Y;q_i}^{-1} (g_i(x) )= H(x)$ as $n\to\infty$.  
\end{proof}

\begin{proof}[Proof of Theorem \ref{theorem:continuous_extension}]
Consider the extension of $H$ as in Corollary \ref{corollary:extension}. It remains to show that $H$ satisfies the transboundary upper gradient inequality, as stated in Theorem \ref{theorem:continuous_extension}. Suppose that $h$, $\rho_h$, and $\Gamma_0$ are as in Definition \ref{definition:packing_quasiconformal} and let $\rho_H=\rho_h$. By enlarging the curve family $\Gamma_0$, we may assume that it still has conformal modulus zero and all subcurves of every curve $\gamma\notin \Gamma_0$ also satisfy the transboundary upper gradient inequality (see Remark \ref{remark:modulus}). We fix a curve $\gamma\colon [a,b] \to \widehat{\C}$ outside $\Gamma_0$.

If $\gamma(a)$ and $\gamma(b)$ lie in $p_i$ for some $i\in \N$, then $H(\gamma(a)),H(\gamma(b))\in q_i$, so $\sigma( H(\gamma(a)), H(\gamma(b)) )\leq \diam(q_i)$. In this case there is nothing to show. 

Suppose that  $\gamma(a)$ and $\gamma(b)$ do not lie on the same  peripheral continuum. Then there exists an open subpath $\gamma_1=\gamma|_{(a_1,b_1)}\colon (a_1,b_1)\to \widehat{\C}$ of $\gamma$ that does not intersect the peripheral continua that possibly contain $\gamma(a)$ or $\gamma(b)$, and the points $\gamma(a_1),\gamma(b_1)$ lie on the boundaries of the  peripheral continua that possibly contain $\gamma(a),\gamma(b)$, respectively. It suffices to show the desired inequality for the open path $\gamma_1$, in view of the previous case where the endpoints lie in the same peripheral continuum. 

There exists a further subpath $\gamma_2= \gamma|_{(a_2,b_2)}$ of $\gamma_1$ such that $\gamma((a_1,a_2])$ (resp.\ $\gamma([b_2,b_1))$) is contained in some set $\widetilde p_i$, $i\in \N$, and $\gamma(a_2)$ (resp.\ $\gamma(b_2)$) can be approximated by points in $|\gamma_2|\cap H^{-1}(Y)$. Note that these conditions could be vacuously true and we could have $\gamma_2=\gamma_1$. By Corollary \ref{corollary:extension}, $H$ is constant on $\gamma((a_1,a_2])$ and on $\gamma([b_2,b_1))$. Therefore, it suffices to show the desired inequality for the open path $\gamma_2$, which has the property that its endpoints can be approximated by points of $|\gamma_2|\cap H^{-1}(Y)$. 

We may find parameters $a_3>a_2$ and $b_3<b_2$ so that $\gamma(a_3),\gamma(b_3)\in H^{-1}(Y)$ and $\gamma(a_3),\gamma(b_3)$ are arbitrarily close to $\gamma(a_2),\gamma(b_2)$, respectively.  Since $\pi_Y$ is injective on $Y$, we have $H=\pi_Y^{-1}\circ h\circ \pi_X$ on $H^{-1}(Y)$. The transboundary upper gradient inequality of $h$ now gives 
\begin{align*}
\sigma( H(\gamma(a_3)), H(\gamma(b_3)) )\leq \int_{\gamma_2}\rho_H\, ds + \sum_{i:p_i\cap |\gamma_2|\neq \emptyset} \diam(q_i).
\end{align*}
By letting $\gamma(a_3)\to \gamma(a_2)$ and $\gamma(b_3)\to \gamma(b_2)$, using the continuity of $H$ we obtain the same inequality for $\sigma( H(\gamma(a_2)), H(\gamma(b_2)) )$. 
\end{proof}

\begin{proof}[Proof of Proposition \ref{proposition:constant}]
We fix $i\in \N$ and we use the same notational conventions as in the proof of Lemma \ref{lemma:h_definition}. By Lemma \ref{lemma:h_definition}, there exists a continuous, proper, and cell-like map $g_i$ from $\Omega$ onto $D$ such that 
$$g_i=\pi_i\circ \pi_Y^{-1}\circ g=\pi_i\circ \pi_Y^{-1}\circ h\circ \pi_X .$$
The relation $\pi_Y\circ H=h\circ \pi_X$ implies that on the set $\widetilde X$ we have $H=\pi_Y^{-1}\circ h\circ \pi_X$ and that $H(\widetilde p_j)=q_j$ for each $j\in \N$. Thus, $g_i= \pi_i\circ H$ on all of $\Omega$. 

Note that the continuous, surjective, and monotone map $H\colon \widehat{\C}\to \widehat{\C}$ is also cell-like (see the comments in Section \ref{section:topological}). This fact and Lemma \ref{lemma:cell_like} \ref{lemma:cell_like:set} imply that the map $\pi_i\circ H$ is continuous and cell-like from $\widehat{\C}$ onto $\mathcal E(Y;q)$.  This shows that $g_i$ extends to a continuous and cell-like map  from $\widehat{\C}$ onto $\mathcal E(Y;q)$.  Since $H^{-1}(\inter(q))=\inter(p)$, we have $g_i^{-1}( \pi_i(\inter(q)))=\inter(p)$. Hence $g_i$ is continuous and cell-like from $\widehat \C\setminus \inter(p)$ onto $\mathcal E(Y;q)\setminus \pi_i(\inter(q))$. By Corollary \ref{corollary:h_extension}, this extension is unique and has the property that it  is  constant on each continuum $E\subset \widetilde p \setminus p$. The desired property of $H$ follows. 
\end{proof}

\section{Homeomorphic extension}

\begin{theorem}\label{theorem:homeomorphic:packing}
Let $X,Y$ be Sierpi\'nski packings and $h\colon \mathcal E(X)\to \mathcal E(Y)$ be a packing-$K$-quasiconformal map for some $K\geq 1$ as in the statement of Theorem \ref{theorem:continuous_extension}. If $Y$ is cofat, then the map $H$ of Theorem \ref{theorem:continuous_extension} is a homeomorphism from $\br X$ onto $\br Y$. In particular, it can be extended to a homeomorphism of $\widehat{\C}$. 
\end{theorem}

\begin{proof}
Assume that $H$ maps $\br X$ homeomorphically onto $\br Y$.  From Theorem \ref{theorem:continuous_extension} we have $H^{-1}(\inter(q_i))=\inter(p_i)$, which implies that $H(\partial p_i)=\partial q_i$.  For each $i\in \N$ we replace $H|_{p_i}$ with a homeomorphism onto $q_i$ to obtain a global homeomorphism of $\widehat{\C}$; it is important here that $\diam(p_i)\to 0$ and $\diam(q_i)\to 0$ as $i\to\infty$.

We now show that for each point $w_0\in \br Y$ the set $H^{-1}(w_0)\subset \br X$ is a singleton. Suppose that there exists $w_0\in \br Y$ such that $E=H^{-1}(w_0)$ is a non-degenerate continuum. Based on the following claim, which we prove afterwards, we complete the proof of Theorem \ref{theorem:homeomorphic:packing}. See Figure \ref{figure:homeo} for an illustration. 

\begin{claim}\label{claim:homeomorphic}
There exists $\delta>0$, $r_0>0$, and a point $z_0\in E$ such that for all $t\in [\delta/2,\delta]$ and for all $r\in (0,r_0)$ the circle $S(z_0,t)$ has an open arc $\beta_t\colon (a,b)\to \widehat{\C}$ such that $\beta_t(a)\in E$, $\beta_t(b)\in \partial H^{-1}(B(w_0,r))$, and $\beta_t((a,b))\subset H^{-1}(B(w_0,r))\setminus E$. Moreover, if $w_0\in q_j$ for some $j\in \N$, then $\beta_t((a,b))\cap p_j=\emptyset$. 
\end{claim}

For fixed $r>0$, let $I(r)=\{ i\in \N: q_i\cap S(w_0,r)\neq \emptyset\}$.  If $i\in I(r)$, we define $\lambda(q_i)= \mathcal H^1(\{s\in (0,r): q_i\cap S(w_0,s)\neq \emptyset\})$ and otherwise we define $\lambda(q_i)=\diam(q_i)$. By the fatness of $q_i$, $i\in \N$, and Lemma \ref{lemma:radial}, we have
\begin{align}\label{theorem:homeomorphism:lambda}
\lambda(q_i)^2 \leq C \Sigma(q_i\cap B(w_0,r)) 
\end{align}
whenever $q_i\cap B(w_0,r)\neq \emptyset$. 

\begin{figure}
\input{homeomorphism.tikz}
\caption{Left: the circle $S(z_0,t)$ and the arc $\beta_t$ (red). Right: the path $H\circ \beta_t$ (red) and the subpath $H\circ \gamma_t$ (green) that connects $w_0$ to a point on $q_{i_0}\cap S(w_0,r_1)$.}\label{figure:homeo}
\end{figure}
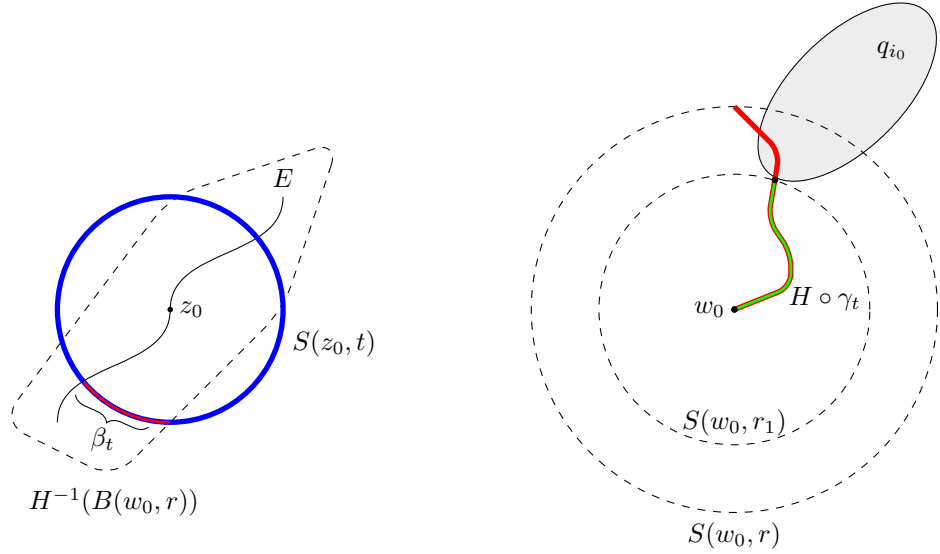

Suppose that $H\circ \beta_t$ does not intersect any $q_i$, $i\in I(r)$. That is, if $q_i\cap |H\circ \beta_t|\neq \emptyset$, then $q_i\subset B(w_0,r)$. By Lemma \ref{lemma:lipschitz}, the transboundary upper gradient inequality holds for a.e.\ $t\in [\delta/2,\delta]$, so we have
\begin{align*}
r\leq \int_{\beta_t}\rho \, ds + \sum_{i:p_i\cap |\beta_t|\neq \emptyset} \diam(q_i)= \int_{\beta_t}\rho \, ds+\sum_{i:p_i\cap |\beta_t|\neq \emptyset} \lambda(q_i) . 
\end{align*} 
If $H\circ \beta_t$ intersects some $q_i$, $i\in I(r)$, we consider an open subpath $\gamma_t$ of $\beta_t$ such that $H\circ \gamma_t$ does not intersect any $q_i$, $i\in I(r)$, and connects $w_0$ to a point of  $q_{i_0}$ for some $i_0\in I(r)$ that lies on a circle $S(w_0,r_1)$; see Figure \ref{figure:homeo}. By the transboundary upper gradient inequality, for a.e.\ $t\in [\delta/2,\delta]$ we have
\begin{align*}
r_1\leq \int_{\gamma_t}\rho \, ds + \sum_{i:p_i\cap |\gamma_t|\neq \emptyset} \diam(q_i)=\int_{\gamma_t}\rho \, ds+ \sum_{i:p_i\cap |\gamma_t|\neq \emptyset} \lambda(q_i) . 
\end{align*} 
Since the continuum $q_{i_0}$ intersects both circles $S(w_0,r_1)$ and $S(w_0,r)$, we have
$$r-r_1\leq  \mathcal H^1(\{s\in (0,r): q_{i_0}\cap  S(w_0,s)\neq \emptyset\}) =\lambda(q_i).$$
Altogether, we have 
\begin{align*}
r\leq  \int_{\beta_t}\rho \, ds+\sum_{i:p_i\cap |\beta_t|\neq \emptyset} \lambda(q_i) . 
\end{align*}

Since $\beta_t$ is contained in $S(z_0,t)\cap (H^{-1}(B(w_0,r))\setminus E)$, for a.e.\ $t\in [\delta/2,\delta]$ and for all $r<r_0$ we have
\begin{align}\label{theorem:homeomorphic:upper_gradient}
r\leq \int_{S(z_0,t)\cap (H^{-1}(B(w_0,r))\setminus E)}\rho \, d\mathcal H^1 + \sum_{\substack{i:p_i\cap S(z_0,t)\neq \emptyset\\ q_i\cap B(w_0,r)\neq \emptyset}} \lambda(q_i) . 
\end{align}
We integrate over $t\in [\delta/2,\delta]$ and then apply the co-area (Proposition \ref{proposition:coarea}) and  Cauchy--Schwarz inequalities to obtain
\begin{align*}
(\delta/2) r&\leq C\int_{H^{-1}(B(w_0,r))\setminus E} \rho \,d\Sigma + \sum_{i:q_i\cap B(w_0,r)} \lambda(q_i)\diam(p_i)\\
 &\leq C \left( \int_{H^{-1}(B(w_0,r))\setminus E} \rho^2\, d\Sigma + \sum_{i:q_i\cap B(w_0,r)\neq \emptyset} \lambda(q_i)^2\right)^{1/2}\\
 &\quad\quad \cdot \left( \Sigma(H^{-1}(B(w_0,r))\setminus E) + \sum_{i:q_i\cap B(w_0,r)\neq \emptyset} \diam(p_i)^2\right)^{1/2}\\
 &\eqqcolon C\cdot A(r)\cdot B(r).
\end{align*}

Suppose that $w_0\notin \bigcup_{i\in \N}q_i$. As $r\to 0$, the term $B(r)$ converges to $0$, since $H^{-1}(B(w_0,r))$ converges to $E$ and $\{i: q_i\cap B(w_0,r)\neq \emptyset\}$ shrinks to the empty set. If $w_0\in p_j$ for some $j\in \N$, then by the last part of Claim \ref{claim:homeomorphic}, $\beta_t$ does not intersect $p_j$, so we may exclude the index $j$ from the sum in \eqref{theorem:homeomorphic:upper_gradient}. In this case, if we perform the same calculations as above based on the modified version of \eqref{theorem:homeomorphic:upper_gradient}, we will obtain a term $B(r)$ such that the summation term does not include the index $j$. Therefore,  $\{i: i\neq j, \,\, q_i\cap B(w_0,r)\neq \emptyset\}$ shrinks again to the empty set. This shows that $B(r)\to 0$ as $r\to 0$ also in this case. 

We now discuss the first term, $A(r)$. By quasiconformality (see the last inequality in Theorem \ref{theorem:continuous_extension}) we have
\begin{align*}
\int_{H^{-1}(B(w_0,r))\setminus E} \rho^2\, d\Sigma \lesssim  \Sigma(B(w_0,r))\lesssim r^2.
\end{align*}
Moreover,  \eqref{theorem:homeomorphism:lambda} implies that
\begin{align*}
\sum_{i:q_i\cap B(w_0,r)\neq \emptyset} \lambda(q_i)^2\leq  C\sum_{i:q_i\cap B(w_0,r)\neq \emptyset} \Sigma(q_i\cap B(w_0,r)) \lesssim r^2.
\end{align*}
Therefore,  $A(r)\lesssim r$, so
\begin{align*}
\delta r \lesssim r B(r). 
\end{align*}
This is true for all $r<r_0$ so we obtain a contradiction as $r\to 0$. 
\end{proof}

\begin{proof}[Proof of Claim \ref{claim:homeomorphic}]
We consider two cases:
\begin{enumerate}[(a)]
	\item\label{case:a} $E$ is not contained in $p_i$ for any $i\in \N$.
	\item\label{case:b} $E\subset \partial p_{j}$ for some $j\in \N$.
\end{enumerate}
We will treat now the first case. Note that $E=H^{-1}(w_0)$ can intersect at most one set $p_i$, $i\in \N$, which occurs only if $w_0\in \partial q_i$. If $\partial E\subset p_i$, then $\partial E$ is an arc of $\partial p_i$, so $\widehat \C \setminus \partial E$ has one or two components. Each of them lies in the interior or exterior of $E$. Hence, either $E=\partial E\subset \partial p_i$, which is a contradiction to Case \ref{case:a}, or $E=\widehat{\C}\setminus \inter(p_i)$, which is again a contradiction as $E\cap p_j=\emptyset$ for $j\neq i$. Therefore, $\partial E$ is not contained in $p_i$ for any $i\in \N$, so $\partial E\cap X\neq \emptyset$.

Consider a point $z_0\in \partial E\cap X$ and note that the circle $S(z_0,t)$ intersects $E$ for all sufficiently small $t>0$. Recall that  $H$ is monotone, so it is also cell-like (see Section \ref{section:topological}). Therefore, $E=H^{-1}(w_0)$ is a non-separating continuum. This implies that the circle  $S(z_0,t)$ intersects $\widehat{\C}\setminus E$ for all sufficiently small $t>0$ (see the argument in the proof of Lemma \ref{lemma:h_constant}).  We fix $\delta>0$ such that for all $t\in [\delta/2,\delta]$ the circle $S(z_0,t)$ intersects both $E$ and $\widehat{\C}\setminus E$.  Moreover, if $w_0\in q_j$, the fact that $z_0\in X$ allows us to choose an even smaller $\delta>0$ so that $S(z_0,t)$ does not intersect $p_j$ for all $t\in [\delta/2,\delta]$. 

For $r>0$ consider the ball $B(w_0,r)$ and the preimage $H^{-1}(B(w_0,r))$, which contains the set $E$ in its interior.  If $r_0>0$ is sufficently small, then for $r<r_0$ the set $H^{-1}(B(w_0,r))$ does not contain any of the circles $S(z_0,t)$, $t\in [\delta/2,\delta]$; indeed, otherwise, by a limiting argument we see that there would exist $t\in [\delta/2,\delta]$ such that $S(z_0,t)\subset E$, a contradiction. We fix $t\in [\delta/2,\delta]$ and $r<r_0$. Then, the circle  $S(z_0,t)$ intersects both $E$ and $\partial H^{-1}(B(w_0,r))$.  We consider an open arc $\beta_t\colon (a,b)\to \widehat{\C}$ of $S(z_0,t)$ such that $\beta_t(a)\in E$, $\beta_t(b)\in \partial H^{-1}(B(w_0,r))$, and $\beta_t((a,b))\subset H^{-1}(B(w_0,r))\setminus E$. See Figure \ref{figure:homeo}. This completes the proof of Claim \ref{claim:homeomorphic} in Case \ref{case:a}.

Suppose that $E\subset \partial p_j$ as in Case \ref{case:b}. Let $z_0\in \partial p_j$ and fix $\delta_0>0$ such that $E$ is not contained in $B(z_0,\delta_0)$. Consider a Jordan arc $\eta \colon (0,1) \to \widehat \C$ such that  $\eta(0)=z_0$,  $\eta(1)\in E\setminus B(z_0,\delta_0)$, and $\eta((0,1))\cap p_j=\emptyset$. Let $\delta<\delta_0$. For each $t\in [\delta/2,\delta]$ there exists a unique open arc $\alpha_t\subset S(z_0,t)$ with endpoints on $\partial p_j$ but otherwise disjoint from $p_j$ such that $\eta(s_0)\in \alpha_t$ for some $s_0\in (0,1)$ and $\eta((0,s_0)) \cap  S(z_0,t)=\emptyset$; in other words, $\alpha_t$ is the ``first" arc of $S(z_0,t)\setminus p_j$ that $\eta$ meets. 

We claim that if $\delta$ is sufficiently small, then for $t\in [\delta/2,\delta]$ at least one endpoint of $\alpha_t$ lies in $E$. Consider the arc $A_t$ of $\partial p_j$ containing $z_0$ and having the same endpoints as $\alpha_t$. By the local connectivity of $\partial p_j$, as $t\to 0$  the diameter of $A_t$ tends to $0$. Hence, $E$ cannot be contained in $A_t$ for small $t>0$. The connectedness of $E$ implies that one of the endpoints of $A_t$, and thus of $\alpha_t$, is contained in $E$ for all sufficiently small $t>0$, as desired. 

Next, we show that if $r_0>0$ is sufficiently small, then $\alpha_t$ is not contained in $H^{-1}(B(w_0,r))$ for all $r<r_0$ and for all $t\in [\delta/2,\delta]$. To see this, note that there exists an arc $\eta([s_1,s_2])$, disjoint from $p_j$, intersected by $\alpha_t$ for each $t\in [\delta/2,\delta]$. Since $H^{-1}(B(w_0,r))$ converges to $H^{-1}(w_0)=E$ as $r\to 0$, there exists $r_0>0$ such that $H^{-1}(B(w_0,r_0))$ is disjoint from $\eta([s_1,s_2])$. Therefore, $\alpha_t$ intersects the complement of $H^{-1}(B(w_0,r))$, for all $r<r_0$, as claimed.

Finally, we consider an open subarc $\beta_t\colon (a,b)\to \widehat{\C}$ of $\alpha_t$ such that $\beta_t(a)\in E$, $\beta_t(b)\in \partial H^{-1}(B(w_0,r))$, and $\beta_t((a,b))\subset H^{-1}(B(w_0,r))\setminus E$. This completes the proof in Case \ref{case:b}.
\end{proof}

\bibliography{biblio}
\end{document}

%% file: branches.tikz
\begin{tikzpicture}[line cap=round,line join=round,>=triangle 45,x=1.0cm,y=1.0cm, scale=.4]
\clip(-9.749387336930544,-3.1520101085991366) rectangle (7.380590123625433,8.705103476538477);

\draw [line width=.5pt] (-0.66,6.46) circle (0.6082762530298222cm);
\draw [line width=.5pt] (-3.,6.) circle (0.6977105417004965cm);
\draw [line width=.5pt] (-5.42,4.98) circle (0.68cm);
\draw [line width=.5pt] (-6.14,3.2) circle (0.6767569726275451cm);
\draw [line width=.5pt] (-6.64,1.08) circle (0.8099382692526632cm);
\draw [line width=.5pt] (-5.5,-0.42) circle (0.7881624198095211cm);
\draw [line width=.5pt] (-2.62,-1.34) circle (0.54cm);
\draw [line width=.5pt] (-0.52,-1.44) circle (0.590592922409336cm);
\draw [line width=.5pt] (0.68,-0.08) circle (0.5249761899362675cm);
\draw [line width=.5pt] (1.32,2.68) circle (0.39849717690342545cm);
\draw [line width=.5pt] (1.36,1.12) circle (0.32984845004941266cm);

\draw[line width=2pt] plot [smooth, tension=1] coordinates {(-2,5) (-1.9,5.5) (-1.5, 6) (-2.2, 6.5) (-2, 7) (-2.1, 7.5) };
\draw[line width=2pt] plot [smooth, tension=1] coordinates { (-2, 7) (-1.8, 7.3) (-1.6, 7.5) };

\draw[line width=2pt] plot [smooth, tension=.5] coordinates {(-3,0) (-4,-.4) (-4.1, -1.2) (-4,-1.5) (-4.3, -2) (-4.2, -2.2) (-4.6, -2.7) };

\draw[line width=2pt] plot [smooth, tension=.5] coordinates {(-4,-1.5) (-3.5, -2) (-3.3, -2.5)};

\draw [line width=2.pt, pattern={Lines[distance=.8mm,angle=135,line width=0.2mm]}] (-0.4499616728935725,4.487547908883034) circle (0.6360159948182594cm);

\draw [fill=white, line width=.5pt] (-2.2,2.3) circle (2.8014282071829006cm);

\draw [fill=white, line width=2.pt] (-4.08,-1.48) ellipse (0.34cm and .7cm);
\draw [ line width=2.pt, pattern={Lines[distance=.8mm,angle=135,line width=0.2mm]}] (-4.08,-1.48) ellipse (0.34cm and .7cm);

\node at (-2.2,2.3) {$p_i$};
\end{tikzpicture}

%% file: monotone.tikz
	\begin{tikzpicture}[scale=.8]
		\draw[fill=black!10] (0,0) circle (2cm);
		\node at (0,0) {$\Omega$};
		\node at (1.5, -1.9) {$\partial U_0$};		
		
		\draw[fill=white] (0,1.5) circle (0.5cm);
		\draw[pattern={Lines[distance=1mm,angle=135,line width=0.1mm], pattern color=red}] (0,1.5) circle (0.5cm);
		\node at(0,1.5) {$U$};
		
		\draw[fill=white] (0.5,1) circle (0.2cm);
		\draw[ pattern={Lines[distance=1mm,angle=135,line width=0.1mm]}] (0.5,1) circle (0.2cm);		
				
		\draw[->] (2.3,0) to [out=30, in=150] (3.7,0);
		\node at (3,0.5) {$f$};
		
		\draw[fill=black!10] (6,0) circle (2cm);
		\node at (6,0) {$D$};
		
		\draw[fill=black] (6,2) circle (1pt);
		\node[anchor=north] at (6,2) {$z$};
	\end{tikzpicture}

%% file: graph.tikz
\begin{tikzpicture}
	\fill[pattern={Lines[distance=2mm,angle=135,line width=0.1mm]}, pattern color=black!20] (-5.5,1) rectangle (5.5,1.5);
	\node at (0,0) {\includegraphics[scale=1]{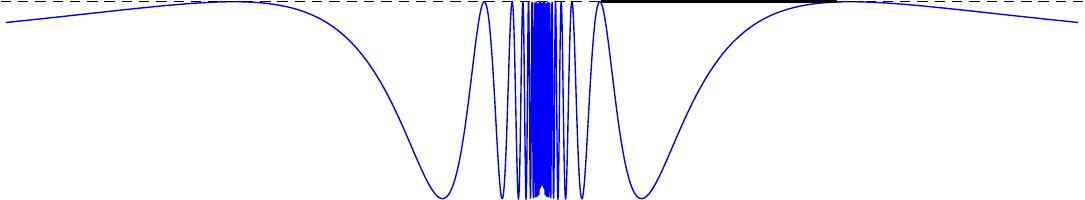}};
			
	\node at (-1,1.5) {$\widehat \C \setminus \br Y$};
	\node at (-1,-1.5) {$\Omega$};
	\node at (5,1.3) {$\partial Y$};
	\node at (5,0.6) {$\partial \Omega$};
	\node at (1.8, 1.2) {$\alpha_n$};
	\node at (1.2, 0) {$V_n$};

	\fill (1, 0.7) circle (1.3pt) node[anchor=west] {$x_n$};
	\fill (1, 1) circle (1.3pt) node[anchor=south] {$y_n$};
	\fill (0.6, 1) circle (1.3pt) node[anchor=south] {$z_n$};
	\fill (0, 1) circle (1.3pt) node[anchor=south] {$x$};
\end{tikzpicture} 

%% file: component.tikz
\begin{tikzpicture}
	\clip (0,1) rectangle (7,7);
	
	\draw[fill=black!10] (3.53,3.53) circle (2cm);	
	\node at (4.65,4.9) {$V$};	
		
	\draw[fill=white] (0,0) circle (5cm);	
	\draw[pattern={Lines[distance=2mm,angle=135,line width=0.1mm]}, pattern color=black!30] (0,0) circle (5cm) ;
	\node[anchor=south west] at (2,2) {$p$};

	\draw (3,5.45) circle (0.5cm);
	\draw (3,5.45) circle (1cm);
	\fill (3,5.45) circle (1pt) node[anchor=south] {$x$};
	\draw[->, dashed] (5, 5.95)  node[anchor=west] {$S(x,r_0/2)$} -- (3.4,5.8);
	\node[anchor=south] at (3,6.45) {$S(x,r)$};
	\fill (2,5.45) circle (1pt) node[anchor=east]{$x'$};

	\draw (5.47, 3) circle (0.5cm);
	\draw (5.47, 3) circle (1cm);
	\fill (5.47, 3) circle (1pt) node[anchor=west] {$y$};
	\draw[->, dashed] (6, 4.7)  node[anchor=south] {$S(y,r_0/2)$} -- (5.7,3.5);
	\node[anchor=north] at (5.47, 2) {$S(y,r)$};
	\fill (6.47,3) circle (1pt) node[anchor=west]{$y'$};

	\draw[rounded corners=4mm] (3,4.95)-- ( 3.5, 4) -- (4.3, 4.2)node[anchor=west] {$\gamma$} -- (5.1, 3.35);
	
\end{tikzpicture}

%% file: homeomorphism.tikz
\begin{tikzpicture}

	\begin{scope}[scale=1.5]
		\draw (1,1) node[anchor=south] {$E$} to[out=-90, in=90] (0,0) to[out=-90,in=90] (-1,-1);
		\draw[rounded corners=10,dashed] (1.5,1.5)-- (1,0) -- (-0.5,-1.5) node[anchor=north] {$H^{-1}(B(w_0,r))$} -- (-1.5,-1)-- (0,1)--cycle;

		\fill[black] (0,0) circle (0.7pt) node[anchor=west] {$z_0$};
		\draw[line width=2pt, blue] (0,0) circle (1cm);
		\node [anchor=west] at (1,-0.3) {$S(z_0,t)$};		
		\draw [red,line width=1pt,domain=220:270] plot ({cos(\x)}, {sin(\x)});		
		
		\begin{scope}[xshift=-0.4cm, yshift=-1.4cm]
		\draw[decorate,decoration={brace,amplitude=5pt},rotate=164]
(-0.1,-0.4) -- (0.6,-0.5);
		\end{scope}
		
		\node at (-0.6,-1.15) {$\beta_t$};
		
	\end{scope}
	
	\begin{scope}[scale=1.5, xshift=5cm]
		\begin{scope}[xshift=1cm, yshift=-0.9cm]
		\draw[rotate=45, fill=black!7] (2,2) ellipse (1cm and 0.5cm);
		\end{scope}
		\draw[dashed] (0,0) circle (1.8cm);
		\draw[dashed] (0,0) circle (1.2cm);
		\node at (1.4,2.3) {$q_{i_0}$};		
		
		\node[anchor=north] at (0,-1.8) {$S(w_0,r)$};
		\node[anchor=south] at (0,-1.2) {$S(w_0,r_1)$};
		
		\draw[rounded corners=5,red, line width=2pt] (0,0)-- (0.5,0.2) -- (0.5,0.5)-- (0.3, 0.8)-- (0.4, 1.4)-- (0, 1.8);
		
		\begin{scope}
		\clip (0,0) rectangle (0.6,1.15);
		\draw[rounded corners=5,green, line width=1pt] (0,0)-- (0.5,0.2) -- (0.5,0.5)-- (0.3, 0.8) -- (0.4, 1.4)-- (0, 1.8);
		\end{scope}		
		
		\draw[fill] (0,0) circle (0.7pt) node[anchor=east] {$w_0$};
		\draw[fill] (0.36, 1.15) circle (0.7pt);
		\node[anchor=west] at (0.4,0.1) {$H\circ \gamma_t$};
	\end{scope}

\end{tikzpicture}